\documentclass[11pt]{article}

\usepackage[utf8]{inputenc}
\usepackage{leftidx}
\usepackage{amsmath}
\usepackage{amscd}
\usepackage{amssymb}
\usepackage{rotating}
\usepackage{amsfonts,bm}
\usepackage{amsthm}
\usepackage{verbatim}
\usepackage{stmaryrd}
\usepackage{mathrsfs}
\usepackage{color}
\usepackage[top=1in, bottom=1in, left=1.25in, right=1.25in]{geometry}

\usepackage{titlesec}
       \titleformat{\chapter}[display]
             {\normalfont\Large\bfseries}{\thechapter}{11pt}{\Large}
       \titleformat{\section}
             {\normalfont\large\bfseries}{\thesection}{11pt}{\large}
       \titlespacing*{\chapter}{0pt}{0pt}{15pt} 
       \titlespacing*{\section}{0pt}{3.5ex plus 1ex minus .2ex}{2.3ex plus .2ex}

    \usepackage[titletoc]{appendix}

\input xy
\xyoption{all}

\allowdisplaybreaks[1]

\newcommand{\pqed}{\hfill\qedsymbol\\}

\newcommand{\Hom}{\mathrm{Hom}}
\newcommand{\Spec}{\mathrm{Spec}}
\newcommand{\Hhom}{\mathscr{Hom}om}

\DeclareMathOperator{\Ext}{Ext}
\DeclareMathOperator{\ext}{\mathscr{E}xt}
\DeclareMathOperator{\bext}{\mathbb{E}xt}

\newtheorem{theorem}{Theorem}[section]
\newtheorem{theorem/definition}{Theorem/Definition}[section]

\newtheorem{proposition}[theorem]{Proposition}
\newtheorem{lemma}[theorem]{Lemma}

\newtheorem{corollary}[theorem]{Corollary}

\newtheorem{conjecture}[theorem]{Conjecture}

\theoremstyle{remark}
\newtheorem{remark}[theorem]{Remark}

\theoremstyle{definition}
 
\newtheorem{definition}[theorem]{Definition}
\newtheorem{construction}[theorem]{Construction}
\newtheorem{notations}[theorem]{Notations}

\newtheorem{question}[theorem]{Question}

\begin{document}

\title
{\normalsize{\textbf{DEFORMATION OF EXCEPTIONAL COLLECTIONS}}}
\author{\normalsize Xiaowen Hu}
\date{}
\maketitle

\begin{abstract}
We show that in a smooth family of  complete varieties, the existence of full exceptional collection on a fiber preserves for the fibers in a neighborhood.
Then we show that the noncommutative deformations of a strong exceptional collection of vector bundles induce the same
map on the second Hochschild cohomology  as the canonical isomorphism induced by the derived equivalence to the corresponding endomorphism algebra. 
\end{abstract}

\section{Introduction}

For a smooth complex projective variety $X$, Dubrovin's conjecture (\cite{Dub98}, \cite{Bay04}) says that $\mathrm{D}^{\mathrm{b}}(X)$ has a full exceptional collection if and only if the Frobenius manifold $\mathcal{M}_X$ corresponding to the quantum cohomology of $X$ is generically semisimple, and moreover, the Gram matrix of a certain full exceptional collection of $\mathrm{D}^{\mathrm{b}}(X)$ is equal to the so-called Stokes data of  $\mathcal{M}_X$. 
The quantum cohomology is deformation invariant. Thus the above conjecture suggests that the existence of full exceptional collections is preserved in a smooth family. In this paper we show the existence in an open neighborhood.
\begin{theorem}\label{def-FEC0}\emph{($=$ corollary \ref{def-FEC2})}
Let  $S$ be a  locally noetherian scheme, $X$ a smooth proper scheme over $S$ with geometrically connected fibers, $s_0$ a point (not necessarily closed)  of $S$. 
 If $\mathrm{D}^{\mathrm{b}}(X_{s_0})$ has a  full  exceptional collection (resp., a strongly full exceptional collection), then there exists an open subset $V$ containing $s_0$ such that for any geometric point $s$ of $V$,  $\mathrm{D}^{\mathrm{b}}(X_{s})$ has a  full  exceptional collection (resp., a strongly full exceptional collection).
\end{theorem}
The deformability of an exceptional collection should have been well-known to the experts. To show the fullness, we use the notion of \emph{helix} of \cite{Bon89}.
To make the arguments work we need to generalize the definition of exceptional collections, their mutations, and helices, to a relative version.

Another motivation of theorem \ref{def-FEC0} is to understand the non-fullness of certain exceptional collections on surfaces of general type. It can be applied to simplify the argument of \cite{BGKS15}, see the remark after proposition \ref{prop-FEC4}. \\

Assume furthermore that the full exceptional collection $(E_i)_{1\leq i\leq n}$ is strong, and denote $A=\mathrm{End}_{\mathcal{O}_X}(\bigoplus_{i=1}^n E_i)$. In section 4, we study the deformations of finite dimensional algebras associated to acyclic quivers. Thus the deformations of $X_{s_0}$ induce a  map
\begin{equation}\label{eq0-1}
T_{s_0} S\rightarrow HH^2(A),
\end{equation}
where the second Hochschild cohomology $HH^2(A)$ parametrizes the deformations of $A$. On the other hand, when the characteristic of $\Bbbk$ is 0, there is a natural decomposition (see \cite{Swan96}, \cite{Yeku02}, \cite{Cal05})
\begin{equation}\label{eq0-2}
HH^i(X)= \bigoplus_{p=0}^{i}H^p(X,\wedge^{i-p}T_X),
\end{equation}
and there is a natural isomorphism (see e.g. \cite{BH13})
\begin{equation}\label{eq0-3}
HH^i(X)\cong HH^i(A).
\end{equation} 
Therefore it is natural to expect the map $H^1(X,T_X)\rightarrow HH^2(A)$ induced by (\ref{eq0-1}) coincides with the map induced by (\ref{eq0-2}) and (\ref{eq0-3}). It is natural to extend this statement to the \emph{noncommutative deformation} of Toda \cite{Toda05}, such that the image is the whole $HH^2(A)$. 
Our second main theorem confirms this under the restriction that the strong full exceptional collection $(E_i)_{1\leq i\leq n}$ \emph{consists of vector bundles}. 
Let us state it more precisely. First note that the existence of full exceptional collection implies that $H^2(X,\mathcal{O}_X)=0$. Then for $\beta\in H^1(X, T_X)$ and $\gamma\in H^0(X, \wedge^2 T_X)$, denote $u(0,\beta,\gamma)\in HH^2(A)$  the deformation of $A$ arising from the deformation of $X$. Denote $\Phi^i: \bigoplus_{p=0}^{i}H^p(X,\wedge^{i-p}T_X)\rightarrow HH^i(A)$ the composition of the isomorphisms (\ref{eq0-2}) and (\ref{eq0-3}). 
\begin{theorem}\label{theorem0-II}( $=$ theorem \ref{comparison0})
Let  $\Bbbk$ be a field of the characteristic zero, and $X$ a smooth proper scheme over $\Bbbk$, with a strong full exceptional collection of vector bundles $(E_i)_{1\leq i\leq n}$. Let $A=\mathrm{End}(\bigoplus_{i=1}^n E_i)$. Then in the  notations explained above we have
\begin{equation}\label{eq0-4}
\Phi^2(0,\beta,\gamma)=u(0,\beta,\gamma).
\end{equation}
\end{theorem}
The keeping of the redundant zero component in the above notations is made to be consistent with the notations in the maintext. In section 5 and 6, we work over characteristic zero, but one can easily check that the theorem \ref{theorem0-II} and the intermediate statements concerning only Hochschild cohomology of degree $\leq 2$ remain valid in characteristic $>3$.

I cannot find a direct conceptual proof of this theorem. Our proof is a bit involved;  part of the reason is that in the definition of noncommutative deformations the Hodge-type decomposition (\ref{eq0-2}) is used. In fact the reader will see that dealing with the $\lambda$-decomposition is the most technical part of the proof. 
 Our basic strategy  is to find  explicit Hochschild cochains of $A$ that represent both sides of (\ref{eq0-4}). In the process we also need to study the deformation of exceptional collections on the first order deformations of $A$ and the first order noncommutative deformations of $X$. Notice that in this paper by an \emph{explicit construction}, we mean explicit in the sense modulo the not-really-explicit construction of inverse images of  \v{C}ech coboundaries.

I do not the pursue  to remove the assumptions in theorem \ref{theorem0-II} that the full exceptional collection is strong and consists of vector bundles in this paper. The construction in the proof of \cite[prop. 6.1]{Toda05} may be helpful for the general case. Finally I remark that many existence results of infinitesimal deformations in this paper can also be deduced from the main theorems of \cite{Low05}.\\

We organize the paper as follows. In section 2 we introduce the notion of relative exceptional collections and helices. In section 3 we first show the existence of  deformations of exceptional collections, then prove  theorem \ref{def-FEC0}.
In section 4 we show the existence of full strong exceptional collection on a first order deformation of the finite dimensional algebra associated to an acyclic quiver modulo an admissible ideal of paths, and more relevant to the proof of  theorem \ref{theorem0-II}, we obtain the formula 
(\ref{eq4-0}). In section 5 we recall Toda's definition of noncommutative deformations associated to $(\alpha,\beta,\gamma)\in H^2(X,\mathcal{O}_X)\oplus H^1(X,T_X)\oplus H^0(X,\wedge^2 T_X)$, and constructed in an  explicit way the deformation of a exceptional collection of vector bundles on such a noncommutative deformation, and by comparing with (\ref{eq4-0}), we obtain an explicit expression of $u(\alpha,\beta,\gamma)$. In section 6, we recall three crucial properties of Hochschild cohomology: the HKR isomorphisms, the Morita equivalence and the $\lambda$-decomposition (also called Hodge-type decomposition). Then we construct a bar-type resolution of the diagonal $\Delta_X$ induced by a strong full exceptional collection,  and finally obtain an explicit representation of $\Phi^2(0,\beta,\gamma)$;  theorem \ref{theorem0-II} follows by a direct comparison of constructions \ref{constr-1} and \ref{constr-2}. In section 7 we propose some open problems.\\

\noindent\textbf{Acknowledgement}\quad
I am indebted to Huazheng Ke, many discussions with whom inspired the question, and to Zhan Li for  introducing \cite{AT08} to my scope and for discussions, and to Shizhuo Zhang, who shared to me a lot of knowledge and his interesting ideas in this field. I am also grateful to Yifei Chen, Chenyang Xu, Yuri Prokhorov, Feng Qu, Qizheng Yin, Qingyuan Jiang, Ying Xie, Lei Zhang, Jinxing Xu, Yining Zhang, Zhiyu Tian, Lei Song and Asher Auel for discussions on various problems related to this paper.  

In the previous version of this paper, I assumed the ampleness of $\omega_{X/S}$ or $\omega_{X/S}^{-1}$ outside of $s_0$, which is needed in the first proof of theorem \ref{helix4}. During my visit to CUHK invited by Yu Qiu, Ying Xie told me how to remove this assumption.  I am very grateful to Xie for  permission to present his argument in this version.

This work is supported by  NSFC 34000-31610265, 34000-41030364 and 34000-41030338.\\

\noindent\textbf{Notations}\quad
In this paper, unless otherwise stated, $\delta$ will denotes the differentials of a \v{C}ech complex $\check{C}(\cdot)$,  $Z^i(\cdot)$, $B^i(\cdot)$ the corresponding $i$-th group of cocycles and coboundaries,  $\mathfrak{b}$ the differentials of a Hochschild cochain complex, and $\mathfrak{b}'$ the differentials of a bar  complex. The symbol $\llcorner$ will always denotes the contraction of sections of tangent sheaves and cotangent sheaves, or the contraction of sections of their wedge products. The symbols $\mathbf{R}$ and $\mathbf{L}$ indicate the derived functors. The symbol $\epsilon$ will always denotes a square zero element, e.g., $\Bbbk[\epsilon]=\Bbbk[\epsilon]/(\epsilon^2)$. A geometric point means the spectrum of a separably closed field.

\section{Relative exceptional collections and helices}
In this section we collect some definitions and standard facts on admissible subcategories and semiorthogonal decompositions of a triangulated category. In passing we define the relative exceptional collection and helices. All triangulated subcategories are assumed full.
\subsection{Relative exceptional collections}

\begin{definition}
Let $\mathscr{A}$ be a triangulated category. A triangulated subcategory $\mathscr{B}$ is called \emph{right admissible} (resp., \emph{left admissible}) if the inclusion $\mathscr{B}\hookrightarrow \mathscr{A}$ has a right adjoint (resp., has a left adjoint). If $\mathscr{B}$ is both right and left admissible, it is called an \emph{admissible} subcategory.
\end{definition}

We will need the following lemma on admissible subcategories.

\begin{lemma}\label{adm-1}
\emph{\cite[3.1]{Bon89}}
Let $\mathscr{A}$ be a triangulated category, $\mathscr{B}$ a triangulated subcategory, $\mathscr{B}^{\perp}$ (resp., $\leftidx{^\perp}{\mathscr{B}}$) the right orthogonal (resp., the left orthogonal) to $\mathscr{B}$. Then the following are equivalent:
\begin{enumerate}
	\item $\mathscr{B}$ is right admissible (resp., left admissible).
	\item For every $X\in \mathscr{A}$ there is a distinguished triangle $Y\rightarrow X\rightarrow Z$ with $Y\in \mathscr{B}$ and $Z\in \mathscr{B}^\perp$ (resp., $Y\in \leftidx{^\perp}{\mathscr{B}}$ and $Z\in \mathscr{B}$).
	\item $\mathscr{B}$ and $\mathscr{B}^\perp$ (resp., $\leftidx{^\perp}{\mathscr{B}}$ and $\mathscr{B}$) generate $\mathscr{A}$.
	\item The inclusion $\mathscr{B}^{\perp}\hookrightarrow \mathscr{A}$ has a left adjoint (resp., the inclusion $\leftidx{^\perp}{\mathscr{B}}\hookrightarrow \mathscr{A}$ has a right adjoint).
\end{enumerate}	
\end{lemma}
For a scheme $S$, denote by $\mathrm{D}^{\mathrm{p}}(S)$ the triangulated category of perfect complexes on $S$. For a scheme $X$ over $S$, denote by $\mathrm{D}^{\mathrm{p}}(X/S)$  the category of $S$-perfect complexes on $X$ \cite{Lieb06}. If $X$ is smooth over $S$, $\mathrm{D}^{\mathrm{p}}(X/S)$ is equivalent to $\mathrm{D}^{\mathrm{p}}(X)$.

\begin{definition}(\cite{Kuz11})
An \emph{$S$-linear subcategory} of $\mathrm{D}^{\mathrm{p}}(X/S)$ is a triangulated subcategory which is closed under the operations of the form $\mathbf{L}\pi^* M\otimes^{\mathbf{L}}$ where $M\in \mathrm{D}^{\mathrm{p}}(S)$, where $\pi:X\rightarrow S$ is the structure morphism.
\end{definition}

Recall that we say that a locally noetherian scheme $X$ satisfies the resolution property if for every quasi-coherent sheaf $\mathcal{F}$ there is an epimorphism $\mathcal{E}\rightarrow \mathcal{F}$ where $\mathcal{E}$ is a direct sum of locally free coherent sheaves. We refer the reader to \cite[example 5.9]{Gro17} for examples of schemes satisfying the resolution property.

The following is a variant of \cite[2.7]{Kuz11}, and we reproduce his argument.
\begin{lemma}\label{ortho-1}
Let $S$ be a locally noetherian scheme satisfying the resolution property, $X$ a locally noetherian scheme, and $\pi:X\rightarrow S$ a quasi-compact and quasi-separated morphism of schemes. Then a pair of $S$-linear subcategories $\mathcal{A},\mathcal{B}$ of $\mathrm{D}^{\mathrm{p}}(X/S)$ is semiorthogonal if and only if $\mathbf{R}\pi_* \mathbf{R}\Hhom(B,A)=0$ for any $A\in \mathcal{A}$, $B\in \mathcal{B}$.
\end{lemma}
\begin{proof} Let $A\in \mathcal{A}$, $B\in \mathcal{B}$. Since $X$ is locally noetherian, $\mathbf{R}\Hhom(B,A)$ lies in $\mathrm{D}_{\mathrm{qc}}(X)$, the derived category of complexes of $\mathcal{O}_X$-modules with quasi-coherent cohomologies (\cite[II.3.3]{Har66}). Then since $\pi$ is quasi-compact and quasi-separated, $\mathbf{R}\pi_{*}\mathbf{R}\Hhom(B,A)\in \mathrm{D}_{\mathrm{qc}}(S)$ (see e.g. \cite[3.9.2]{Lip09}).  If $\mathbf{R}\pi_{*}\mathbf{R}\Hhom(B,A)\in \mathrm{D}_{\mathrm{qc}}(S)$ is nonzero, since $S$ satisfies the resolution property, there is a nonzero homomorphism $P\rightarrow \mathbf{R}\pi_{*}\mathbf{R}\Hhom(B,A)\in \mathrm{D}_{\mathrm{qc}}(S)$. By \cite[theorem A]{Spa88} or \cite[2.6.1, 3.2.1]{Lip09}, $\mathbf{R}\Hom(P,\mathbf{R}\pi_{*}\mathbf{R}\Hhom(B,A))\cong\mathbf{R}\Hom(\mathbf{L}f^* P,\mathbf{R}\Hhom(B,A))\cong  \mathbf{R}\Hom(B\otimes^{\mathbf{L}}\mathbf{L}f^* P, A)$, contradicting that $\mathcal{B}$ is left orthogonal to $\mathcal{A}$. The converse is obvious. 
\end{proof}

For a scheme $X$ and a perfect complex $M\in \mathrm{D}^{\mathrm{p}}(X)$, denote by $M^{\vee}$  the derived dual of $M$, i.e., $M^{\vee}=\mathbf{R}\Hhom (M,\mathcal{O}_Y)$. In the following of this section we assume $\pi:X\rightarrow S$ to be smooth and proper.

\begin{corollary}\label{adm-2}
Let $\mathcal{A}$ be an $S$-linear triangulated subcategory of $\mathrm{D}^{\mathrm{p}}(X/S)$. Then $\mathcal{A}^{\perp}$ and $\leftidx{^\perp}{\mathcal{A}}$ are also $S$-linear triangulated subcategories.
\end{corollary}
\begin{proof} For $M\in \mathrm{D}^{\mathrm{p}}(S)$ and $A, B\in \mathrm{D}^{\mathrm{p}}(X/S)$, we have
\[
\mathbf{R}\pi_* \mathbf{R}\Hhom(\mathbf{L}\pi^* M\otimes^{\mathbf{L}}B,A)=\mathbf{R}\Hhom(M,\mathbf{R}\pi_* \mathbf{R}\Hhom(B,A))
\]
and 
\[
\mathbf{R}\pi_* \mathbf{R}\Hhom(B,\mathbf{L}\pi^* M\otimes^{\mathbf{L}}A)=\mathbf{R}\Hhom(M^\vee,\mathbf{R}\pi_* \mathbf{R}\Hhom(B,A)).
\]
Thus the conclusion follows from lemma \ref{ortho-1}.
\end{proof}

\begin{definition}\label{definition-ec}
An ordered set of objects $(E_1,\cdots, E_n)$ of  $\mathrm{D}^{\mathrm{p}}(X/S)$ is called an \emph{exceptional collection of $X/S$ of length $n$} if  for $i>j$,
\begin{equation}\label{exceptional1}
\mathbf{R}\pi_*\mathbf{R}\Hhom(E_i,E_j)=0,
\end{equation}
and the canonical map
\begin{equation}\label{exceptional2}
\mathcal{O}_S\rightarrow \mathbf{R}\pi_*\mathbf{R}\Hhom(E_i,E_i)
\end{equation}
is an isomorphism for  $1\leq i\leq n$. It is called a \emph{strongly exceptional collection} if moreover the cohomology sheaves $\mathscr{H}^k\big(\mathbf{R}\pi_*\mathbf{R}\Hhom(E_i,E_j)\big)$ vanish for $k\neq 0$ and all pairs $i,j$. An exceptional collection of length 2 is called an \emph{exceptional pair}.
\end{definition}

The following fiberwise criterion for exceptionality is immediate from the definition.
\begin{lemma}\label{fiber1}
Let $X$ be a smooth proper scheme over $S$, $E=(E_1,\cdots,E_n)$ a sequence of objects of $\mathrm{D}^{\mathrm{p}}(X/S)$. Then $E$ is an exceptional collection of $\mathrm{D}^{\mathrm{p}}(X/S)$ if and only if $\mathbf{L}s^* E$ is an exceptional collection of $\mathrm{D}^{\mathrm{b}}(X_\xi)$ for every geometric point $s:\xi\rightarrow S$.
\end{lemma}
\pqed

\begin{definition}
Let $(E,F)$ be an exceptional pair. The \emph{left mutation} $L_E F$ and the \emph{right mutation} $R_F E$ are defined by the distinguished triangles
\begin{eqnarray*}
L_E F\rightarrow \mathbf{L}\pi^*\mathbf{R}\pi_*\mathbf{R}\Hhom (E,F)\otimes^{\mathbf{L}} E\rightarrow F\rightarrow (L_E F)[1],\\
E\rightarrow \big(\mathbf{L}\pi^*\mathbf{R}\pi_*\mathbf{R}\Hhom(E,F)\big)^\vee\otimes^{\mathbf{L}} F\rightarrow R_F E\rightarrow E[1],
\end{eqnarray*}
where both second arrows are induced by adjointness.
\end{definition}

For an exceptional collection $\sigma=(E_1,\cdots,E_n)$ we  define the $i$-th right and left mutations
\begin{eqnarray*}
R_i\sigma=(E_1,\cdots,E_{i-1}, E_{i+1}, R_{E_{i+1}}E_i,E_{i+2},\cdots,E_n),\\
L_i\sigma=(E_1,\cdots,E_{i-1}, L_{E_{i}}E_{i+1},E_i,E_{i+2},\cdots,E_n).
\end{eqnarray*}

\begin{definition}
For a set of object $G_1,\cdots, G_n$ of $\mathrm{D}^{\mathrm{p}}(X/S)$, denote by $\langle G_1,\cdots, G_n\rangle$ the smallest  $S$-linear triangulated  subcategory satisfying: (i) it containing $G_1,\cdots, G_n$, (ii) it is isomorphism closed, i.e., if $F\in \langle G_1,\cdots, G_n\rangle$ and $F'\cong F$ then $F'\in \langle G_1,\cdots, G_n\rangle$. 
\end{definition}

\begin{lemma}\label{lem-adm}
Let $\pi:X\rightarrow S$ be a smooth and proper morphism, and $(E_1,...,E_n)$ an exceptional collection of $\mathrm{D}^{\mathrm{p}}(X/S)$. Then $\langle E_1,...,E_n\rangle$ is an admissible subcategory of $\mathrm{D}^{\mathrm{p}}(X/S)$.
\end{lemma}
\begin{proof} For an object $A\in \mathrm{D}^{\mathrm{p}}(X/S)$, define $L^{i}A$ inductively by
\[
L^{i+1}A\rightarrow \mathbf{L}\pi^*\mathbf{R}\pi_*\mathbf{R}\Hhom (E_{n-i},L^i A)\otimes^{\mathbf{L}} E_{n-i}\rightarrow L^{i}A\rightarrow L^{i+1}A[1],
\]
and similarly we define $R^{i}A$. Then it is straightforward to see that $L^{n+1}[n+1]$ and $R^{n+1}[-n-1]$ are left and right adjoints of the inclusion $\langle E_1,...,E_n\rangle^{\perp}\hookrightarrow \mathrm{D}^{\mathrm{p}}(X/S)$ and we apply lemma \ref{adm-1}. 
\end{proof}


\begin{lemma}
\begin{enumerate}
	\item $\langle \sigma\rangle=\langle L_i\sigma\rangle=\langle R_i\sigma\rangle$.
	\item There are relations
\begin{eqnarray*}
R_i L_i\cong L_i R_i\cong 1,& R_i R_{i+1} R_i\cong R_{i+1}R_i R_{i+1}, & L_{i} L_{i+1} L_i\cong L_{i+1}L_i L_{i+1},\\
R_i R_j=R_j R_i, & L_i L_j=L_j L_i, & |i-j|\geq 2.
\end{eqnarray*}
\end{enumerate}
\end{lemma}
\begin{proof} The proof goes verbatim as the absolute case.
\end{proof}
\begin{lemma}
Let $(E_1,\cdots, E_n)$ be an exceptional collection of  $\mathrm{D}^{\mathrm{p}}(X/S)$. For an object $F$ of $\mathrm{D}^{\mathrm{p}}(X/S)$, we define inductively $L^i F$ and $R^i F$ by $L^0 F=R^0 F=F$, and the distinguished triangles
\begin{eqnarray*}
L^{k+1}F \rightarrow \mathbf{L}\pi^*\mathbf{R}\pi_*\mathbf{R}\Hhom (E_{n-k},L^k F)\otimes^{\mathbf{L}} E_{n-k}\rightarrow L^k F\rightarrow (L^{k+1} F)[1],\\
R^k F\rightarrow \big(\mathbf{L}\pi^*\mathbf{R}\pi_*\mathbf{R}\Hhom(R^k F,E_{k+1})\big)^\vee\otimes^{\mathbf{L}} E_{k+1}\rightarrow R^{k+1} F\rightarrow R^k F[1].
\end{eqnarray*}
Then $L^k F\in \langle E_{n-k+1},\cdots, E_{n}\rangle^{\perp}$  and 
$R^k F\in \leftidx{^\perp}{\langle E_1,\cdots,E_{k}\rangle}$ for $1\leq k\leq n$. Moreover, $\langle E_1,\cdots, E_n\rangle$ is an admissible subcategory of $\mathrm{D}^{\mathrm{p}}(X/S)$.
\end{lemma}
\begin{proof} We prove the first assertion by induction on $k$. The claim for $k=0$ is empty. Suppose $L^k F\in \langle E_{n-k+1},\cdots, E_{n}\rangle^{\perp}$, then for $j>0$, we have 
\[
\Hom(E_{n-k+j}, L^k F)=0.
\]
and for $j\geq 0$
\begin{eqnarray*}
&&\mathbf{R}\pi_* \mathbf{R}\Hhom(E_{n-k+j},\mathbf{L}\pi^*\mathbf{R}\pi_*\mathbf{R}\Hhom (E_{n-k},L^k F)\otimes^{\mathbf{L}} E_{n-k})\\
&=& \mathbf{R}\pi_*\big(\mathbf{L}\pi^*\mathbf{R}\pi_*\mathbf{R}\Hhom (E_{n-k},L^k F)\otimes^{\mathbf{L}} E_{n-k}\otimes^{\mathbf{L}} E_{n-k+j}^{\vee} \big)\\
&=& \mathbf{R}\pi_*\mathbf{R}\Hhom (E_{n-k},L^k F)\otimes^{\mathbf{L}} \mathbf{R}\pi_*(E_{n-k}\otimes^{\mathbf{L}} E_{n-k+j}^{\vee})\\
&=& \mathbf{R}\pi_*\mathbf{R}\Hhom (E_{n-k},L^k F)\otimes^{\mathbf{L}} \mathbf{R}\pi_*\mathbf{R}\Hhom(E_{n-k+j},E_{n-k})
\end{eqnarray*}
which vanishes if $j>0$ by the definition of exceptional collections, and is isomorphic to $\mathbf{R}\pi_*\mathbf{R}\Hhom (E_{n-k},L^k F)$ if $j=0$, and in this case the canonical map to  $\mathbf{R}\pi_*\mathbf{R}\Hhom (E_{n-k},L^k F)$ induced by the map
\[
\mathbf{L}\pi^*\mathbf{R}\pi_*\mathbf{R}\Hhom (E_{n-k},L^k F)\otimes^{\mathbf{L}} E_{n-k}\rightarrow L^k F
\]
is the identity. Thus $\mathbf{R}\pi_*\mathbf{R}\Hhom (E_{n-k},L^{k+1} F)=0$ and therefore $L^{k+1}F\in \langle E_{n-k},\cdots,E_n\rangle^\perp$. The proof of the conclusion for $R^k F$ is similar. The second assertion follows from the first one by the octahedral axiom.
\end{proof}

\subsection{Helices}
For an exceptional collection $\sigma=(E_1,\cdots,E_n)$ of $\mathrm{D}^{\mathrm{p}}(X/S)$, define inductively
\begin{equation}\label{generate1}
E_{n+i}=R^{n-1} E_i,\quad E_{n-i}=L^{n-1} E_i.
\end{equation}
\begin{definition}
Let $S$ be a locally noetherian scheme.
Suppose $X/S$ is smooth and proper of pure relative dimension $d$.
We call the sequence $\{E_i\}_{-\infty\leq i\leq \infty}$ a \emph{helix of period $n$}  if $E_i\cong E_{n+i}\otimes^{\mathbf{L}}\omega_{X/S}[d-n+1]$ for all $i$. We call an exceptional collection $\sigma=(E_1,\cdots,E_n)$ \emph{a thread of a helix} if the infinite sequence \eqref{generate1}  generated by $\sigma$ is a helix of period $n$.
\end{definition}

\begin{lemma}\label{rep1}
Let $(E_1,\cdots, E_n)$ be an exceptional collection of  $\mathrm{D}^{\mathrm{p}}(X/S)$. Then there is a canonical isomorphism
\[
\mathbf{R}\pi_* \mathbf{R}\Hhom(E_n, F)^\vee\xrightarrow{\sim} \mathbf{R}\pi_* \mathbf{R}\Hhom(F, L^{n-1} E_n[n-1])
\]
for $F\in \langle E_1,\cdots,E_n\rangle$.
\end{lemma}
\begin{proof} By the proof of \cite[4.2]{Bon89} there is a natural homomorphism
\[
\mathbf{R}\pi_* \mathbf{R}\Hhom(E_n, F)^\vee\rightarrow \mathbf{R}\pi_* \mathbf{R}\Hhom(F, L^{n-1} E_n[n-1])
\] 
and lemma \ref{fiber1} reduces the conclusion to the absolute case, then use the conclusion of \cite[4.2]{Bon89}. 
\end{proof}

\begin{lemma}\label{helix3}
Let $E_1,\cdots, E_n$ be an exceptional sequence of length $n$. It is a thread of a helix of period $n$ if and only if $E_i\cong R^{n-1} E_{i}\otimes \omega_{X/S}[d-n+1] $ for $i=1,\cdots,n$.
\end{lemma}
\begin{proof} We compute
\begin{eqnarray*}
(E_{-n},\cdots,E_{-1})&=&(L_1 L_2\cdots L_{n})^{n}(E_1,\cdots,E_n)\\
&=& (L_1 L_2\cdots L_{n})^{n}\big((E_1,\cdots,E_n)\otimes \omega_{X/S}^{-1}[-(d-n+1)]\big)\otimes \omega_{X/S}[d-n+1]\\{}
&=& \big((L_1 L_2\cdots L_{n})^{n}(R_n \cdots R_1)^n(E_1,\cdots,E_n)\big)\otimes \omega_{X/S}[d-n+1]\\
&=& (E_1,\cdots,E_n)\otimes \omega_{X/S}[d-n+1].
\end{eqnarray*}
\end{proof}

\subsection{Observations on ranks}
Since $\pi:X\rightarrow S$ is smooth, an $S$-perfect complex $E\in \mathrm{D}^{\mathrm{p}}(X/S)$ is in fact $\mathcal{O}_X$-perfect. Taking a local representative of $E$ as a bounded complex of locally free sheaves of finite ranks
\[
\cdots\rightarrow E^{i-1}\rightarrow E^i\rightarrow E^{i+1}\rightarrow\cdots
\]
and define the rank of $E$ to be 
\[
\mathrm{rank}(E)=\sum_{i=-\infty}^{\infty}(-1)^{i}\mathrm{rank}(E^i).
\]
Then $\mathrm{rank}(E)$ is a well-defined locally constant function on $X$. 
\begin{lemma}\label{rank1}
The following composition of natural maps
\begin{eqnarray*}\label{rank1.5}
\mathcal{O}_X\rightarrow E^\vee\otimes^{\mathbf{L}} E\rightarrow \mathcal{O}_X
\end{eqnarray*}
is the multiplication by $\mathrm{rank}(E)$.
\end{lemma}
\begin{proof} Write $E=E^\bullet$ as a bounded complex of locally free coherent sheaves. Let $f$ be a local section of $\mathcal{O}_X$ and denote by $m_f$ the multiplication by $f$. The first map sends $f$ to $(f_i=m_f)_{i\in \mathbb{Z}}$ where $f_i: E^i\rightarrow E^i$. The second map sends $(f_i)_{i\in \mathbb{Z}}$ which represents an element of $\mathscr{H}^0( E^\vee\otimes^{\mathbf{L}} E)$, to $\sum_i (-1)^i \mathrm{tr}(f_i)$. Composing the two maps we obtain (\ref{rank1.5}).
\end{proof}

\begin{lemma}\label{rank2}
Let $S$ be a field, $(E_1,\cdots, E_n)$ be a full exceptional collection of  $\mathrm{D}^{\mathrm{p}}(X/S)=\mathrm{D}^{\mathrm{b}}(X)$. Then $\mathrm{gcd}(\mathrm{rank}(E_1),\cdots,\mathrm{rank}(E_n))=1$.
\end{lemma}
\begin{proof} Since $(E_1,\cdots, E_n)$ is a full exceptional collection, the classes $[E_i]$ form a basis of $K_0(X)$, thus the conclusion follows.
\end{proof}

\section{Deformation of full exceptional collections}
\subsection{Existence of deformations}

We need to recall Lieblich's theorem of the representability of the moduli of objects in $\mathrm{D}^{\mathrm{p}}(X/S)$ \cite{Lieb06}. 

Let $\pi:X\rightarrow S$ be a  flat morphism  between schemes. An $S$-perfect complex $E$ is called \emph{gluable} if $\mathbf{R}\pi_* \mathbf{R}\Hhom(E,E)\in \mathrm{D}^{\geq 0}(S)$, and \emph{universally gluable} if this remains true for arbitrary base change $T\rightarrow S$. Let $\mathscr{D}_{\mathrm{pug}}^\mathrm{b}(X/S)$ be the following groupoid fibered over the  category of $S$-schemes,
\[
T\mapsto \{\mbox{universally gluable}\ T\mbox{-perfect complexes on}\ X_T\}.
\]
\begin{theorem}\label{Lieb1}\emph{\cite[4.2.1]{Lieb06}}
Let $\pi: X\rightarrow S$ be a proper flat morphism of finite presentation. Then $\mathscr{D}_{\mathrm{pug}}^\mathrm{b}(X/S)$ is an Artin stack locally of finite presentation, locally quasi-separated with separated diagonal, over $S$.
\end{theorem}

We also need the following theorem on the deformation and obstruction theory of the perfect complexes, see \cite[3.1.1]{Lieb06}, \cite{Low05} and \cite{HT10}.

\begin{theorem}
Let $I\rightarrow A\rightarrow A_0\rightarrow 0$ be a square zero extension of rings, and $X$ a scheme flat quasi-separated and of finite presentation and over $A$, $E_0\in \mathrm{D}^\mathrm{b}(X_{A_0})$. 
\begin{itemize}
   \item[(1)] There is an element $\omega(E_0)$ in $\Ext^2_{X_{A_0}}(E_0,E_0\otimes^{\mathbf{L}}_{A_0}I)$ which vanishes if and only if there exits  $E\in \mathrm{D}^\mathrm{b}(X_{A})$ such that $\mathbf{L}\iota^* E\cong E_0$, where $\iota:\Spec(A_0)\hookrightarrow \Spec(A)$ is the embedding.
   \item[(2)] If the deformation $E$ exists, the set of deformations of $E_0$ is a torsor under 
   \[
      \Ext^1_{X_{A_0}}(E_0,E_0\otimes^{\mathbf{L}}_{A_0}I).
   \]
 \end{itemize} 
\end{theorem}
From these theorems we can deduce the existence of deformations of exceptional collections in some open neighborhood.
\begin{theorem}\label{def-EC}
Let $S$ be  a  scheme, $X$ a smooth proper scheme over $S$, $s_0$ a point of $S$. Suppose $\mathrm{D}^\mathrm{b}(X_{s_0})$ has an    exceptional collection (resp., a strongly  exceptional collection) $\sigma=(E_1,...,E_n)$, then there exists an  étale neighborhood $U$ of $s_0$ in $S$, such that there exists a unique exceptional collection  (resp., a strong  exceptional collection) $\tau=(\mathcal{E}_1,...,\mathcal{E}_n)$ of $\mathrm{D}^\mathrm{p}(X_U/U)$ whose derived restriction to $\mathrm{D}^\mathrm{b}(X_{s_0})$ is   $\sigma$.
\end{theorem}
\begin{proof}  Let $A=\mathcal{O}_{S,s_0}$, $\mathfrak{m}$ the maximal ideal of $A$, and $A_k=A/\mathfrak{m}^{k+1}$. In particular $A_0$ is the residue field. Let $E_0$ be an exceptional object of $\mathrm{D}^\mathrm{b}(X_{s_0})$. We will show inductively that there exists uniquely an exceptional object $E_k$ of 
$\mathrm{D}^\mathrm{p}(X_{A_k}/A_k)$ such that $\mathbf{L}\iota_k^* E_k=E_0$, where $\iota_k:\Spec(A_0)\hookrightarrow \Spec(A_k)$ is the closed embedding. Since $E_0$ is exceptional, $\Ext^1_{X_{A_0}}(E_0, E_0)=\Ext^2_{X_{A_0}}(E_0, E_0)=0$, so there exists a unique deformation $E_1$ in $\mathrm{D}^\mathrm{p}(X_{A_1}/A_1)$. Suppose we have obtained $E_k$. Then for any $i$,
\begin{eqnarray*}
&&\Ext^i_{X_{A_k}}(E_k,E_k\otimes^{\mathbf{L}}_{A_k} \mathfrak{m}^{k+1}/\mathfrak{m}^{k+2})\\
&=&\Ext^i_{X_{A_k}}(E_k,E_k\otimes^{\mathbf{L}}_{A_k}A_0\otimes^{\mathbf{L}}_{A_0} \mathfrak{m}^{k+1}/\mathfrak{m}^{k+2})\\
&\cong&\Ext^i_{X_{A_k}}(E_k,E_k\otimes^{\mathbf{L}}_{A_k}A_0)^{\oplus\dim_{A_0} \mathfrak{m}^{k+1}/\mathfrak{m}^{k+2}}\\
&\cong&\Ext^i_{X_{A_k}}(E_k,\iota_* L \iota^* E_k)^{\oplus\dim_{A_0} \mathfrak{m}^{k+1}/\mathfrak{m}^{k+2}}\\
&\cong&\Ext^i_{X_{A_0}}(L\iota^*E_k, L \iota^* E_k)^{\oplus\dim_{A_0} \mathfrak{m}^{k+1}/\mathfrak{m}^{k+2}}\\
&\cong&\Ext^i_{X_{A_0}}(E_0, E_0)^{\oplus\dim_{A_0} \mathfrak{m}^{k+1}/\mathfrak{m}^{k+2}}.
\end{eqnarray*}
So there exists a unique deformation $E_{k+1}$ of $E_{k}$ in $\mathrm{D}^\mathrm{p}(X_{A_{k+1}}/A_{k+1})$, and by \cite[3.2.4]{Lieb06} $E_{k+1}$ remains perfect over $A_{k+1}$. 

By the definition of exceptional objects, $E_k$ is exceptional if and only if the cone of the map
\[
\mathcal{O}_{A_k}\rightarrow \mathbf{R}\pi_* \mathbf{R}\Hhom(E_k,E_k)
\]
is acyclic. It is perfect over $\Spec(A_k)$ and its restriction to $\Spec(A_0)$ is acyclic, therefore by the semicontinuity theorem (for perfect complexes, \cite[7.7.5]{EGAIII}), it is acyclic over $\Spec(A_k)$. The same argument deduces the existence and uniqueness of the deformation of the exceptional collection onto $X_{A_k}$. The existence of a formal deformation (i.e. a deformation of $E$ over $X_{\widehat{A}}$, where $\widehat{A}$ is the completion of $A$), and the algebraization (existence of a deformation over an open subset $U$ containing $s_0$), both follow from theorem \ref{Lieb1} on the representability of $\mathscr{D}_{\mathrm{pug}}^\mathrm{b}(X/S)$ as an Artin stack.

Finally, applying the semicontinuity theorem again, one sees that the resulting sequence of exceptional objects $(\mathcal{E}_1|_U,...,\mathcal{E}_n|_U)$ is an exceptional sequence, with $U$ shrunk if necessary. 
\end{proof}

\subsection{Fullness}

The following theorem is an enhancement of \cite[theorem 4.1]{Bon89}; notice that the word \emph{foliation} in the statement of the English version of loc. cit.  means \emph{bundle}, according to the russian version.

\begin{theorem}\label{helix4}
 Let $S$ be a connected locally noetherian scheme satisfying the resolution property, $\pi: X\rightarrow S$ a smooth proper morphism with connected fibers.
 Let $(E_1,\cdots,E_n)$ be an exceptional collection of $\mathrm{D}^{\mathrm{p}}(X/S)$. Consider the following two properties:
 \begin{itemize}
 	\item[1)] The exceptional collection $(E_1,\cdots,E_n)$ is full;
 	\item[2)] The collection $(E_1,\cdots,E_n)$ is a thread of a helix of period $n$.
 \end{itemize}
 \end{theorem}  

\begin{proof} By the connectedness of fibers, $\pi$ is equidimensional, and we assume the relative dimension is $d$.

 1) $\Rightarrow$ 2): By the Grothendieck duality,
\begin{eqnarray*}
&&\mathbf{R}\pi_* \mathbf{R}\Hhom(E_i, F)^\vee\cong \mathbf{R}\pi_* \mathbf{R}\Hhom(\mathbf{R}\Hhom(E_i, F), \omega_{X/S}[d])\\
&\cong & \mathbf{R}\pi_* \mathbf{R}\Hhom(F, E_i\otimes^{\mathbf{L}}\omega_{X/S}[d]),
\end{eqnarray*}
which by lemma \ref{rep1} induces a map
\begin{equation}\label{rep2}
L^{n-1} E_i[n-1]\rightarrow E_i\otimes^{\mathbf{L}}\omega_{X/S}[d],
\end{equation}
which is an isomorphism if $(E_1,\cdots, E_n)$ is a full exceptional collection of  $\mathrm{D}^{\mathrm{p}}(X/S)$. \\
2) $\Rightarrow$ 1): By lemma \ref{lem-adm}, we need only to show  $\langle E_1,\cdots, E_n\rangle^{\perp}=0$.
 We give two proofs for this. The first one is an improvement of that of \cite[theorem 4.1]{Bon89}, and has to assume the ampleness of $\omega_{X/S}^{-1}$ and $\mathrm{gcd}(\mathrm{rank}(E_1),\cdots,\mathrm{rank}(E_n))=1$. The second one is provided by Ying Xie.

\emph{Proof 1}:  We assume $\mathrm{gcd}(\mathrm{rank}(E_1),\cdots,\mathrm{rank}(E_n))=1$ ($E_i$ has  a constant rank for $S$ is connected) and $\omega_{X/S}^{-1}$ relatively ample; the case $\omega_{X/S}$ relatively ample is similar. Then there exists $r>0$ such that $\omega_{X/S}^{-r}$ is relatively very ample, which induces an embedding $\iota:X\hookrightarrow\mathbb{P}_S^N$. Suppose $F\in \langle E_1,\cdots, E_n\rangle^{\perp}$. Since $E_1,...,E_n$ generates a helix, by lemma \ref{ortho-1} one has
\begin{equation}\label{full1}
\mathbf{R}\pi_* \mathbf{R}\Hhom(E_i\otimes \omega_{X/S}^{k}, F)=0,
\end{equation}
for $1\leq i\leq n$ and any integer $k$.
On the other hand, writing $E_i=E_i^\bullet, F=F^\bullet$ as complexes of locally free coherent sheaves on $X$. 
\begin{eqnarray*}
\mathbf{R}\pi_* \mathbf{R}\Hhom(E_i^\bullet\otimes \omega_{X/S}^{rk}, F^\bullet)
=\mathbf{R}\pi_* \iota_*(E_i^{\bullet\vee}\otimes F^\bullet\otimes\omega_{X/S}^{-rk}),
\end{eqnarray*}
and for $k\gg 0$, 
\begin{eqnarray*}
\mathbf{R}\pi_* \mathbf{R}\Hhom(E_i^p\otimes \omega_{X/S}^{rk}, F^q)
=\pi_* \iota_*(E_i^{p\vee}\otimes F^q\otimes\omega_{X/S}^{-rk}).
\end{eqnarray*}
Taking into account (\ref{full1}), by Serre's theorem \cite[3.4.3]{EGAII}, $\iota_*(E_i^{\bullet\vee}\otimes F^\bullet)=0$ in $\mathrm{D}^p(\mathbb{P}_S^N/S)$, and thus $E_i^{\bullet\vee}\otimes F^\bullet=0$ in $\mathrm{D}^{\mathrm{p}}(X/S)$. Thus the composition
\[
F\rightarrow E_i\otimes^{\mathbf{L}} E_i^\vee\otimes^{\mathbf{L}} F \rightarrow F
\]
is zero.  However, by lemma \ref{rank1}, this composition is the multiplication by $\mathrm{rank}(E_i)$. By assumption $\mathrm{gcd}(\mathrm{rank}(E_1),\cdots,\mathrm{rank}(E_n))=1$, thus $F=0$.

\emph{Proof 2 (after Ying Xie)}: Let $F\in \langle E_1,\cdots, E_n\rangle^{\perp}$. By Grothendieck duality,
\[
\mathbf{R}\pi_* \mathbf{R}\Hhom(F, E_i\otimes \omega_{X/S}[d])\cong \big(\mathbf{R}\pi_* \mathbf{R}\Hhom(E_i,F)
\big)^{\vee}.
\]
Then by (\ref{full1}) and the definition of helix, $F\in \leftidx{^{\perp}}\langle E_1,\cdots, E_n\rangle$. Therefore $\langle E_1,\cdots, E_n\rangle$ and $\langle E_1,\cdots, E_n\rangle^{\perp}$ form a decomposition of $\mathrm{D}^{\mathrm{p}}(X/S)$, which is equivalent to the category of perfect complexes on $X$, because $\pi$ is smooth. But by \cite[corollary 4.6]{COS13}, $\mathrm{D}^{\mathrm{p}}(X/S)$ is indecomposable. So $\langle E_1,\cdots, E_n\rangle^{\perp}=0$, and we are done. 
\end{proof}

\begin{theorem}\label{def-FEC1}
Let  $S$ be a locally noetherian    scheme, $\pi:X\rightarrow S$ a smooth proper morphism with geometrically connected fibers, $s_0$ a point (not necessarily closed)  of $S$. 
If $\mathrm{D}^\mathrm{b}(X_{s_0})$ has a  full  exceptional collection (resp., a strong full exceptional collection), then there exists an étale neighborhood $W$ of $s_0$  such that $\mathrm{D}^\mathrm{p}(X_{W}/{W})$ has a  full  exceptional collection (resp., a strong full exceptional collection).
\end{theorem}
\begin{proof} Shrinking $S$ if necessary, we can assume that $S$ is affine noetherian and connected, thus satisfies the resolution property, and $X$ is of pure relative dimension $d$ over $S$. Let $\sigma_0$ be a full exceptional collection (resp., a strong full exceptional collection) of $\mathrm{D}^\mathrm{p}(X_{s_0})$. 
By theorem \ref{def-EC}, there exists an \'{e}tale neighborhood $U$ of $s_0$ and an exceptional collection (resp., a strong  exceptional collection) $\sigma=(E_1,\cdots,E_n)$ of  $\mathrm{D}^\mathrm{p}(X_U/U)$ extending  $\sigma_0$. 
 By lemma \ref{rep1}, there is a natural map
\begin{equation}\label{rep3}
L^{n-1} E_i[n-1]\rightarrow E_i\otimes^{\mathbf{L}}\omega_{X/U}[d],
\end{equation}
as in the proof of theorem \ref{helix4}. Since $\sigma_0$ is  full, (\ref{rep3}) is a quasi-isomorphism after restricting to $X_{s_0}$. By the semicontinuity theorem, there exists an open subset $W$ of $U$ containing $s_0$ such that the restriction of (\ref{rep3}) to $X_{W}$ is also a quasi-isomorphism, for $1\leq i\leq n$. By lemma \ref{helix3}, this means that $\sigma$ is a thread of a helix over $X_{W}$.  
Then by theorem \ref{helix4}, $\sigma|_{W}$ is a full exceptional collection of $\mathrm{D}^\mathrm{p}(X_{W}/{W})$.
\end{proof}

\begin{remark}
In the first version of this paper, we assume the relative ampleness of $\omega_{X/S}$ in theorem \ref{helix4}, and  the relative ampleness of $\omega_{X/S}$ or $\omega_{X/S}^{-1}$  over $S-\{s_0\}$ in  theorem \ref{def-FEC1}. The condition $\mathrm{gcd}(\mathrm{rank}((E_1)_{s_0}),\cdots,\mathrm{rank}((E_n)_{s_0}))=1$ is satisfied automatically because $((E_1)_{s_0},\cdots,(E_n)_{s_0})$ is full.
\end{remark}

\begin{corollary}\label{def-FEC2}
Let  $S$ be  a locally noetherian   scheme, $\pi:X\rightarrow S$ a smooth proper morphism with geometrically connected fibers, $s_0$ a point (not necessarily closed)  of $S$. 
If $\mathrm{D}^\mathrm{b}(X_{s_0})$ has a  full  exceptional collection (resp., a strong full exceptional collection), then there exists an open subset $V$ containing $s_0$ such that for any geometric point $s$ of $V$,  $\mathrm{D}^\mathrm{b}(X_{s})$ has a  full  exceptional collection (resp., a strong full exceptional collection).
\end{corollary}
\begin{proof} Use theorem \ref{helix4} and \ref{def-FEC1}, and that the helicity of a relative exceptional collection can be checked fiberwise.
\end{proof}

\begin{remark}
Passing to an étale neighborhood or the geometric point is necessary. In fact, consider a family of Brauer-Severi varieties, then each geometric fiber has a full exceptional collection. But the generic fiber may not have a full exceptional collection, unless after a base change to a separable extension of the base function field such that the generic fiber becomes split, by \cite[theorem 3.1]{Rae16}.
\end{remark}

During  the proof of theorem \ref{def-FEC1} we have in fact shown the following.
\begin{proposition}\label{prop-FEC4}
Let  $S$ be a locally noetherian    scheme, $\pi:X\rightarrow S$ a smooth proper morphism with connected fibers, $s_0$ a point  of $S$. Suppose that $\sigma=(E_1,...,E_n)$ is an exceptional collection of $\mathrm{D}^\mathrm{p}(X/S)$. If $\sigma$ is full at $s_0$, then there exists an open subset $U$ containing $s_0$ such that $\sigma|_U$ is full.
\end{proposition}

In \cite{BGKS15} a family of exceptional collection $\sigma=(\mathcal{L}_1,...,\mathcal{L}_{11})$ is constructed on certain determinantal Barlow surfaces, and they are shown not full on general fibers $S_t$, by proving that their endomorphism algebras are not deformed. For the specific Barlow surface $S_0$ (constructed by Barlow), the non-fullness is shown by using Kuznetsov's method via \emph{(pseudo)heights} of exceptional collections \cite{Kuz15}. Now as an application of proposition \ref{prop-FEC4}, the non-fullness of $\sigma$ restricted to $S_0$  follows directly from the non-fullness of $\sigma$ for general determinantal Barlow surfaces.

\section{First order deformations of full exceptional collections of modules of finite dimensional algebras associated to acyclic quivers with relations}
Fix a base field $\Bbbk$. We follow the terminology on finite dimensional algebras and quivers of \cite{ASS}. For example, the algebra associated to the following quiver
\[
\xy <1cm,0cm>:
 (0,0)*{\bullet} ;
 (1.5,0)*{\bullet} **\dir{-} ?>* \dir{>};
  (3,0)*=0{\bullet}  **\dir{-} ?>* \dir{>};
 (0,0.3)*+{p_1} ; (1.5,0.3)*+{p_2}; (3,0.3)*+{p_3};
 (0.75,0.3)*+{a}; (2.25,0.3)*+{b}
\endxy
\]
is the $\Bbbk$-algebra generated by $W=\{p_1,p_2,p_3,a,b\}$ with the relations $p_i^2=p_i$ for $1\leq i\leq 3$, $p_1 a=a=ap_2$, $p_2 b=b=bp_3$, and  all the other products of two elements of $W$, except $ab$, are zero. In this section we need to quote several results of \cite{ARS}, the reader should notice that our convention of the products is different from that of loc. cit, because this is more convenient for considering right modules. Recall that a quiver is called \emph{acyclic} if it has no oriented nontrivial cycles of arrows; in \cite{Bon89} an acyclic quiver was called \emph{ordered}.

 Throughout this section, we assume that $A$ is a finite dimensional algebra of the form $\Bbbk \Delta/\mathcal{I}$, where $\Delta$ is an acyclic quiver and $\mathcal{I}$ is an admissible ideal, i.e., $R_{\Delta}^m\subset \mathcal{I}\subset R_{\Delta}^2$ for some integer $m\geq 2$, where $R_{\Delta}$ is the ideal of nontrivial arrows of $\Delta$.  Let $\{ p_1,...,p_n\}$ be the set of vertices of $\Delta$. Then $p_i A$, $1\leq i\leq n$ form a complete set of indecomposable projective right $A$-modules. Denote by $\mathrm{D}^{\mathrm{b}}(A)$  the derived category of finite dimensional right $A$-modules. We arrange the order of $p_1,..,p_n$ such that for $1\leq i<j\leq n$, there is no path  in $\Delta$ that starts from $p_i$ and ends at $p_j$. Thus $p_i x p_j=0$ for $i<j$ and all $x\in A$.  Then by \cite[section 5]{Bon89}, $(p_1 A,...,p_n A)$ is a strong full exceptional collection of $\mathrm{D}^{\mathrm{b}}(A)$, and 
\[
A\cong \bigoplus_{1\leq i,j\leq n}\Hom_{A}(p_i A,p_j A),
\]
where $\Hom_A$ is taken in the category of right $A$-modules.

Denote $\Bbbk[\epsilon]=\Bbbk[\epsilon]/(\epsilon^2)$, and $S=\Spec(\Bbbk[\epsilon])$.  A \emph{deformation} of $A$ over $S$ is a flat $\Bbbk[\epsilon]$-algebra $A^\dagger$ with an isomorphism $A^\dagger\otimes_{\Bbbk[\epsilon]}\Bbbk\cong A$. Denote by $HH^n(A)=HH^n(A,A)$ the Hochschild cohomology of the $\Bbbk$-algebra $A$. Then the deformation of $A$ over $S$ is parametrized by $HH^2(A)$ due to \cite{Ger64}. For later use let us recall this fact. 
Explicitly, a Hochschild  $n$-cochain is a $\Bbbk$-linear map $f:A^{\otimes_{\Bbbk} n}\rightarrow A$, and  the coboundary is  given by
\begin{multline}\label{hoch-diff1}
\mathfrak{b}(f)(a_1,\cdots,a_{n+1})=a_1 f(a_2,\cdots,a_{n+1})\\
+\sum_{1\leq i\leq n}(-1)^i f(a_1,\cdots,a_{i}a_{i+1},\cdots,a_{n+1})+(-1)^{n+1}f(a_1,\cdots,a_n)a_{n+1}.
\end{multline}
Given a 2-cocycle $u$, the corresponding deformation of $A_u$ over $S$ is given by the multiplication
\begin{equation}\label{deform-alg0}
(a_0+\epsilon a_1)\cdot_u (b_0+\epsilon b_1)=a_0b_0+\epsilon(a_1 b_0+a_0 b_1+u(a_0,b_0)).
\end{equation}
for $a_i,b_i\in A$, $i=0,1$. If $u$ and $u'$ differ by a coboundary $\mathfrak{b}(v)$, then there is an induced isomorphism $A_u\cong A_{u'}$ over $S$. It is straightforward to see that all deformations of $A$ over $S$ arise in this way. 

For a flat finite dimensional algebra $A^\dagger$ over $S$, denote by $\mathrm{D}^{\mathrm{p}}(A^\dagger/S)$ the full $S$-linear triangulated subcategory of $\mathrm{D}^{\mathrm{b}}(A)$ generated by the bounded complexes of projective right $A^\dagger$-modules. The following is the main theorem of this section.

\begin{theorem}\label{deform-alg1}
Let $u\in HH^2(A)$.
\begin{enumerate}
  \item[(i)] For $1\leq i\leq n$, there exists a unique projective right $A_u$-module $P_i$ that deforms $p_i A$. 
  \item[(iii)] There exists $\lambda_i\in \Bbbk$ and $a_i,b_i,c_i\in A$ such that $P_i=(p_i+\epsilon(-\lambda_i p_i +a_i p_i +p_i b_i+c_i))A_u$.
  \item[(iii)] The sequence $(P_1,...,P_n)$ is a  strong full  exceptional collection of the $S$-linear triangulated category $\mathrm{D}^{\mathrm{p}}(A_u/S)$, i.e., 
  \[
  \begin{cases}
  \Bbbk[\epsilon]\xrightarrow{\sim} \Hom_{A_u}(P_i,P_i)\ \mbox{for $1\leq i\leq n$},\\
  \Hom_{A_u}(P_i,P_j)=0\ \mbox{for $i>j$},\\
  \Ext^k(P_i,P_j)=0\ \mbox{for $k>0$ and $1\leq i,j\leq n$}.
  \end{cases}
  \]

  \item[(iv)] We have
\begin{equation}\label{eq4-0}
A_u\cong\bigoplus_{1\leq i\leq j\leq n}\Hom_{A_u}(P_i,P_j).
\end{equation}
\end{enumerate}
\end{theorem}

One can find more descriptions on $a_i,b_i,c_i$ in the following lemmas, which we do not spell out in the above theorem. We remark that the choices of $a_i$ and $b_i$ are not unique.
We refer the reader to \cite{HS88} for a result  close to ours. The rest of this section is devoted to an elementary proof of theorem \ref{deform-alg1}.\\

In the following of this section we fix a Hochschild 2-cocycle $u: A\otimes_{\Bbbk} A\rightarrow A$. For $\forall a, b, c\in A$,
\begin{equation}\label{cocycle1}
a u(b,c)-u(ab,c)+u(a,bc)-u(a,b)c=0.
\end{equation}

\begin{lemma}\label{cocycle2}
\begin{enumerate}
  \item[(i)] For $1\leq k\leq n$, there exist a unique $\lambda_k\in \Bbbk$, and a unique $c_k\in A$ which modulo $\mathcal{I}$ is a linear combination of paths whose beginnings and ends are not $p_k$, such that 
  \begin{equation}\label{cocycle3}
  u(p_k,p_k)=\lambda_k p_k+c_k.
  \end{equation}
  \item[(ii)] For $1\leq i,j\leq n$ and $i\neq j$, there is a unique $d_{ij}\in A$ which modulo $\mathcal{I}$ is a linear combination of paths connecting $p_i$ towards $p_j$, such that
  \begin{equation}\label{cocycle4}
  u(p_i,p_j)= d_{ij}-p_i c_j-c_i p_j,
  \end{equation}
  where $c_i,c_j\in A$ are as in (i). In particular, if $i<j$, then $d_{ij}=0$.
  \item[(iii)] For any $x\in A$, we have $x u(1,1)=u(x,1)$ and $u(1,1)x=u(1,x)$.
  \item[(iv)] Let $\lambda_k$, $c_k$, $1\leq k\leq n$, and $d_{ij}$, $1\leq i\neq j\leq n$ be determined as in (i) and (ii).  Let $1_u$ be the identity element of $A_u$, then
  \begin{equation}\label{cocycle5}
  1_u=1-\epsilon\big(\sum_{1\leq k\leq n}(\lambda_k p_k-c_k)+\sum_{\stackrel{1\leq i,j\leq n}{i\neq j}}d_{ij}\big).
  \end{equation}
  \item[(v)]
   $A_u$ is a $\Bbbk[\epsilon]/(\epsilon^2)$-algebra via the map
\begin{equation}\label{scalar0}
\Bbbk[\epsilon]/(\epsilon^2)\rightarrow A_u,\ \mu+\epsilon\nu\mapsto \mu\cdot 1_u+\epsilon\nu=\mu+\epsilon(-\mu u(1,1)+\nu).
\end{equation}
\end{enumerate}
\end{lemma}
\begin{proof} (i) Taking $a=b=c=p_k$ in (\ref{cocycle1}), we obtain $p_k u(p_k,p_k)=u(p_k,p_k)p_k$. Thus (\ref{cocycle3}) holds for some $\lambda_k\in \Bbbk$, and $c_k$ satisfies $p_k c_k=c_k p_k=0$, i.e. $c_k$ is a linear combination of paths whose beginning and ends are not $p_k$.

(ii) Taking $a=b=p_i$ and $c=p_j$ in (\ref{cocycle1}), we obtain
\[
p_i u(p_i,p_j)-u(p_i,p_j)-u(p_i,p_i)p_j=0.
\]
Taking $a=p_i$ and $b=c=p_j$ in (\ref{cocycle1}), we obtain
\[
p_i u(p_j, p_j)+u(p_i,p_j)-u(p_i,p_j)p_j=0.
\]
Thus
\begin{eqnarray*}
 u(p_i,p_j)&=&p_i u(p_i,p_j)-u(p_i,p_i)p_j=p_i (u(p_i,p_j)p_j-p_i u(p_j,p_j))-u(p_i,p_i)p_j\\
&=& p_i u(p_i,p_j)p_j-p_i u(p_j,p_j)-u(p_i,p_i)p_j\\
&=& d_{ij}-p_i c_j-c_i p_j,
\end{eqnarray*}
where $p_i d_{ij}p_j=d_{ij}$, i.e. $d_{ij}$ is a linear combination of paths connecting $p_i$ towards $p_j$.

(iii) For the first identity, take $a=x$ and $b=c=1$ in (\ref{cocycle1}). For the second one, take $a=b=1$ and $c=x$ in (\ref{cocycle1}).

(iv) By (i) and (ii), 
\[
u(1,1)=\sum_{i,j}u(p_i,p_j)=\sum_i(\lambda_i p_i+c_i)+\sum_{i\neq j}(d_{ij}-c_i p_j-p_i c_j)=\sum_i(\lambda_i p_i-c_i)+\sum_{i\neq j}d_{ij}.
\]
Thus (\ref{cocycle5}) follows from (\ref{deform-alg0}). (v) is obvious.
\end{proof}

From now on in this section we omit the term ``modulo $\mathcal{I}$".

\begin{lemma}\label{cocycle6}
Let $\lambda_k, c_k, d_{ij}$ be the elements uniquely determined by lemma \ref{cocycle2}.
\begin{enumerate}
  \item[(i)] For $1\leq k\leq n$, the equation 
\begin{equation}\label{cocycle6-1}
(p_k+\epsilon x)\cdot_u (p_k+\epsilon x)=p_k+\epsilon x
\end{equation}
has  solutions, which of the form $x=-\lambda_k p_k+a_k p_k+p_k b_k+c_k$ such that 
\begin{equation}\label{cocycle6-1.5}
p_k a_k=b_k p_k=0,
\end{equation}
 i.e., $a_k$ is a linear combination of paths that do not start at $p_k$, and $b_k$ is a linear combination of paths that do not end at $p_k$.
  \item[(ii)] The system of idempotents $\{p_k+\epsilon  x_k\}_{1\leq k\leq n}$,  are orthogonal if and only if
\begin{equation}\label{cocycle6-2}
p_i a_j p_j+p_i b_i p_j+d_{ij}=0
\end{equation}
for $1\leq i\neq j\leq n$. Such a system of idempotents exist, and they satisfy
\begin{equation}\label{cocycle6-3}
\sum_i \big(p_i+\epsilon (-\lambda_i p_i +a_i p_i +p_i b_i+c_i)\big)=1_u.
\end{equation}
\end{enumerate}
 
\end{lemma}
\begin{proof} The equation (\ref{cocycle6-1}) is equivalent to
\[
x=p_k x+xp_k+u(p_k,p_k),
\]
i.e.,
\[
x=p_k x+xp_k+\lambda_k p_k +c_k.
\]
From this one easily deduces (i). For $i\neq j$,
\begin{eqnarray*}
&&\big(p_i+\epsilon (-\lambda_i p_i +a_i p_i +p_i b_i+c_i)\big)\big(p_j-\epsilon (-\lambda_j p_j+a_j p_j+p_j b_j+c_j)\big)\\
&=& \epsilon \big(p_i a_j p_j+p_i c_j+p_i b_i p_j+c_i p_j+u(p_i,p_j)\big)\\
&=&\epsilon (p_i a_j p_j+p_i b_i p_j+d_{ij}),
\end{eqnarray*}
where we have used (\ref{cocycle4}). The existence of the solutions for the system of equations (\ref{cocycle6-2}) follows by induction on $i$, using the acyclicity of the graph $\Delta$. Assuming (\ref{cocycle6-2}), one has, by (\ref{cocycle5}),
\begin{eqnarray*}
&&\sum_i (p_i+\epsilon (-\lambda_i p_i +a_i p_i +p_i b_i+c_i))\\
&=& 1+\sum_i\epsilon (-\lambda_i p_i+1\cdot a_i p_i +p_i b_i\cdot 1+c_i)\\
&=& 1+\sum_i\epsilon \big(-\lambda_i p_i+(\sum_j p_j) \cdot a_i p_i +p_i b_i\cdot (\sum_j p_j)+c_i\big)\\
&=& 1+\epsilon \big(-\sum_i\lambda_i p_i-\sum_{i\neq j}d_{ij}+\sum_i c_i\big)=1_u,
\end{eqnarray*}
where for the third equality we use (\ref{cocycle6-1.5}) and (\ref{cocycle6-2}), and for the final equality we use (\ref{cocycle5}).
\end{proof}

\begin{lemma}\label{cocycle7}
Let $p_k+\epsilon (-\lambda_k p_k +a_k p_k +p_k b_k+c_k)$ and $p_k+\epsilon (-\lambda_k p_k +a'_k p_k +p_k b'_k+c_k)$ be two solutions of (\ref{cocycle6-1}). Then as right $A_u$-modules, $(p_k+\epsilon (-\lambda_k p_k +a_k p_k +p_k b_k+c_k))A_u$ is isomorphic to $(p_k+\epsilon (-\lambda_k p_k +a'_k p_k +p_k b'_k+c_k))A_u$.
\end{lemma}
\begin{proof} For $y,z\in A$ we compute
\begin{eqnarray*}
&& \big(1+y\epsilon \big)\big(p_k+\epsilon (-\lambda_k p_k +a_k p_k +p_k b_k+c_k)\big)\big(1+z\epsilon \big)\\
&=& \big(p_k+\epsilon (-\lambda_k p_k +a_k p_k +p_k b_k+c_k+yp_k+u(1,p_k))\big)\big(1+z\epsilon \big)\\
&=& p_k+\epsilon \big(-\lambda_k p_k +a_k p_k +p_k b_k+c_k+yp_k+u(1,p_k)+p_k z+u(p_k,1)\big)\\
&=& p_k+\epsilon \big(-\lambda_k p_k +a_k p_k +p_k b_k+c_k+(y+u(1,1))p_k+p_k (z+u(1,1))\big).
\end{eqnarray*}
Thus we can choose $y,z\in A$ such that 
\begin{multline*}
\big(1+y\epsilon \big)\big(p_k+\epsilon (-\lambda_k p_k +a_k p_k +p_k b_k+c_k)\big)\big(1+z\epsilon \big)\\=p_k+\epsilon \big(-\lambda_k p_k +a'_k p_k +p_k b'_k+c_k\big),
\end{multline*}
then map for $\forall w,v\in A$,
\begin{multline*}
\big(p_k+\epsilon (-\lambda_k p_k +a'_k p_k +p_k b'_k+c_k)\big)\big(w+v\epsilon \big)\\
\mapsto \big(p_k+\epsilon (-\lambda_k p_k +a_k p_k +p_k b_k+c_k)\big)\big(1+z\epsilon \big)\big(w+v\epsilon \big)
\end{multline*}
gives an isomorphism 
\[
\big(p_k+\epsilon (-\lambda_k p_k +a_k p_k +p_k b_k+c_k)\big)A_u \xrightarrow{\sim} \big(p_k+\epsilon \big(-\lambda_k p_k +a'_k p_k +p_k b'_k+c_k)\big)A_u.
\]
\end{proof}

\begin{lemma}\label{endo0}
We have the following identities.
\begin{enumerate} 
  \item[(i)] For $1\leq k\leq n$,
  \begin{equation}\label{endo0-6} 
  p_k u(p_k,x)-u(p_k,x)+u(p_k, p_k x)-u(p_k,p_k)x=0,
  \end{equation}
  \begin{equation}\label{endo0-7}
  x u(p_k,p_k)-u(xp_k,p_k)+u(x,p_k)-u(x,p_k)p_k=0,
  \end{equation}
  \begin{equation}\label{endo0-8}
  p_k u(p_k x,p_k)-u(p_k x,p_k)+u(p_k,p_kxp_k)-u(p_k,p_k x)p_k=0,
  \end{equation}
  \begin{equation}\label{endo0-9}
  p_k u(xp_k,p_k)-u(p_k xp_k,p_k)+u(p_k,xp_k)-u(p_k,xp_k)p_k=0,
  \end{equation}
  \begin{equation}\label{endo0-10}
  p_k u(x,p_k)-u(p_k x,p_k)+u(p_k,xp_k)-u(p_k,x)p_k=0,
  \end{equation}
  \item[(ii)] For $1\leq i<j\leq n$,
  \begin{equation}\label{endo0-3}
  p_i u(xp_j,p_j)+u(p_i,xp_j)-u(p_i,xp_j)p_j=0,
  \end{equation}
  \begin{equation}\label{endo0-4}
  p_i u(x,p_j)-u(p_i x,p_j)+u(p_i,xp_j)-u(p_i,x)p_j=0,
  \end{equation}
  \begin{equation}\label{endo0-5}
  p_i u(p_i x,p_j)-u(p_i x,p_j)-u(p_i,p_i x)p_j=0.
  \end{equation}
  \item[(iii)] For $1\leq k\leq n$,
  \begin{multline}\label{endo0-14}
  p_k u(xp_k,p_k)-p_k u(x,p_k)+p_k u(p_k x,p_k)\\
  =u(p_k,p_k x)p_k+u(p_k,xp_k)p_k-u(p_k,x)p_k.
  \end{multline}
  \item[(iv)] For $1\leq i<j\leq n$,
  \begin{multline}\label{endo0-13}
  u(p_i,p_i x)p_j+u(p_i,xp_j)p_j-u(p_i,x)p_j\\
  =  p_i u(p_i x,p_j)-p_i u(x,p_j)+p_i u(xp_j,p_j)=0.
  \end{multline}
  
\end{enumerate}
\end{lemma}
\begin{proof} (i) and (ii) follows easily from (\ref{cocycle1}) and that there is no path starts from $p_i$ and ends at $p_j$. For (iii) we compute
\begin{eqnarray*}
&& u(p_i,p_i x)p_j+u(p_i,xp_j)p_j-u(p_i,x)p_j\\
&=& p_i u(p_i x,p_j)-u(p_i x,p_j)+u(p_i,xp_j)p_j-u(p_i,x)p_j\\
&=& p_i u(p_i x,p_j)-p_i u(x,p_j)-u(p_i,xp_j)+u(p_i,xp_j)p_j\\
&=& p_i u(p_i x,p_j)-p_i u(x,p_j)+p_i u(xp_j,p_j),
\end{eqnarray*}
where for the first equality we use (\ref{endo0-5}), for the second we use (\ref{endo0-4}) and for the third we use (\ref{endo0-3}).

For (iv) we compute
\begin{eqnarray*}
&&p_k u(xp_k,p_k)-p_k u(x,p_k)+p_k u(p_k x,p_k)\\
&=&u(p_k xp_k,p_k)-u(p_k,xp_k)+u(p_k,xp_k)p_k-p_k u(x,p_k)+p_k u(p_k x,p_k)\\
&=&u(p_k xp_k,p_k)+u(p_k,xp_k)p_k+p_k u(p_k x,p_k)-u(p_k x,p_k)-u(p_k,x)p_k\\
&=&u(p_k,p_k x)p_k+u(p_k,xp_k)p_k-u(p_k,x)p_k,
\end{eqnarray*}
where for the first equality we use (\ref{endo0-9}), for the second we use (\ref{endo0-10}),  for the third we use (\ref{endo0-8}) and   $u(p_k x p_,,p_k)=u(p_k,p_k x p_k)$ because $p_k x p_k=\mu p_k$ for some $\mu \in \Bbbk$.
\end{proof}

\begin{lemma}\label{endo1}
For $1\leq i\leq n$, let $p_i^\dagger=p_i+\epsilon (-\lambda_i p_i +a_i p_i +p_i b_i+c_i)$ be a system of solutions to (\ref{cocycle6-1}) and (\ref{cocycle6-2}), and let  $P_i=p_i^\dagger A_u$, then
\begin{enumerate}
  \item[(i)] For $1\leq i<j\leq n$, $\Hom_{A_u}(P_j,P_i)=0$.
  \item[(ii)] For $1\leq i\leq n$, the composition of homomorphisms 
\[
\Bbbk[\epsilon]\rightarrow A_u\rightarrow\Hom_{A_u}(P_i,P_i)
\]
   is an isomorphism.
  \item[(iii)]
  \begin{equation}\label{endo1-1}
A_u\cong\bigoplus_{1\leq i\leq j\leq n}\Hom_{A_u}(P_i,P_j).
\end{equation}
\end{enumerate}

\end{lemma}
\begin{proof} By definition $p_i^\dagger$ are idempotents of $A_u$. By \cite[prop. I.4.9]{ARS}, for $1\leq i,j\leq n$,
\[
 \Hom_{A_u}(P_i,P_j)=p_j^\dagger A_u p_i^\dagger.
 \] 
Thus we can show (i) and (ii) by direct computations.
For $i<j$ and $x,y\in A$, we compute
\begin{eqnarray*}
&& \big(p_i+\epsilon (-\lambda_i p_i+c_i+a_i p_i +p_i b_i)\big)\big(x+\epsilon y\big)\big(p_j+\epsilon (-\lambda_j p_j+c_j+a_j p_j +p_j b_j)\big)\\
&=&\big(p_i x+\epsilon (-\lambda_i p_i x+c_i x+a_i p_i x+p_i b_i x+p_i y+u(p_i,x))\big)
\big(p_j+\epsilon (-\lambda_j p_j+c_j+a_j p_j +p_j b_j)\big)\\
&=& \epsilon \big(c_i x p_j+u(p_i,x)p_j+p_i x c_j+u(p_i x,p_j)\big)\\
&=& \epsilon \big(u(p_i,p_i x)p_j+p_i x c_j+u(p_i x,p_j)\big)\\
&=&\epsilon \big(u(p_i,p_i x)p_j+p_i u(xp_j,p_j)-p_i u(x,p_j)+u(p_ix,p_j)\big)\\
&=& \epsilon \big(u(p_i,p_i x)p_j+u(p_i,xp_j)p_j-u(p_i,xp_j)
-p_i u(x,p_j)+u(p_ix,p_j)\big)\\
&=& \epsilon \big(u(p_i,p_i x)p_j+u(p_i,xp_j)p_j-u(p_i,x)p_j\big),
\end{eqnarray*}
where for the first and second equalities we use the  property about directions of $p_i$, $p_j$ and $c_i$, $c_j$ described in lemma \ref{cocycle2} (i), for the third equality we use (\ref{endo0-6}), for the fourth  we use (\ref{endo0-7}), for the fifth  we use (\ref{endo0-3}), and for the sixth  we use (\ref{endo0-4}). Thus by (\ref{endo0-13}) we obtain (i).\\

For $1\leq k\leq n$, and $x,y\in A$ we compute
\begin{eqnarray}\label{endo1-2}
&& \big(p_k+\epsilon (-\lambda_k p_k+c_k+a_k p_k +p_k b_k)\big)\big(x+\epsilon y\big)\big(p_k+\epsilon (-\lambda_k p_k+c_k+a_k p_k +p_k b_k)\big)\nonumber\\
&=&\big(p_k x+\epsilon (-\lambda_k p_k x+c_k x+a_k p_k x+p_k b_kx+p_k y+u(p_k,x))\big)\big(p_k+\epsilon (-\lambda_k p_k+c_k+a_k p_k +p_k b_k)\big)\nonumber\\
&=& p_k x p_k+\epsilon \big(-\lambda_k p_k x p_k+c_k x p_k+a_k p_k x p_k+p_k b_kx p_k+p_k y p_k+u(p_k,x)p_k\nonumber\\
&&-\lambda_k p_k x p_k+p_k x c_k+p_k x a_k p_k +p_k xp_k b_k+u(p_k x,p_k)\big)\nonumber\\
&=& p_k x p_k+\epsilon \big(-2\lambda_k p_k x p_k+c_k x p_k+a_k p_k x p_k+p_k y p_k+u(p_k,x)p_k\nonumber\\
&&+p_k x c_k+p_k xp_k b_k+u(p_k x,p_k)\big),
\end{eqnarray}
where for the first and second equalities we use the  property about directions of $p_i$, $p_j$ and $c_i$, $c_j$ described in lemma \ref{cocycle2} (i), and for the third equality we use
\[
p_k b_kx p_k=p_k x a_k p_k=0,
\]
which follows from $p_k a_k=b_k p_k=0$ by lemma \ref{cocycle6}. Then we compute
\begin{eqnarray}\label{endo1-3}
&& c_k x p_k+u(p_k,x)p_k+p_k x c_k+u(p_k x,p_k)\nonumber\\
&=&-2 \lambda_k p_k xp_k+u(p_k,p_k)xp_k+u(p_k,x)p_k+p_k xu(p_k,p_k)+u(p_k x,p_k)\nonumber\\
&=&-2 \lambda_k p_k xp_k+p_k u(p_k,x)p_k+u(p_k, p_k x)p_k+p_k xu(p_k,p_k)+u(p_k x,p_k)\nonumber\\
&=&-2 \lambda_k p_k xp_k+p_k u(p_k,x)p_k+u(p_k, p_k x)p_k\nonumber\\
&&+p_k u(xp_k,p_k)-p_k u(x,p_k)+p_k u(x,p_k)p_k+u(p_k x,p_k)\nonumber\\
&=&-2 \lambda_k p_k xp_k+p_k u(p_k,x)p_k+p_k u(xp_k,p_k)-p_k u(x,p_k)+p_k u(x,p_k)p_k\nonumber\\
&&+p_k u(p_k x,p_k)+u(p_k,p_kxp_k),
\end{eqnarray}
where for the first equality we use (\ref{cocycle3}), for the second we use (\ref{endo0-6}), for the third we use (\ref{endo0-7}) and for the fourth we use (\ref{endo0-8}). Now by (\ref{endo0-14}) we have
\begin{equation*}
p_k u(xp_k,p_k)-p_k u(x,p_k)+p_k u(p_k x,p_k)
=\big(p_k u(xp_k,p_k)-p_k u(x,p_k)+p_k u(p_k x,p_k)\big)p_k,
\end{equation*}
thus substitute (\ref{endo1-3}) into the equality of (\ref{endo1-2}), we obtain
\begin{eqnarray}\label{endo1-4}
&& \big(p_k+\epsilon (-\lambda_k p_k+c_k+a_k p_k +p_k b_k)\big)\big(x+\epsilon y\big)\big(p_k+\epsilon (-\lambda_k p_k+c_k+a_k p_k +p_k b_k)\big)\nonumber\\
&=& p_k x p_k+\epsilon \big(p_k T(x,k) p_k+a_k p_k x p_k+p_k y p_k+p_k xp_k b_k+u(p_k,p_kxp_k)\big),
\end{eqnarray}
where $T(x,k)\in A$ depends only $x$ and $k$, and we can ignore its complicated form. There exists $\mu(x,k),\nu(y_k), \tau(x,k)\in \Bbbk$ which depends on $k$ and $x$ or $y$ as the notations indicate, such that 
\begin{equation}\label{endo1-5}
p_k x p_k=\mu(x,k)p_k,\ p_k y p_k=\nu(y,k)p_k,\ p_k T(x,k) p_k=\tau(x,k)p_k.
\end{equation}
But for an element $\mu+\epsilon \nu\in \Bbbk[\epsilon]$, the scalar multiplication on $P_k$ is given by (\ref{scalar0}), i.e.,
\begin{eqnarray}\label{endo1-6}
&& \big(\mu\cdot 1_u+\epsilon  \nu\big)\big(p_k+\epsilon (-\lambda_k p_k+c_k+a_k p_k +p_k b_k)\big)\nonumber\\
&=&\big(\mu+\epsilon (-\mu u(1,1)+\nu)\big)\big(p_k+\epsilon (-\lambda_k p_k+c_k+a_k p_k +p_k b_k)\big)\nonumber\\
&=& \mu p_k+\epsilon (-\mu\lambda_k  p_k+\mu c_k+\mu a_k p_k +\mu p_k b_k-\mu u(1,1)p_k+\nu p_k+u(\mu,p_k))\nonumber\\
&=& \mu p_k+\epsilon (-\mu\lambda_k  p_k+\mu c_k+\mu a_k p_k +\mu p_k b_k+\nu p_k),
\end{eqnarray}
where for the third equality we use $\mu u(1,1)p_k=u(\mu,p_k)$ by lemma \ref{cocycle2} (iii). Comparing (\ref{endo1-4}) and (\ref{endo1-6}) using (\ref{endo1-5}), we obtain
\begin{multline*}
\big(p_k+\epsilon (-\lambda_k p_k+c_k+a_k p_k +p_k b_k)\big)A_u\big(p_k+\epsilon (-\lambda_k p_k+c_k+a_k p_k +p_k b_k)\big)\\
=\Bbbk[\epsilon].\big(p_k+\epsilon (-\lambda_k p_k+c_k+a_k p_k +p_k b_k)\big).
\end{multline*}
This completes the proof of (ii). Finally, by (\ref{cocycle6-3}), $\{P_i\}_{1\leq i\leq n}$ is a \emph{complete} list of projective $A_u$-modules. Regard $A_u$ as a $\Bbbk$-\emph{algebra}, we obtain (\ref{endo1-1}) by \cite[section 5]{Bon89}.
\end{proof}

\begin{proof}[Proof of theorem \ref{deform-alg1}] Since $p_j A$ is an indecomposable projective right $A$-module, by \cite[I.4.4]{ARS} a projective right $A_u$-module $P_i$ deforming $p_j A$ is still indecomoposable. Then by \cite[I.4.5 and I.4.8]{ARS}, one easily sees that $P_j$ is of the form $p_j^\dagger A_u$ where $p_j^\dagger$ is an idempotent of $A_u$ such that $p_j^\dagger\ \mathrm{mod}\ (\epsilon)=p_j$. Thus (i) follows from lemma \ref{cocycle6} (i) and lemma \ref{cocycle7} and the remaining statements follow from lemma \ref{cocycle6} (ii) and lemma \ref{endo1}.
\end{proof}

\section{First order noncommutative deformations of exceptional collections}
Let $X$ be a smooth proper scheme over $\Bbbk$, $T_X$ the tangent sheaf of $X$. 
Given $\beta\in  H^1(X,T_X)$, there is a canonically associated smooth projective scheme $X_{\beta}$ over $\Bbbk[\epsilon]$ which deforms $X$. Generalizing this classical fact, Toda in \cite{Toda05} introduced the notion of \emph{noncommutative deformation} $X_{\alpha,\beta,\gamma}$ associated to an element  $(\alpha,\beta,\gamma)\in H^2(X,\mathcal{O}_X)\oplus H^1(X,T_X)\oplus H^0(X,\wedge^2 T_X)$. Let us recall Toda's definition. 
\begin{definition}\label{def-ncdeform}\cite[\S 3, \S 4]{Toda05}
Choose an affine open covering  $\mathcal{U}=\{U_i\}_{i\in I}$ of $X$, and choose \v{C}ech representatives $\{\alpha_{ijk}\}_{i,j,k\in I}\in \check{C}^2(\mathcal{U},\mathcal{O}_X)$, $\{\beta_{ij}\}_{i,j\in I}\in \check{C}^1(\mathcal{U},T_X)$ of $\alpha$ and $\beta$ respectively. We regard $\gamma\in \Gamma(X,T^2_X)$ as an \emph{antisymmetric bi-derivation}, i.e, a $\Bbbk$-linear homomorphism $\mathcal{O}_X\otimes_{\Bbbk} \mathcal{O}_X\rightarrow \mathcal{O}_X$, which are derivations in both arguments, and is antisymmetric. Étale locally, $\gamma$ can be written as $\sum_{i,j=1}^{\dim X}f_{ij} \partial_{x_{i}}\wedge \partial_{x_j}$ where $x_1,...,x_{\dim X}$ are étale local coordinates of $X$, and $f_{ij}$ are  regular functions on the corresponding chart.

The noncommutative deformation $X_{\alpha,\beta,\gamma}$ consists of the following data. 
\begin{enumerate}
  \item The underlying space is identified to $X$.
  \item There is a sheaf of $\Bbbk[\epsilon]$-algebras $\mathcal{O}_{X_{\beta,\gamma}}$ defined as follows. As a sheaf of $\Bbbk[\epsilon]$-modules, $\mathcal{O}_{X_{\beta,\gamma}}$ is the kernel of 
  \[
  \mathcal{O}_X\oplus \check{C}^0(\mathcal{U},\mathcal{O}_X)\ni (a+\epsilon\{b_i\}_{i\in I})\mapsto 
  \{-\beta_{ij}(a)+\delta\{b_i\}\}_{i,j\in I}\in \check{C}^1(\mathcal{U},\mathcal{O}_X),
  \]
  and the multiplication is given by 
  \[
  (a+\epsilon \{b_i\})\cdot (c+\epsilon\{d_i\})=ac+\epsilon\{b_i c+a d_i+\gamma(a,c)\}.
  \]
  \item An $\mathcal{O}_{X_{\beta,\gamma}}$-module twisted by $\alpha$, is a collection $\{F_i\}_{i\in I}$, where $F_i$ is an $\mathcal{O}_{X_{\beta,\gamma}}|_{U_i}$-module, and a collection $\{\phi_{ij}\}_{i,j\in I}$, where $\phi_{ij}: F_i|_{U_i\cap U_j}\rightarrow F_j|_{U_i\cap U_j}$ is an isomorphism $\mathcal{O}_{X_{\beta,\gamma}}|_{U_{ij}}$-modules, such that 
  \[
  \phi_{ki}\circ \phi_{jk}\circ \phi_{ij}=\mathrm{id} -  \alpha_{ijk}\epsilon.
  \]
\end{enumerate}

The above definition is independent of the choice of $\mathcal{U}$. For brevity we call an $\mathcal{O}_{X_{\beta,\gamma}}$-module twisted by $\alpha$, an  \emph{$\mathcal{O}_{X_{\alpha,\beta,\gamma}}$-module}. Similarly for the $\mathcal{O}_{X_{\alpha,\beta,\gamma}}$-linear homomorphisms. We say that an $\mathcal{O}_{\alpha,\beta,\gamma}$-module $F$ is quasi-coherent (resp. coherent, resp. locally free), if $F|_{U_i}$ is a quasi-coherent (resp. coherent, resp. locally free) $\mathcal{O}_{X_{\beta,\gamma}}|_{U_i}$-module. 
  A locally free coherent  $\mathcal{O}_{\alpha,\beta,\gamma}$-module is also called a vector bundle on $X_{\alpha,\beta,\gamma}$. The derived category of $\mathcal{O}_{\alpha,\beta,\gamma}$-modules (resp. quasi coherent $\mathcal{O}_{\alpha,\beta,\gamma}$-module, resp. coherent $\mathcal{O}_{\alpha,\beta,\gamma}$-modules)  are denoted by $\mathrm{D}^*(\mathcal{O}_{X_{\alpha,\beta,\gamma}})$ (resp. $\mathrm{D}_{\mathrm{qcoh}}^*(\mathcal{O}_{X_{\alpha,\beta,\gamma}})$, resp. $\mathrm{D}_{\mathrm{coh}}^*(\mathcal{O}_{X_{\alpha,\beta,\gamma}})$), where $*=-,+$ or $\mathrm{b}$. The full subcategory of $\mathrm{D}_{\mathrm{coh}}^{\mathrm{b}}(\mathcal{O}_{X_{\alpha,\beta,\gamma}})$ consisting of  perfect complexes over $X_{\alpha,\beta,\gamma}$ is denoted by $\mathrm{D}^{\mathrm{p}}(X_{\alpha,\beta,\gamma})$. 

  There is a natural morphism of ringed space $\pi: X_{\alpha,\beta,\gamma}\rightarrow \Spec(\Bbbk[\epsilon])$, whose corresponding homomorphism of sheaf of rings $\pi^{-1}\Bbbk[\epsilon]\rightarrow \mathcal{O}_{X_{\beta,\gamma}}$ is flat. In particular, on $X_{0,\beta,\gamma}$, the notion of quasi-coherent sheaves (resp. coherent sheaves, resp. locally free sheaves) reduce to the usual ones on a ringed space.

  By \cite[lemma 4.3]{Toda05}, the category of $\mathcal{O}_{X_{\alpha,\beta,\gamma}}$-modules have enough injectives, thus the derived functor $\mathbf{R}\Hhom$ is defined. 
\end{definition}

The following corollary \ref{nc-perf-2} will be used only in the proof of the strongness statement of theorem \ref{deform-bundle2}, which is not needed for the proof of theorem \ref{comparison0}. We outline a proof parallel to the usual one for schemes.
\begin{lemma}\label{nc-finite-1}
Given $X_{0,\beta,\gamma}$, there exists $N$ such that for $q>N$ and any $\mathcal{O}_{X_{0,\beta,\gamma}}$-quasi-coherent sheaf $F$, $\mathbf{R}^q \pi_* F=0$.
\end{lemma}
\begin{proof} Let $F$ be an $\mathcal{O}_{X_{0,\beta,\gamma}}$-quasi-coherent sheaf. For any affine open subset $U$ of $X$, $F|_U$ is an $\mathcal{O}_X$-quasi-coherent sheaf, thus $H^i(U,F)=0$ for $i>0$. By Leray's theorem \cite[II, 5.9.2]{Gode58}, it follows that for any finite affine open covering $\mathcal{U}$ of $X$, $\check{H}^i(\mathcal{U},F)\cong H^i(X_{0,\beta,\gamma},F)$. Thus the conclusion follows from the existence of a finite affine open covering of $X$ because of the properness of $X\rightarrow \Spec(\Bbbk)$. 
\end{proof}

Now the arguments of \cite[3.7 and 3.7.1]{Ill71} carry over verbatim to deduce the following two lemmas. See also \cite[8.3.8]{Ill05}.
\begin{lemma}\label{nc-proj-formula}
For any  $G\in \mathrm{D}^{\mathrm{b}}_{\mathrm{qcoh}}(\Bbbk[\epsilon])$ and $F\in\mathrm{D}^{\mathrm{b}}_{\mathrm{qcoh}}(X_{0,\beta,\gamma})$, there is a canonical isomorphism
$\mathbf{R}\pi_*(\pi^*G\otimes^{\mathbf{L}}F)\cong G\otimes^{\mathbf{L}}_{\Bbbk[\epsilon]}\mathbf{R}\pi_*F$.
\end{lemma}

\begin{lemma}\label{nc-perf-1}
Let $F$ be a perfect complex of $\mathcal{O}_{X_{0,\beta,\gamma}}$-modules. Then $\mathbf{R}\pi_* F$ is perfect over $\Spec(\Bbbk[\epsilon])$.
\end{lemma}

\begin{corollary}\label{nc-perf-2}
Let $E,F$ be  perfect complexes of $\mathcal{O}_{X_{\alpha,\beta,\gamma}}$-modules. Then 
\[
\mathbf{R}\pi_* \mathbf{R}\Hhom_{\mathcal{O}_{X_{\alpha,\beta,\gamma}}}(E,F)
\] is perfect over $\Spec(\Bbbk[\epsilon])$.
\end{corollary}
\begin{proof} For open immersions $j:U\rightarrow X$, the extension by zero $j_!$ is exact and left adjoint to $j^*$ \cite[\S 4]{Toda05}, thus $j^* I$ is injective on $U$ for an injective $\mathcal{O}_{X_{\alpha,\beta,\gamma}}$-module. So local properties of $\mathbf{R}\pi_* \mathbf{R}\Hhom_{\mathcal{O}_{X_{\alpha,\beta,\gamma}}}(E,F)$ can be computed locally. Then one easily sees that 
$\mathbf{R}\Hhom_{\mathcal{O}_{X_{\alpha,\beta,\gamma}}}(E,F)$ is a perfect complex of $\mathcal{O}_{X_{0,\beta,\gamma}}$-modules. The conclusion follows from lemma \ref{nc-perf-1}.
\end{proof}

We define  exceptional collections (resp. strong ..., resp. full ...) of $\mathrm{D}^{\mathrm{p}}(X_{\alpha,\beta,\gamma})$ relative to $\Bbbk[\epsilon]$ as the definition \ref{definition-ec}.

From now on in  this section we study the deformations of strong exceptional collections consisting of vector bundles, over a noncommutative deformation $X_{\alpha,\beta,\gamma}$.

Let $E$ be a vector bundle over $X$,  $\mathcal{U}=\{U_i\}_{i\in I}$  an open covering of $X$, and denote $U_{ij}=U_i\cap U_j$, $U_{ijk}=U_i\cap U_j\cap U_k$, for $i,j,k\in I$. 
Regarding $E|_{U_i}\oplus E|_{U_i}$ as a vector bundle over $U_i\times_{\Bbbk}\Bbbk[\epsilon]$, we want to glue them to obtain a vector bundle over $X_{\alpha,\beta,\gamma}$. Then we need to specify the isomorphisms
\[
 (E|_{U_i}\oplus E|_{U_i})|_{U_{ij}}\xrightarrow{\psi_{ij}} (E|_{U_j}\oplus E|_{U_j})|_{U_{ij}}
\]
where $\psi_{ij}\in \Hom_{\Bbbk[\epsilon]}(E|_{U_i}\oplus E|_{U_i},E|_{U_j}\oplus E|_{U_j})$, 
 such that 
\begin{equation}\label{eq5}
\psi_{ki}\circ\psi_{jk}\circ\psi_{ij}=\mathrm{id}-\alpha_{ijk}\epsilon.  
\end{equation}
Shrinking $\mathcal{U}$ if necessary, we choose connections $\nabla_i:E|_{U_i}\rightarrow E|_{U_i}\otimes_{\mathcal{O}_{U_i}}\Omega^1_{U_i}$ for $i\in I$.
\begin{lemma}\label{def-bundle-1}
The isomorphisms $\{\psi_{ij}\}_{i,j\in I}$ glue $\{E|_{U_i}\oplus E|_{U_i}\}_{i\in I}$ to be a vector bundle over $X_{\alpha,\beta,\gamma}$ if and only if $\psi_{ij}$ are of the form
\[
\psi_{ij}=\left(
\begin{array}{cc}
1 & 0\\
g_{ij} & 1
\end{array}
\right)
\]
where $g_{ij}\in \Hom_{\Bbbk}(E|_{U_i},E|_{U_j})$, and $(g_{ij})_{i,j\in I}$ satisfy
\begin{equation}\label{def-bundle-1-1}
\begin{cases}
g_{ij}r-rg_{ij}=-\beta_{ij}(r)+\gamma(r,\cdot)\circ \nabla_j-\gamma(r,\cdot)\circ \nabla_i, & \forall r \in \Gamma(U_{ij},\mathcal{O}_X),\\
g_{ij}|_{U_{ijk}}+g_{jk}|_{U_{ijk}}+g_{ki}|_{U_{ijk}}=-\alpha_{ijk}.
\end{cases}
\end{equation}
\end{lemma}
\begin{proof} Write
\[
\psi_{ij}=\left(
\begin{array}{cc}
1 & f_{ij}\\
g_{ij} & h_{ij}
\end{array}
\right),
\]
where $f_{ij},g_{ij},h_{ij}$ are, a priori, $\Bbbk$-linear endomorphisms of $E|_{U_{ij}}$. The $\mathcal{O}_{X_{\beta,\gamma}}$-linearity of $\psi_{ij}$ means
\[
\psi_{ij} \left(
\begin{array}{cc}
r & 0\\
s_i+\gamma(r,\cdot)\circ \nabla_i & r
\end{array}
\right)=\left(
\begin{array}{cc}
r & 0\\
s_j+\gamma(r,\cdot)\circ \nabla_j & r
\end{array}
\right)\psi_{ij},
\]
for any $r,s_i,s_j\in \Gamma(U_{ij},\mathcal{O}_X) $ satisfying 
\begin{equation}\label{eq5.5}
s_i-s_j=\beta_{ij}(r).
\end{equation}
 Thus
\begin{eqnarray*}
&&\left(
\begin{array}{cc}
r+f_{ij}s_i+f_{ij}\circ \gamma(r,\cdot)\circ \nabla_i & f_{ij}r\\
g_{ij}r+h_{ij}s_i+h_{ij}\circ \gamma(r,\cdot)\circ \nabla_i & h_{ij}r
\end{array}
\right)\\
&=&
\left(
\begin{array}{cc}
r & rf_{ij}\\
s_j+\gamma(r,\cdot)\circ \nabla_j+ r g_{ij} & s_j f_{ij}+\gamma(r,\cdot)\circ \nabla_j\circ f_{ij}+r h_{ij}
\end{array}
\right).
\end{eqnarray*}
So $f_{ij}=0$, and $h_{ij}r=rh_{ij}$, i.e. $h_{ij}$ is $\mathcal{O}_X$-linear, and 
\begin{equation}\label{eq6}
g_{ij}r+h_{ij}s_i+h_{ij}\circ \gamma(r,\cdot) \circ \nabla_i=s_j+\gamma(r,\cdot)\circ \nabla_j+r g_{ij}.
\end{equation}
Since $s_j=s_i-\beta_{ij}(r)$, (\ref{eq6}) holds for all $r,s_i,s_j\in \Gamma(U_{ij},\mathcal{O}_X)$ satisfying (\ref{eq5.5}) if and only if
\[
h_{ij}=\mathrm{id}
\]
and
\begin{equation*}
g_{ij}r-rg_{ij}=-\beta_{ij}(r)+\gamma(r,\cdot)\circ \nabla_j-\gamma(r,\cdot)\circ \nabla_i.
\end{equation*}
The condition (\ref{eq5}) reduces to the second equation of (\ref{def-bundle-1-1}). 
\end{proof} 

\begin{lemma}\label{def-bundle-1.5}
If $E$ is exceptional, there exists an open covering $\mathcal{U}$ such that the solution to (\ref{def-bundle-1-1}) exists, and the corresponding vector bundle on $X_{\alpha,\beta,\gamma}$ is unique up to  canonical isomorphisms.
\end{lemma}
\begin{proof} Shrinking $\mathcal{U}$ if necessary, we can assume that $\mathcal{U}$ is an affine covering, and that there exists a solution $(\tilde{g}_{ij})$ of the first equation of  (\ref{def-bundle-1-1}). Then for $r\in \Gamma(U_{ijk},\mathcal{O}_X)$,
\[
(\tilde{g}_{ij}+\tilde{g}_{jk}+\tilde{g}_{ki})r-r(\tilde{g}_{ij}+\tilde{g}_{jk}+\tilde{g}_{ki})=-(\beta_{ij}+\beta_{jk}+\beta_{ki})(r).
\]
Since $\beta\in H^1(X,T_X)$, the assignment $(i,j,k)\mapsto \tilde{g}_{ij}+\tilde{g}_{jk}+\tilde{g}_{ki}$ lies in 
$\check{Z}^2(\mathcal{U},\Hhom_{\mathcal{O}_X}(E,E))$; denote it by $\delta \tilde{g}$, and notice that it does not lie in $\check{B}^2(\mathcal{U},\Hhom_{\mathcal{O}_X}(E,E))$ because $\tilde{g}_{ij}$ is not $\mathcal{O}_X$-linear. It suffices to find $x=(x_{ij})\in \check{C}^1(\mathcal{U},\Hhom_{\mathcal{O}_X}(E,E))$, such that
\[
\delta x=-\alpha- \delta \tilde{g},
\]
and thus $g=x+\tilde{g}$ is a solution to (\ref{def-bundle-1-1}). Since $\check{H}^2(\mathcal{U},\Hhom_{\mathcal{O}_X}(E,E))=\Ext^2(E,E)=0$, such $x$ exists. 

If $g'$ is another solution,  $h=g-g'$ is $\mathcal{O}_X$-linear and therefore lies in $\check{Z}^1(\mathcal{U},\Hhom_{\mathcal{O}_X}(E,E))$. Since $\check{H}^1(\mathcal{U},\Hhom_{\mathcal{O}_X}(E,E))=\Ext^1(E,E)=0$, $h=\delta x$ for some $x\in \check{C}^0(\mathcal{U},\Hhom_{\mathcal{O}_X}(E,E))$, and therefore it is easy to construct an isomorphism between the  vector bundle corresponding to $g$ and $g'$. 
\end{proof}

\begin{definition}
For an exceptional vector bundle $E$ on $X$, and $(\alpha,\beta,\gamma)\in  H^2(X,\mathcal{O}_X)\oplus H^1(X,T_X)\oplus H^0(X,\wedge^2 T_X)$, denote the unique vector bundle on $X_{\alpha,\beta,\gamma}$ deforming $E$ by $E_{\alpha,\beta,\gamma}$.
\end{definition}

Let $E$ and $F$ be a strong exceptional pair of vector bundles on $X$. We want to compute
\begin{equation}\label{def-bundle-2}
\Hom_{\mathcal{O}_{X_{\alpha,\beta,\gamma}}}(E_{\alpha,\beta,\gamma},F_{\alpha,\beta,\gamma}).
\end{equation}
 Still take an open cover $\mathcal{U}=(U_i)$ and follow the above notations. First of all, an element of (\ref{def-bundle-2}) modulo $\epsilon$ is an element of $\Hom_{\mathcal{O}_X}(E,F)$. So we fix $a\in \Hom_{\mathcal{O}_X}(E,F)$, and denote the restriction $a|_{U_i}$ still by $a$. 
\begin{lemma}
Assume $b_i\in \Hom_{\Bbbk}(F|_{U_i},E|_{U_i})$, $c_i\in \Hom_{\Bbbk}(E|_{U_i},F|_{U_i})$, $d_i\in \Hom_{\Bbbk}(F|_{U_i},F|_{U_i})$ for $i\in I$. Then
\[
\left(
\begin{array}{cc}
a & b_i \\ c_i & d_i
\end{array}\right):
E|_{U_i}\oplus E|_{U_i}\rightarrow F|_{U_i}\oplus F|_{U_i}
\]
glue to be an $\mathcal{O}_{X_{\alpha,\beta,\gamma}}$-linear homomorphism from $E_{\alpha,\beta,\gamma}$ to $F_{\alpha,\beta,\gamma}$ if and only if $b_i=0$, $d_i=0$ and
\begin{equation}\label{eq12}
\begin{cases}
c_i r-rc_i=\gamma(r,\cdot)\circ \nabla_i\circ a-a\circ \gamma(r,\cdot)\circ \nabla_i,\ \forall r\in \Gamma(U_i,\mathcal{O}_X)
\\
c_i-c_j=a g^E_{ij}-g^F_{ij}a.
\end{cases}
\end{equation}
\end{lemma}
\begin{proof} Suppose 
\[
\left(
\begin{array}{cc}
a & b_i \\ c_i & d_i
\end{array}\right):
E|_{U_i}\oplus E|_{U_i}\rightarrow F|_{U_i}\oplus F|_{U_i}
\]
is an $\mathcal{O}_{X_{\alpha,\beta,\gamma}}$-linear homomorphism. Then
\[
\left(
\begin{array}{cc}
r & \\ s_i+\gamma(r,\cdot)\circ \nabla_i & r
\end{array}
\right)
\left(
\begin{array}{cc}
a & b_i \\ c_i & d_i
\end{array}\right)=
\left(\begin{array}{cc}
a & b_i \\ c_i & d_i
\end{array}\right)
\left(\begin{array}{cc}
r & \\ s_i+\gamma(r,\cdot)\circ \nabla_i & r
\end{array}
\right)
\]
for any $r,s_i\in \Gamma(U_i,\mathcal{O}_X) $. This is equivalent to
\[
b_i=0,\ r d_i=d_i r,
\]
and
\begin{equation}\label{eq9}
s_i a+\gamma(r,\cdot)\circ \nabla_i\circ a+r c_i=c_i r+d_i s_i+d_i\circ \gamma(r,\cdot)\circ \nabla_i.
\end{equation}
These holds for all $r,s_i$ if and only if $d_i=a$ and
\begin{equation}\label{eq11}
c_i r-r c_i=\gamma(r,\cdot)\circ \nabla_i\circ a-a\circ \gamma(r,\cdot)\circ \nabla_i.
\end{equation}
Moreover, a system of homomorphisms 
\[
\left\{\left(
\begin{array}{cc}
a & 0 \\ c_i & a
\end{array}\right)\right\}_{i\in I}
\]
glue to be an element of $\Hom_{\mathcal{O}_{X_{\alpha,\beta,\gamma}}}(E_{\alpha,\beta,\gamma},F_{\alpha,\beta,\gamma})$ if and only if
\[
\left(
\begin{array}{cc}
1 & \\ g^F_{ij} & 1
\end{array}
\right)
\left(
\begin{array}{cc}
a & 0 \\ c_i & a
\end{array}\right)=
\left(\begin{array}{cc}
a & 0 \\ c_j & a
\end{array}\right)
\left(\begin{array}{cc}
1 & \\ g^E_{ij} & 1
\end{array}
\right)
\]
which is equivalent to
\[ 
g^F_{ij}a_i+c_i=a_j g^E_{ij}+c_j.
\]
\end{proof}

\begin{lemma}\label{lem-nd2} 
Let $E,F$ be an strong exceptional pair of vector bundles on $X$. Then there exists an open covering $\mathcal{U}$ such that there exists a solution $\{c_i\}_{i\in I}$ to the system of equations (\ref{eq12}). 
And two different  solutions differ by  $\{c'_i\}_{i\in I}$, where $c'_i=c'|_{U_i}$, $i\in I$, for some $c'\in \Hom_{\mathcal{O}_X}(E,F)$.
\end{lemma}
\begin{proof} Shrinking $\mathcal{U}$ if necessary, we can assume that $\mathcal{U}$ is an affine covering, and that there exists a solution $\{\tilde{c}_{i}\}_{i\in I}$ of the first equation. Thus
\[
(\tilde{c}_{i}-\tilde{c}_{j})r-r(\tilde{c}_{i}-\tilde{c}_{j})=\gamma(r,\cdot)\circ(\nabla_i-\nabla_j)\circ a-a\circ\gamma(r,\cdot)\circ(\nabla_i-\nabla_j),
\]
while, by the first equation of (\ref{def-bundle-1-1}), 
\[
(a g^E_{ij}-g^F_{ij}a)r-r(a g^E_{ij}-g^F_{ij}a)=\gamma(r,\cdot)\circ(\nabla_i-\nabla_j)\circ a-a\circ\gamma(r,\cdot)\circ(\nabla_i-\nabla_j). 
\]
So the assignment $(i,j)\mapsto -(\tilde{c}_{i}-\tilde{c}_{j})+(a g^E_{ij}-g^F_{ij}a)$ lies in $\check{C}^1(\mathcal{U},\Hhom_{\mathcal{O}_X}(E,F))$. Moreover, by (\ref{def-bundle-1-1}), 
\[
(a g^E_{ij}-g^F_{ij}a)+(a g^E_{jk}-g^F_{jk}a)+(a g^E_{ki}-g^F_{ki}a)=0,
\]
thus the assignment $(i,j)\mapsto -(\tilde{c}_{i}-\tilde{c}_{j})+(a g^E_{ij}-g^F_{ij}a)$ lies in $\check{Z}^1(\mathcal{U},\Hhom_{\mathcal{O}_X}(E,F))$.
 Since $\check{H}^1(\mathcal{U},\Hhom(E,F))=\Ext^1(E,F)=0$, there exists $x\in \check{C}^0(\mathcal{U},\Hom(E,F))$ such that $\delta x=\{-(\tilde{c}_{i}-\tilde{c}_{j})+(a g^E_{ij}-g^F_{ij}a)\}_{i,j\in I}$, thus $x+\{\tilde{c}_{i}\}_{i\in I}$ gives a  solution to (\ref{eq12}). The second statement is obvious. 
 \end{proof}

Now let $(E_j)_{1\leq j\leq n}$ be a strong exceptional collection of vector bundles.
Let $E=F=\bigoplus_{j=1}^n E_i$, and $A=\Hom_{\mathcal{O}_X}(E,E)$. 
\begin{construction}\label{constr-1}
Choosing a  $\Bbbk$-basis of $A$, by lemma \ref{def-bundle-1.5} and lemma \ref{lem-nd2} there exists  an affine open covering $\mathcal{U}=\{U_i\}_{i\in I}$ of $X$ such that for any $a$ in the chosen basis, 
the system of equations for $g_{ij}\in \Hom_{\Bbbk}(E|_{U_{ij}},E|_{U_{ij}})$ for $i,j\in I$ and $i\neq j$, and $c_i\in \Hom_\Bbbk(E|_{U_i},E|_{U_i})$ for $i\in I$
\begin{equation}\label{eq12.5}
\begin{cases}
g_{ij}r-rg_{ij}=-\beta_{ij}(r)+\gamma(r,\cdot)\circ \nabla_j-\gamma(r,\cdot)\circ \nabla_i, \ \forall r \in \Gamma(U_{ij},\mathcal{O}_X),\\
g_{ij}+g_{jk}+g_{ki}=-\alpha_{ijk},\\
c_i r-rc_i=\gamma(r,\cdot)\circ \nabla_i\circ a-a\circ \gamma(r,\cdot)\circ \nabla_i,\ \forall r\in \Gamma(U_{i},\mathcal{O}_X)\\
c_i-c_j=a g_{ij}-g_{ij}a
\end{cases}
\end{equation}
 has a solution. Thus we can  assign a solution $c(a)_i$ for each $a\in A$, such that $c(\lambda a)_i=\lambda c(a)_i$ for $\lambda\in \Bbbk$. On each $U_i$ we define
\begin{equation}\label{eq13}
u_{\alpha,\beta,\gamma}(a',a)_i=-c(a')_i a-a' c(a)_i+c(a' a)_i,
\end{equation}
which glue to be an $\mathcal{O}_X$-endomorphism of $E$ by the following lemma \ref{lem-nd3}, thus we obtain an element $u_{\alpha,\beta,\gamma}(a',a) \in A$. 
\end{construction}

\begin{lemma}\label{lem-nd3}
The elements $-c(a')_i a-a' c(a)_i+c(a' a)_i$ constructed above are independent of $i$, and are $\mathcal{O}_X$-linear.
\end{lemma}
\begin{proof} First we check the independence of $i$. 
\begin{eqnarray*}
&&\big(-c(a')_i a-a' c(a)_i+c(a' a)_i\big)-\big(-c(a')_j a-a' c(a)_j+c(a' a)_j\big)\\
&=&\big(-c(a')_i a+c(a')_j a\big)+\big(-a' c(a)_i+a' c(a)_j\big)+\big(c(a' a)_i-c(a' a)_j\big)\\
&=& (-a' g_{ij}a+g_{ij}a'a)+(-a' a g_{ij}+a'g_{ij}a)+(a' a g_{ij}-g_{ij}a'a)=0,
\end{eqnarray*}
where for the second equality we use the fourth equation of (\ref{eq12.5}). 
Then we check the $\mathcal{O}_X$-linearity. 
\begin{eqnarray*}
&& \big(-c(a')_i a-a' c(a)_i+c(a' a)_i\big)r-r\big(-c(a')_i a-a' c(a)_i+c(a' a)_i\big)\\
&=&-\big(c(a')_i r-r c(a')_i\big)a-a'\big(c(a)_i r-r c(a)_i\big)+\big(c(a' a)_i r-rc(a' a)_i\big)\\
&=& -\big(\gamma(r,\cdot)\circ \nabla_i\circ a'-a'\circ \gamma(r,\cdot)\circ \nabla_i\big)a-
a'\big(\gamma(r,\cdot)\circ \nabla_i\circ a-a\circ \gamma(r,\cdot)\circ \nabla_i\big)\\
&&+\big(\gamma(r,\cdot)\circ \nabla_i\circ aa'-aa'\circ \gamma(r,\cdot)\circ \nabla_i\big)=0,
\end{eqnarray*}
where for the second equality we use the third equation of (\ref{eq12.5}).
\end{proof}

\begin{lemma}\label{deform-bundle1}
The assignment $(a',a)\mapsto u_{\alpha,\beta,\gamma}(a',a)$ gives an element $u_{\alpha,\beta,\gamma}\in Z^2(A,A)$, and the choices of $c(a)$ and the open covering $\mathcal{U}$ do not affect the class of $u_{\alpha,\beta,\gamma}$ in $HH^2(A)$. Moreover, $u_{\alpha,\beta,\gamma}$ depends $\Bbbk$-linearly on $\alpha$, $\beta$ and $\gamma$.
\end{lemma}
\begin{proof} By construction, $ u_{\alpha,\beta,\gamma}(a',a)$ is $\Bbbk$-linear in $a$ and $a'$, and it is straightforward to verify that $u_{\alpha,\beta,\gamma}(\cdot,\cdot)$ is a cocycle. Given a solution $\{g_{ij}\}_{i,j\in I}$ to the first and second equations of (\ref{eq12.5}), by the last statement of lemma \ref{lem-nd2}, different choices of $c(a)$ do not change the class of $u_{\alpha,\beta,\gamma}$ in $HH^2(A)$. If $\{g'_{ij}\}_{i,j\in I}$ is another solution to the first and second equations of (\ref{eq12.5}), then $\{g'_{ij}\}_{i,j\in I}-\{g_{ij}\}_{i,j\in I}=\{x_{ij}\}_{i,j\in I}$, where $\{x_{ij}\}_{i,j\in I}\in \check{Z}^1(\mathcal{U},\Hhom_{\mathcal{O}_X}(E,E))$. Since $\check{H}^1(\mathcal{U},\Hhom_{\mathcal{O}_X}(E,E))=\Ext^1(E,E)=0$, there exists $\{y_i\}_{i\in I}\in \check{C}^0(\mathcal{U},\Hhom_{\mathcal{O}_X}(E,E))$ such that $\delta(\{y_i\}_{i\in I})=\{x_{ij}\}_{i,j\in I}$. Then we can solve the third and the fourth equation by 
\[
\tilde{c}(a)_i=c(a)+ay_i-y_i a,
\]
for $i\in I$ and $a\in A$. Then the corresponding $\tilde{u}_{\alpha,\beta,\gamma}$ is given by
\begin{eqnarray*}
&&\tilde{u}_{\alpha,\beta,\gamma}(a',a)_i=-\tilde{c}(a')_i a-a' \tilde{c}(a)_i+\tilde{c}(a' a)_i\\
&=& u_{\alpha,\beta,\gamma}(a',a)_i-(a' y_i-y_i a')a-a'(ay_i-y_i a)+(a' a y_i-y_i a' a)\\
&=& u_{\alpha,\beta,\gamma}(a',a)_i.
\end{eqnarray*}
 The remaining statements are also obvious from the construction.
 \end{proof}

Now we are ready to come to the main theorem of this section.

\begin{theorem}\label{deform-bundle2}
Let $\{E_j\}_{1\leq j\leq n}$ be a strong exceptional collection of vector bundles on $X$, and denote $E=\bigoplus_{j=1}^n E_j$ and $A=\Hom_{\mathcal{O}_X}(E,E)$. For $(\alpha,\beta,\gamma)\in  H^2(X,\mathcal{O}_X)\oplus H^1(X,T_X)\oplus H^0(X,\wedge^2 T_X)$, there exists a unique strong exceptional collection of vector bundles $\{(E_j)_{\alpha,\beta,\gamma}\}_{1\leq j\leq n}$ on $X_{\alpha,\beta,\gamma}$ such that $(E_j)_{\alpha,\beta,\gamma}$ is the unique deformation of $E_j$, and
\begin{equation}\label{deform-bundle2-1}
\Hom_{\mathcal{O}_{X_{\alpha,\beta,\gamma}}}(\bigoplus_i (E_i)_{\alpha,\beta,\gamma},\bigoplus_i (E_i)_{\alpha,\beta,\gamma})\cong A_{u_{\alpha,\beta,\gamma}}.
\end{equation}
\end{theorem}
\begin{proof} By corollary \ref{nc-perf-2}, the complexes $\mathbf{R}\pi_*\mathbf{R}\Hhom((E_i)_{\alpha,\beta,\gamma},(E_j)_{\alpha,\beta,\gamma})$ are perfect $\Bbbk[\epsilon]$-complexes. Thus the strong exceptionality follows from the semicontinuity and base change theorem on $\Spec(\Bbbk[\epsilon])$.
 To show (\ref{deform-bundle2-1}), it suffices to notice that, by (\ref{deform-alg0}) and (\ref{eq13}), the product of $a+\epsilon c(a)_i$ and $a'+\epsilon c(a')_i$ in $A_{u_{\alpha,\beta,\gamma}}$ is $a a'+\epsilon c(aa')_i$, as wanted.
\end{proof}

\begin{remark}
I do not address the problem of  fullness of the exceptional collections $\{(E_j)_{\alpha,\beta,\gamma}\}$ in this paper.
  I expect that a theory of noncommutative Grothendieck duality will show the fullness along the line of the proof of theorem \ref{def-FEC1}. 
\end{remark}

\section{A comparison theorem}
In this section we assume that $\Bbbk$ is a field of characteristic zero, and $X$ a smooth proper scheme over $\Bbbk$, $(E_1,\cdots,E_m)$  an strong full exceptional collection of \emph{vector bundles} on $X$, and denote
\[
E=\bigoplus_{i=1}^m E_i.
\]
Thus $E$ is a tilting object of $\mathrm{D}^\mathrm{b}(X)$. Denote 
\[
 A=\Hom_{\mathcal{O}_X}(E,E).
 \] 
 Our goal is to show that the assignment
\[
(\alpha,\beta,\gamma)\in H^2(X,\mathcal{O}_X)\oplus H^1(X,T_X)\oplus H^0(,\wedge^2 T_X) \mapsto u_{\alpha,\beta,\gamma} \in HH^2(A)
\]
coincides with the composition
\[
H^2(X,\mathcal{O}_X)\oplus H^1(X,T_X)\oplus H^0(,\wedge^2 T_X)\cong HH^2(X)\cong HH^2(A).
\]
First recall that we have  $\alpha=0$, by the following well-known fact.
\begin{lemma}
For a smooth proper scheme $X$  over a field of characteristic zero, if $\mathrm{D}^\mathrm{b}(X)$ has a full exceptional collection, then 
$H^2(X,\mathcal{O}_X)=0$.
\end{lemma}
\begin{proof} Since the characteristic is zero, one has the HKR isomorphism (\cite{Swan96}, \cite{Yeku02} or \cite{Cal05})
\[
HH_i(X)\cong \bigoplus_{q-p=i}H^p(X,\Omega^q).
\]
It sufficies to show that under the assumption of  existence of a full exceptional collection, one has $HH_i(X)=0$ for $i>0$. This is well-known. One way (in the spirit of this paper) to see this, at least in the case that a strong full exceptional collection exist, is via the isomorphism $HH_i(X)=HH_i(A)$, and use the theorem of \cite{Cib86} which says that the higher Hochschild homology of, an algebra associated to an acyclic quiver with relations, is zero. For the general case (there exists a full exceptional collection which is not necessarily strong), one notices that Cibils' theorem can be easily generalized to the case of acyclic \emph{dg-quivers} with relations, so we can apply the main theorem of \cite{Bod15} to conclude. 
\end{proof}  

To state our comparison theorem, we need to recall the definition of  the canonical isomorphisms 
\[
\bigoplus_{i=0}^{n} H^{i}(X,\wedge^{n-i}T_X)\cong HH^n(X)\cong HH^n(A).
\]

\subsection{HKR isomorphisms}
By \cite[section 1]{Swan96} there is a spectral sequence
\begin{multline}\label{ss1}
E_2^{p,q}=H^p(X,\ext^q_{\mathcal{O}_{X\times X}}(\mathcal{O}_\Delta,\mathcal{O}_\Delta))=H^p(X\times X,\ext^q_{\mathcal{O}_{X\times X}}(\mathcal{O}_\Delta,\mathcal{O}_\Delta))\\
\Rightarrow \Ext^{p+q}_{\mathcal{O}_{X\times X}}(\mathcal{O}_\Delta,\mathcal{O}_\Delta).
\end{multline}

By using a theorem of \cite{GS87}, Swan showed that \cite[cor. 2.6]{Swan96} this spectral sequence degenerates, and there is moreover a Hodge-type decomposition. See also \cite{Yeku02} and \cite{Cal05}. Some details of the isomorphism $\Upsilon^n$ will be reviewed in section \ref{sec-some-canonical}.
\begin{theorem}
The spectral sequence (\ref{ss1}) degenerates at $E_2$, and there is a canonical decomposition
\begin{equation}\label{hodge1}
\Upsilon^n:\bigoplus_{i=0}^{n}H^i(X,\ext^{n-i}_{\mathcal{O}_{X\times X}}(\mathcal{O}_\Delta,\mathcal{O}_\Delta))\xrightarrow{\sim}\Ext^{n}_{\mathcal{O}_{X\times X}}(\mathcal{O}_\Delta,\mathcal{O}_\Delta).
\end{equation}
\end{theorem}

We need also  the HKR isomorphism for smooth affine algebras, due to \cite{HKR62}. Our presentation follows \cite[section 3.4]{Loday}. Let $R$ be a commutative algebra over $\Bbbk$, and 
\[
T^1_{R/\Bbbk}=\Hom_R(\Omega^1_{R/\Bbbk},R)=\mathrm{Der}_{\Bbbk}(R,R),
\]
and let
\[
T^n_{R/\Bbbk}:=\wedge^n T^1_{R/\Bbbk}
\]
be the $n$-th exterior product of $T^1_{R/\Bbbk}$. If $R$ is smooth over $\Bbbk$, $T^n_{R/\Bbbk}\cong \Hom_R(\Omega^n_{R/\Bbbk},R)$. For $f_1,...,f_n\in \mathrm{Der}_{\Bbbk}(R,R)$, define the \emph{antisymmtrization map} 
\[
\epsilon_n: \mathrm{Der}_{\Bbbk}(R,R)^{\otimes n}\rightarrow \Hom_{\Bbbk}(R^{\otimes n},R)
\]
to be
\begin{equation}\label{hkr0}
\epsilon_n(f_1\otimes...\otimes f_n)(a_1,...,a_n)=\sum_{\sigma\in S_n}\mathrm{sgn}(\sigma)(f_1(a_{\sigma(1)}),...,f_n(a_{\sigma(n)})).
\end{equation}
Then $\epsilon_n$ induces a map, still denoted by $\epsilon_n$,
\[
\epsilon_n:T^n_{R/\Bbbk}\rightarrow  \Hom_{\Bbbk}(R^{\otimes n},R)=\Hom_{R^e}(C^{\mathrm{bar}}_n(R),R).
\]

\begin{lemma}
The image of $\epsilon_n$ lies in the kernel of $\mathfrak{b}$. Thus there is an induced map
\begin{equation}\label{hkr1}
\epsilon_n:T^n_{R/\Bbbk}\rightarrow  HH^n(R).
\end{equation}
\end{lemma}

\begin{theorem}\label{hkr2}
If $R$ is a smooth $\Bbbk$-algebra, the map (\ref{hkr1}) is an isomorphism.
\end{theorem}

\begin{corollary}\label{hkr3}
There are canonical isomorphisms 
\[
\epsilon_i:\wedge^i T_{X}\xrightarrow{\sim}\ext^{i}_{\mathcal{O}_{X\times X}}(\mathcal{O}_\Delta,\mathcal{O}_\Delta)
\]
 and
\begin{equation}\label{hkr4}
H^p(X,\wedge^q T_{X})\xrightarrow{\sim}H^p(X,\ext^{q}_{\mathcal{O}_{X\times X}}(\mathcal{O}_\Delta,\mathcal{O}_\Delta)).
\end{equation}
\end{corollary}
\pqed
Denote the isomorphism (\ref{hkr4}) by $\mathfrak{E}^{p,q}$, and 
\[
\mathfrak{E}^n=\bigoplus_{p=0}^{n}\mathfrak{E}^{p,n-p}: \bigoplus_{p=0}^{n} H^p(X,\wedge^{n-p} T_{X})\xrightarrow{\sim}\bigoplus_{p=0}^{n} H^p(X,\ext^{n-p}_{\mathcal{O}_{X\times X}}(\mathcal{O}_\Delta,\mathcal{O}_\Delta)).
\]

\subsection{Statement of the theorem}
The part 1 of  the following theorem is \cite[6.2]{Bon89}, and the part 2 is  \cite[3.4, 3.5]{BH13}. Recall that $F\boxtimes G:=q_1^* F\otimes q_2^*G$, where $q_1$ and $q_2$ are the two projections from $X\times X$ to $X$.

\begin{theorem}\label{equiv0}
Let $Y$ be a smooth proper scheme over $\Bbbk$, and $E$ a tilting object of $\mathrm{D}^\mathrm{b}(Y)$, and $A=\Hom_{\mathcal{O}_X}(E,E)$, $A^e=A^{\mathrm{op}}\otimes_{\Bbbk} A$.
\begin{enumerate}
  \item The functor
\begin{equation}\label{equiv1}
\Psi=\mathbf{R}\Hom_{\mathcal{O}_X}(E,\cdot) : \mathrm{D}^{\mathrm{b}}(X)\rightarrow \mathrm{D}^{\mathrm{b}}(A)
\end{equation}
is an equivalence. Moreover, $\Psi(E)=A$.
  \item The functor
\begin{equation}\label{equiv2}
\Psi^e=\mathbf{R}\Hom_{\mathcal{O}_{X\times X}}(E^\vee\boxtimes^\mathbf{L}E,\cdot) : \mathrm{D}^{\mathrm{b}}(X\times X)\rightarrow \mathrm{D}^{\mathrm{b}}(A^e)
\end{equation}
is an equivalence. Moreover, $\Psi^e(E^\vee\boxtimes^{\mathbf{L}}E)=A^e$,  $\Psi^e(\mathcal{O}_\Delta)=A$.
\end{enumerate}
\end{theorem}
\pqed
 Thus $\Psi^e$ induces an isomorphism
\begin{eqnarray}\label{equiv3}
\Xi_E^n:\Ext^{n}_{\mathcal{O}_{X\times X}}(\mathcal{O}_\Delta,\mathcal{O}_\Delta)
\xrightarrow{\sim}\Ext^{n}_{A^e}(A,A)= HH^n(A).
\end{eqnarray}

\begin{definition}
We denote  the composition of the isomorphisms (\ref{hkr4}), (\ref{hodge1}) and (\ref{equiv3}) by
\begin{equation}\label{hkr5}
\Phi^n= \Xi_E^n\circ\Upsilon^n\circ\mathfrak{E}^n: \bigoplus_{p=0}^n H^p(X,\wedge^{n-p} T_{X})\xrightarrow{\sim}
HH^n(A).
\end{equation}
\end{definition}
Now we are ready to state our theorem.

\begin{theorem}\label{comparison0}
For $\beta\in H^1(X,T_X)$, $\gamma\in H^0(X,\wedge^2 T_X)$, 
\begin{equation}
\Phi^2(\beta,\gamma)=u_{0,\beta,\gamma},
\end{equation}
where $u_{\alpha,\beta,\gamma}$ is given by the construction \ref{constr-1}.
\end{theorem}
The  proof of this theorem occupies  the rest of this section. The following corollary is a direct consequence of theorem \ref{comparison0}.

\begin{corollary}\label{comparison0-cor}
A first order noncommutative deformation of $X$  is trivial, if it induces a trivial deformation of $A$.
\end{corollary}

\begin{remark}
This corollary is also a very special case  (i.e., smooth proper varieties with a strong full exceptional collection of vector bundles) of a consequence of \cite[prop. 5.1]{AT08}. 
\end{remark}

\subsection{Morita equivalence and $\lambda$-decomposition}\label{sec-morita}
In this subsection we review the Morita equivalence and the  $\lambda$-decomposition of 
 Hochschild cohomology, and make some observations that we will need later. Our references are \cite{Loday}, \cite{GS87}.
 Let $B$ be a $\Bbbk$-algebra, $M_r(B)$ the  $\Bbbk$-algebra of matrices of rank $r$ with coefficients in $B$. The $(i,j)$-entry of a matrix $G$ is denoted by $G_{ij}$.
\begin{definition}
 For $f\in C^0(B,B)=B$, define $\mathrm{cotr}(f)=f\cdot \mathrm{id} \in M_r(B)=C^0(M_r(B),M_r(B))$. For $n\geq 1$ and $f\in C^n(B,B)$, define $\mathrm{cotr}(f)$ to be the element of $C^n(M_r(B),M_r(B))$ such that for $\alpha^1,...,\alpha^n\in M_r(B)$, 
\begin{equation}\label{cotr1}
\mathrm{cotr}(f)(\alpha^1,...,\alpha^n)_{ij}=\sum_{i_2,...,i_n} f(\alpha^1_{i i_2},\alpha^2_{i_2 i_3},...,\alpha^n_{i_n j})
\end{equation}
where the sum is over all possible indices $1\leq i_2,...,i_n\leq r$. 
\end{definition}

For a given positive integer $r$, let $E_{i,j}(a)$ be the $r\times r$ matrix whose entry at $(i,j)$ is $a$, and all the other entry is zero. The inclusion map 
\[
\mathrm{inc}^*: C^n(M_r(B),M_r(B))\rightarrow C^n(B,B)
\]
is defined by
\begin{equation}\label{inc1}
\mathrm{inc}^*(F)(a_1,...,a_n)=F\big(E_{11}(a_1),...,E_{11}(a_n)\big)_{11}
\end{equation}
for a $\Bbbk$-linear map $F:M_r(B)^{\otimes n}\rightarrow M_r(B)$. It is easily seen that $\mathrm{cotr}$ and $\mathrm{inc}^*$ are chain maps. The following theorem is given in \cite[1.5.6]{Loday} without a proof. For the readers' convenience I write a proof by mimicking the proof of the homological version \cite[1.2.4]{Loday}.

\begin{theorem}\label{morita1}
For positive integers $n$, $\mathrm{cotr}$ and $\mathrm{inc}^*$ induce  isomorphisms of Hochschild cohomology
\[
 \mathrm{cotr}: HH^n(B)\xrightarrow{\sim} HH^n(M_r(B)),\ \mathrm{inc}^*: HH^n(M_r(B))\xrightarrow{\sim} HH^n(B),
 \] 
 and which are inverse to each other.
\end{theorem}
\begin{proof} It is obvious that $\mathrm{inc}^*\circ \mathrm{cotr}=\mathrm{id}$. It suffices to show that $\mathrm{cotr}\circ \mathrm{inc}^*$ is homotopic to $\mathrm{id}$. By definition,
\begin{eqnarray}
(\mathrm{cotr}\circ \mathrm{inc}^*)(F)(\alpha^1,...,\alpha^n)_{ij}
=
\sum_{i_2,...,i_n}F\big(E_{11}(\alpha^1_{i i_2}),...,E_{11}(\alpha^n_{i_n j})\big)_{11},
\end{eqnarray}
i.e.,
\begin{eqnarray}
(\mathrm{cotr}\circ \mathrm{inc}^*)(F)(\alpha^1,...,\alpha^n)
=
\sum_{i,i_2,...,i_n,j}E_{i,1}F\big(E_{11}(\alpha^1_{i i_2}),...,E_{11}(\alpha^n_{i_n j})\big)E_{1,j}.
\end{eqnarray}
For $i=1,...,n-1$, define
\[
h_i: \Hom_{\Bbbk}(M_r(B)^{\otimes n},M_r(B))\rightarrow \Hom_{\Bbbk}(M_r(B)^{\otimes n-1},M_r(B))
\]
by
\begin{multline}
h_i(F)(\alpha^1,...,\alpha^{n-1})=\sum_{k,m,...,p,q} E_{k1}(1) F\big( E_{11}(\alpha_{km}^1)\otimes\\
...\otimes E_{11}(\alpha_{pq}^i)\otimes E_{1q}(1)\otimes \alpha^{i+1}\otimes...\otimes \alpha^{n-1}\big).
\end{multline}
Set
\begin{equation}
h_0(F)(\alpha^1,...,\alpha^{n-1})=\sum_{k} E_{k1}(1) F\big( E_{1k}(1)\otimes \alpha^{1}\otimes...\otimes \alpha^{n-1}\big).
\end{equation}
Set temporarily (in this proof),
\begin{eqnarray*}
\mathfrak{b}_0(F)(\alpha^1,...,\alpha^{n+1})&:=&\alpha^1 F(\alpha^2,...,\alpha^{n+1}),\\
\mathfrak{b}_i(F)(\alpha^1,...,\alpha^{n+1})&:=&F(\alpha^1,...,\alpha^{i}\alpha^{i+1},... \alpha^{n+1})),\ \mbox{for}\ 1\leq i\leq n,\\
\mathfrak{b}_n(F)(\alpha^1,...,\alpha^{n+1})&:=&F(\alpha^1,...,\alpha^{n})\alpha^{n+1}
\end{eqnarray*}
such that
\[
\mathfrak{b}(F)=\sum_{i=0}^{n+1}(-1)^i \mathfrak{b}_i(F).
\]
Thus 
\begin{eqnarray*}
h_0 \mathfrak{b}_0=\mathrm{id},\ h_n \mathfrak{b}_{n+1}= \mathrm{cotr}\circ \mathrm{inc}^*.
\end{eqnarray*}
One can verify by some tedious computations the \emph{pre-cosimplicial homotopy} relations
\begin{equation}
\begin{cases}
h_i \mathfrak{b}_j=\mathfrak{b}_j h_{i-1},& 0\leq j<i\leq n,\\
h_i \mathfrak{b}_i=h_{i-1}\mathfrak{b}_i,& 0<i\leq n,\\
h_i \mathfrak{b}_j=\mathfrak{b}_{j-1}h_i,& 1\leq i+1<j\leq n+1, \\
\end{cases}
\end{equation}
which imply
\begin{eqnarray*}
\big(\sum_{i=0}^{n}(-1)^i h_i\big)\circ \big(\sum_{j=0}^{n+1}(-1)^j \mathfrak{b}_j\big)+\big(\sum_{j=0}^{n}(-1)^j \mathfrak{b}_j\big)\circ\big(\sum_{i=0}^{n-1}(-1)^i h_i\big)=h_0 \mathfrak{b}_0-h_n \mathfrak{b}_{n+1},
\end{eqnarray*}
and therefore give the homotopy from $\mathrm{id}$ to $\mathrm{cotr}\circ \mathrm{inc}^*$.
\end{proof}

Now let $L$ be free $B$ module of rank $r$, and $M=\mathrm{End}_{B}(L)$. Choosing a $B$-basis of $L$, we obtain an isomorphism $M\cong M_r(B)$, and thus the isomorphisms of Hochschild cohomology.
\begin{lemma}\label{indep-1}
The induced isomorphisms
\begin{equation}\label{morita1.1}
\mathrm{cotr}: HH^n(B)\xrightarrow{\sim} HH^n(M),\ \mathrm{inc}^*: HH^n(M)\xrightarrow{\sim} HH^n(B)
\end{equation}
are independent of the choice of $B$-basis of $L$.
\end{lemma}
\begin{proof} The conclusion is a direct consequence of a more general Morita equivalence, see e.g. \cite[1.2.5]{Loday}. Recall that two $\Bbbk$-algebras $R$ and $S$ are Morita equivalent if there are $R$-$S$-bimodule $P$ and $S$-$R$-bimodule $Q$ and an isomorphism of $R$-bimodules $u: P\otimes_S Q\cong R$, and an isomorphism of $S$-bimodules $v: Q\otimes_R P\cong S$. Moreover, such $u$ and $v$ induce a natural isomorphism 
\begin{equation}\label{indep-2}
HH^*(R,R)\cong HH^*(S,Q\otimes_R P).
\end{equation}
Consider $R=B$, $S=M=\mathrm{End}_{B}(L)$, and take $P=L$, $Q=L^\vee=\Hom_B(L,B)$. Then there are an obvious isomorphism of $B$-bimodules $u: L\otimes_M L^\vee\cong B$ given by the pairing, and an obvious isomorphism of $M$-bimodules $L^\vee\otimes_B L\cong M$, and notice that $u$ and $v$ do not depend on the choice of basis of $L$.

A proof of Hochschild homology version of (\ref{indep-2}) is given in \cite[1.2.7]{Loday}, and one easily checks the construction of the isomorphism coincides with the isomorphism of  $\mathrm{inc}_*: H_*(B,B)\cong H_*(M_r(B),M_r(B))$ after choosing a basis of $M$, which implies the independence of basis for Hochschild homology. The case for Hochschild cohomology is similar, as the proof of theorem \ref{morita1}, and we omit it.
\end{proof}

Next we recall the Hodge-type decomposition \cite{GS87}, which is called $\lambda$-decomposition in \cite[\S 4.5]{Loday}. Denote by $S_n$ the symmetric group of $n$ elements. For the definition of the elements $e_n^{(i)}$ of $\mathbb{Q}(S_n)$, and the proof of the following proposition, see e.g. \cite[4.5.2, 4.5.3, 4.5.7]{Loday}.
\begin{proposition}\label{euler-idem-1}
The elements $e_n^{(1)},...,e_n^{(n)}$  satisfy
\begin{enumerate}
  \item[(i)] $\mathrm{id}=e_n^{(1)}+...+e_n^{(n)}$.
  \item[(ii)] $e_n^{(i)}e_n^{(j)}=0$ for $1\leq i\neq j\leq n$, and $e_n^{(i)}e_n^{(i)}=e_n^{(i)}$ for $1\leq i\leq n$.
  \item[(iii)] In particular, 
\[
  e_2^{(1)}=\frac{1}{2}\big(\mathrm{id}+(12)\big),
\]
and 
\[
e_n^{(n)}=\frac{1}{n!}\sum_{\sigma\in S_n}\mathrm{sgn}(\sigma)\sigma=\frac{1}{n!}\epsilon_n.
\]
\end{enumerate}
\end{proposition}
\begin{definition}
Let $\Bbbk$ be a field of characteristic zero, $B$  a $\Bbbk$-algebra. 
For $\sigma\in S_n$, and $f\in C^n(A,A)$, define
\[
\sigma(f)(a_1,...,a_n)=f(a_{\sigma(1)},...,a_{\sigma(n)}),
\]
and extend the action linearly to $\mathbb{Q}(S_n)$.
\end{definition}

\begin{theorem}\label{hodge-decomp-1}\cite[4.5.10, 4.5.12]{Loday}
Let $\Bbbk$ be a field of characteristic zero, and $B$ a \emph{commutative} $\Bbbk$-algebra. 
\begin{enumerate}
  \item[(i)]
  \begin{equation}\label{hodge-decomp-2}
  \mathfrak{b}\circ e_{n}^{(i)}=e_{n+1}^{(i)}\circ\mathfrak{b}.
  \end{equation}
  \item[(ii)] The idempotents $e_n^{(i)}$ split the Hochschild cochain complex $C^*(B,B)$ into a direct sum 
  \begin{equation}\label{hodge-decomp-3}
  C^*(B,B)=\bigoplus_{i\geq 0}C^{*}_{(i)}(B,B),
  \end{equation}
  where $C^{n}_{(i)}(B,B)=0$ for $i>n$. This induces a direct decomposition of Hochschild cohomology
  \begin{equation}\label{hodge-decomp-4}
  HH^n(B)=\sum_{i=1}^{n}HH^{n}_{(i)}(B).
  \end{equation}
  \item[(iii)] If $B$ is smooth, then $HH^{n}_{(i)}(B)=0$ for $i<n$ and the isomorphism (\ref{hkr4}) reduces to 
    $\epsilon_n: T^n_{B/\Bbbk}\cong HH^{n}_{(n)}(B)$.
\end{enumerate}
\end{theorem}

Now let $B$ be a commutative $\Bbbk$-algebra, and $L$ a free $B$-module of rank $r$, $M=\mathrm{End}_B(L)$. Then the Morita equivalence and the $\lambda$-decomposition induce a decomposition
\begin{equation}\label{hodge-decomp-5}
HH^n(M)\cong \bigoplus_{i=1}^{n}HH^n_{(i)}(M)
\end{equation}
such that $HH^n_{(i)}(M)= HH^n_{(i)}(B)$ via the isomorphism (\ref{morita1.1}). However, to my knowledge, we do not have an $\lambda$-decomposition on the cochain level  $C^*(M,M)$. Fortunately, the following naive characterization is enough for our use.
\begin{lemma}\label{hodge-decomp-6}
Let $F\in Z^n(M,M)$, i.e., $F$ a Hochschild $n$-cocycle of $M$. Then the class of $F$ lies in $HH^n_{(i)}(M)$ if, after choosing a basis of $L$ and identify $M$ to $M_r(B)$, 
\begin{equation}\label{hodge-decomp-7}
e_n^{(i)}(F)\big(E_{11}(b_1),...,E_{11}(b_n)\big)_{11}=F\big(E_{11}(b_1),...,E_{11}(b_n)\big)_{11}
\end{equation}
for all  $b_1,...,b_n\in B$.
\end{lemma}
\begin{proof} By the definition (\ref{inc1}) of $\mathrm{inc}^*$, (\ref{hodge-decomp-7}) implies $e_n^{(i)}\mathrm{inc}^*(F)=\mathrm{inc}^*(F)$. 
\end{proof}

\subsection{A bar resolution}
For $i=1,2$, the homomorphisms of $\mathcal{O}_X$-modules (regarding $A$ as a constant sheaf)
\[
E^\vee\otimes_{\Bbbk}A=E^\vee\otimes_{\Bbbk} \Hom_{\mathcal{O}_X}(E,E)\rightarrow E^\vee
\]
and 
\[
E\otimes_{\Bbbk}A=\Hom_{\mathcal{O}_X}(E,E)\otimes_{\Bbbk} E\rightarrow E
\]
induce  homomorphisms of $\mathcal{O}_{X\times X}$-modules, respectively,
\[
\sigma:
q_i^* E^\vee\otimes_{\Bbbk}A\rightarrow q_i^* E^\vee
\]
and
\[
\tau:A\otimes_{\Bbbk}q_i^* E\rightarrow q_i^* E.
\]
Set
\[
\mathcal{C}^\mathrm{bar}_{i}=E^\vee\boxtimes E\otimes_{\Bbbk}A^{\otimes_{\Bbbk}i},
\]
and define $\mathfrak{b}'_i: \mathcal{C}^\mathrm{bar}_{i}\rightarrow \mathcal{C}^\mathrm{bar}_{i-1}$ by
\begin{multline}\label{bar1}
\mathfrak{b}'_i(x,y,a_1,...,a_i)=(\sigma(x\otimes a_1),y,a_2,...,a_i)\\
+\sum_{j=1}^{i-1}(-1)^j (x,y,a_1,...,a_{j-1},a_{j}a_{j+1},a_{j+2},...,a_{i})+(-1)^i (x,\tau(a_{i}\otimes y),a_1,...,a_{i-1}).
\end{multline}
We define an augmentation map  $\mu:E^\vee\boxtimes E\rightarrow \mathcal{O}_{\Delta}$ by adjointness, via $q_2^* E\rightarrow q_1^* E\otimes \mathcal{O}_{\Delta}\cong q_2^* E|_{\Delta}$, or equivalently, via $q_1^* E^\vee\rightarrow q_2^* E^\vee\otimes \mathcal{O}_{\Delta}\cong q_1^* E^\vee|_{\Delta}$. 
\begin{lemma}\label{morita2}
There is an quasi-isomorphisms of complex of coherent sheaves
\begin{equation}\label{bar2}
\mathcal{C}^\mathrm{bar}_{\bullet}(E)\rightarrow \mathcal{O}_{\Delta}
\end{equation}
on $X\times X$, which is transformed by $\Psi^e$ to the bar resolution $C_{\bullet}^{\mathrm{bar}}(A)$ of $A$.
\end{lemma}
\begin{proof} It is easy to check that $\mathfrak{b}'_{i}\circ \mathfrak{b}'_{i+1}=0$ and $\mu\circ \mathfrak{b}'_1=0$. By the definition of the bar resolution of a $\Bbbk$-algebra \cite[1.1.11]{Loday}, one easily sees that{} $\Psi^e(\mathcal{C}^\mathrm{bar}_{\bullet}(E))=C_{\bullet}^{\mathrm{bar}}(A)$. By \cite[1.1.12]{Loday} and theorem \ref{equiv0}, $\mathcal{C}^\mathrm{bar}_{\bullet}(E)$ is a resolution of $\mathcal{O}_{\Delta}$.
\end{proof}

For an open  subset $U_i$ of $X$, regarded as an open subset of the diagonal $\Delta_X\subset X\times X$, by theorem \ref{hkr2} and lemma \ref{morita2} we have
\begin{equation}\label{morita0}
\bigwedge^q T_{U_i} \xrightarrow{\sim} \ext^q_{\mathcal{O}_{U_i\times U_i}}(\mathcal{O}_{\Delta_ {U_i}},\mathcal{O}_{\Delta_{U_i}})
\xrightarrow{\sim} \mathscr{H}^q(\Hhom^\bullet_{\mathcal{O}_{U_i\times U_i}}(\mathcal{C}^\mathrm{bar}_{\bullet}(E)|_{U_i\times U_i},\mathcal{O}_{\Delta_ {U_i}})).
\end{equation}
It will turn out to be more convenient to work with a Hochschild cochain complex rather than the bar resolution. Let us introduce first the Hochschild cochain complex for a module over a sheaf of algebras.
\begin{definition}\label{sheaf-hoch-1}
For a sheaf $\mathcal{A}$ of $\Bbbk$-algebras over a topological space $Y$, let $\mathcal{A}^{\otimes i}$ be the sheaf associated to the presheaf $U\mapsto \Gamma(U,A)^{\otimes_{\Bbbk} i}$, which is still a sheaf of $\Bbbk$-algebras. For a sheaf $\mathcal{M}$ of $\mathcal{A}$-bimodules, we define the Hochschild cochain complex $\mathcal{C}^{\bullet}(\mathcal{A},\mathcal{M})$  of sheaves of $\Bbbk$-vector spaces on $X$ by 
\[
\mathcal{C}^{k}(\mathcal{A},\mathcal{M})=\Hhom_{\Bbbk}(\mathcal{A}^{\otimes k}, \mathcal{M})
\]
with the differentials given by
\begin{multline}\label{hoch-diff2}
\mathfrak{b}(f)(a_1,\cdots,a_{k+1})=a_1 f(a_2,\cdots,a_{k+1})\\
+\sum_{1\leq i\leq k}(-1)^i f(a_1,\cdots,a_{i}a_{i+1},\cdots,a_{k+1})+(-1)^{k+1}f(a_1,\cdots,a_k)a_{k+1}.
\end{multline}
\end{definition}
When $Y$ is a point, $\mathcal{C}^{\bullet}(\mathcal{A},\mathcal{M})$ is the ordinary Hochschild cochain complex which computes the Hochschild cohomology $HH^\bullet(A,M)$ \cite[1.5.1]{Loday}. 

Return to the setup at the beginning of this section. We denote the constant sheaf of $\Bbbk$-algebras associated to  $A=\Hom_{\mathcal{O}_X}(E,E)$ still by $A$. Then $E^\vee\otimes_{\mathcal{O}_X}E\cong \Hhom_{\mathcal{O}_X}(E,E)$ is a sheaf of  $A$-bimodules in an obvious way. The corresponding Hochschild cochain complex is denoted by $\mathcal{C}^\bullet(A,E^\vee\otimes E)$. There is an obvious homomorphism between two Hochschild cochain complex
\begin{equation}\label{hoch-cochain1}
\mathcal{C}^\bullet(E^\vee\otimes E,E^\vee\otimes E)\rightarrow \mathcal{C}^\bullet(A,E^\vee\otimes E)
\end{equation}
induced by the homomorphism of sheaves of algebras $A\rightarrow E^\vee\otimes E$ given by restrictions of global endomorphisms of $E$.

Let us recall the \v{C}ech complex associated to a complex of sheaves.
Let $\mathcal{U}=\{U_i\}_{i\in I}$ be an affine open covering of $X$. For a complex of sheaves $(L^\bullet, \partial_L)$, the associated \v{C}ech double complex is $\check{C}^p(\mathcal{U},L^q)$ with $\delta: \check{C}^p(\mathcal{U},L^q)\rightarrow \check{C}^{p+1}(\mathcal{U},L^q)$ the \v{C}ech coboundary map, $\partial_L: \check{C}^p(\mathcal{U},L^q)\rightarrow \check{C}^p(\mathcal{U},L^{q+1})$ the map induced by $\partial_L$. The differential of the associated simple complex is 
\begin{equation}\label{cech-diff-0}
d=\delta+(-1)^p \partial_L.
\end{equation}

\begin{lemma}\label{hoch-cochain2}
The cohomology of (the simple complex associated to) the double complex
\[
\check{C}^\bullet\big(\mathcal{U},\mathcal{C}^{\bullet}(A,E^\vee\otimes E)\big)
\]
computes the Hochschild cohomology $HH^\bullet(X)$.
\end{lemma}
\begin{proof} For every  integer $m\geq 0$, there is an identity of sheaves on $U_i$ 
\begin{equation}\label{morita3.0}
\Hhom_{\mathcal{O}_{U_i\times U_i}}(\mathcal{C}^\mathrm{bar}_{m}(E)|_{U_i\times U_i},\mathcal{O}_{\Delta_{U_i}})=\Hhom_{\Bbbk}(A^{\otimes m},(E^\vee\otimes E)|_{U_i}).
\end{equation}
One easily checks, by comparing (\ref{bar1}) and (\ref{hoch-diff2}), that (\ref{morita3.0}) induces an isomorphism
\begin{equation}\label{morita3}
\Hhom^\bullet_{\mathcal{O}_{U_i\times U_i}}(\mathcal{C}^\mathrm{bar}_{\bullet}(E)|_{U_i\times U_i},\mathcal{O}_{\Delta_ {U_i}})\cong \mathcal{C}^{\bullet}(A,E^\vee\otimes E)|_{U_i}.
\end{equation}
By  the isomorphisms (\ref{morita0}) and (\ref{morita3}), $HH^\bullet(X)$ is isomorphic to the hypercohomology of $\mathcal{C}^{\bullet}(A,E^\vee\otimes E)$. 
Since $\Hhom_{\Bbbk}(A^{\otimes k},E^\vee\otimes E)$ is coherent, the conclusion follows from e.g. \cite[theorem 2.8.1]{ET}.
\end{proof}

We denote the resulting isomorphism by
\[
\mathfrak{Q}^n: H^n\big(\check{C}^\bullet(\mathcal{U},\mathcal{C}^{\bullet}(A,E^\vee\otimes E))\big)\xrightarrow{\sim}\Ext^n_{\mathcal{O}_{X\times X}}(\Delta_X,\Delta_X).
\]
By (\ref{morita0}) and (\ref{morita3}), there are also isomorphisms
\[
\mathfrak{B}^{p,q}:\check{H}^p\big(\mathcal{U},\ext^q_{\mathcal{O}_{X\times X}}(\Delta_X,\Delta_X)\big)\xrightarrow{\sim} \check{H}^p\big(\mathcal{U},\mathscr{H}^q(\mathcal{C}(A,E^{\vee}\otimes E))\big),
\]
and we denote 
$\mathfrak{B}^n=\bigoplus_{p+q=n}\mathfrak{B}^{p,q}$.

\subsection{Some canonical isomorphisms}\label{sec-some-canonical}
In this subsection we prove some canonical isomorphism together commutativity, for preparing the explicit construction of $\Phi^n$. 

 According to definition \ref{sheaf-hoch-1}, let $\mathcal{C}^{\bullet}(\mathcal{O}_X,\mathcal{O}_X)$ be the Hochschild cochain complex associated to the sheaf of $\Bbbk$-algebras $\mathcal{O}_X$. 

\begin{lemma}\label{hkr6}
The cohomology sheaf $\mathscr{H}^q(\mathcal{C}^{\bullet}(\mathcal{O}_X,\mathcal{O}_X))$ is canonically isomorphic to $T_X^q=\wedge^q T_X$.
\end{lemma}
\begin{proof} This follows from theorem \ref{hkr2}, see also \cite[lemma 2.4 (3)]{Swan96}.
\end{proof}

\begin{corollary}
Let $\mathcal{U}$ be an affine open covering of $X$, then
\begin{equation}\label{cech-1}
H^p(X,\mathscr{H}^q(\mathcal{C}^{\bullet}(\mathcal{O}_X,\mathcal{O}_X)))\cong \check{H}^p(\mathcal{U},\mathscr{H}^q(\mathcal{C}^{\bullet}(\mathcal{O}_X,\mathcal{O}_X))).
\end{equation}
\end{corollary}
\begin{proof} By lemma \ref{hkr6}, $\mathscr{H}^q(\mathcal{C}^{\bullet}(\mathcal{O}_X,\mathcal{O}_X))$ is a coherent sheaf, thus the conclusion follows.
\end{proof}

\begin{lemma}
There are quasi-isomorphisms
\begin{equation}\label{morita5}
\mathcal{C}^{\bullet}(\mathcal{O}_X,\mathcal{O}_X)\rightarrow \mathcal{C}^\bullet(E^\vee\otimes E,E^\vee\otimes E)\rightarrow \mathcal{C}^\bullet(A,E^\vee\otimes E). 
\end{equation}
\end{lemma}
\begin{proof} The first map is induced by the natural maps $E^\vee\otimes E\rightarrow \mathcal{O}_X$ and $\mathcal{O}_X\rightarrow E^\vee\otimes E$. By theorem \ref{morita1} and lemma \ref{indep-1}, the first map is a quasi-isomorphism. The second map is (\ref{hoch-cochain1}).
Then by (\ref{morita0}), (\ref{morita3}) and lemma \ref{hkr6}, the second map is also a quasi-isomorphism.
\end{proof}

\begin{lemma}\label{hodge-decomp-8.0}
There is a canonical isomorphism
\begin{equation}\label{hodge-decomp-8}
\bigoplus_{p+q=n} H^p\big(X,\mathscr{H}^q(\mathcal{C}^{\bullet}(\mathcal{O}_X,\mathcal{O}_X))\big)\cong
 \mathbb{H}^n\big(X,\mathcal{C}^{\bullet}(\mathcal{O}_X,\mathcal{O}_X)\big).
\end{equation}
\end{lemma}
\begin{proof} There is a spectral sequence 
\[
E_2^{p,q}= H^p\big(X,\mathscr{H}^q(\mathcal{C}^{\bullet}(\mathcal{O}_X,\mathcal{O}_X))\big)\Rightarrow \mathbb{H}^{p+q}\big(X,\mathcal{C}^{\bullet}(\mathcal{O}_X,\mathcal{O}_X)\big).
\]
The $\lambda$-decomposition $\mathcal{C}^{q}(\mathcal{O}_X,\mathcal{O}_X)=\bigoplus_{i=0}^{q}\mathcal{C}_{(i)}^{q}(\mathcal{O}_X,\mathcal{O}_X)$ induces the degeneration of the spectral sequence, and moreover the decomposition (\ref{hodge-decomp-8}), see the argument of \cite[cor. 2.6]{Swan96}.
\end{proof}

\begin{notations}
For a given affine open covering $\mathcal{U}$ of $X$, denote by $\eta$ the canonical isomophism
\[
\eta: \bigoplus_{p+q=n} \check{H}^p\big(\mathcal{U},\mathscr{H}^q(\mathcal{C}^{\bullet}(\mathcal{O}_X,\mathcal{O}_X))\big)\xrightarrow{\sim} \mathbb{H}^{n}\big(X,\mathcal{C}^{\bullet}(\mathcal{O}_X,\mathcal{O}_X)\big)
\]
induced by (\ref{cech-1}) and (\ref{hodge-decomp-8}), and denote $\xi$ and $\zeta$ the isomorphisms
\[
\xi: \bigoplus_{p+q=n} \check{H}^p\big(\mathcal{U},\mathscr{H}^q(\mathcal{C}^{\bullet}(\mathcal{O}_X,\mathcal{O}_X))\big)\xrightarrow{\sim} \bigoplus_{p+q=n} \check{H}^p\big(\mathcal{U},\mathscr{H}^q(\mathcal{C}^{\bullet}(A,E^{\vee}\otimes E))\big)
\]
and 
\[
\zeta:\mathbb{H}^{p+q}\big(X,\mathcal{C}^{\bullet}(\mathcal{O}_X,\mathcal{O}_X)\big)
\xrightarrow{\sim}H^n\big(\check{C}^\bullet(\mathcal{U},\mathcal{C}^{\bullet}(A,E^\vee\otimes E))\big)
\]
the isomorphisms induced by  (\ref{morita5}).
\end{notations}

\begin{lemma}\label{commutativity-0}
There are natural isomorphisms $\rho$ and $\sigma$ such that the following diagrams
\begin{equation}\label{diag-1}
\xymatrix{
  \bigoplus_{p+q=n} \check{H}^p\big(\mathcal{U},\ext^q_{\mathcal{O}_{X\times X}}(\Delta_X,\Delta_X)\big) \ar[d]^{\Upsilon^n}_{\wr}  
  & \bigoplus_{p+q=n} \check{H}^p\big(\mathcal{U},\mathscr{H}^q(\mathcal{C}^{\bullet}(\mathcal{O}_X,\mathcal{O}_X))\big) \ar[d]_{\wr}^{\eta} \ar[l]^<<<<<{\sim}_<<<<<{\rho}  \\
   \Ext^n_{\mathcal{O}_{X\times X}}(\Delta_X,\Delta_X)   & \mathbb{H}^{n}\big(X,\mathcal{C}^{\bullet}(\mathcal{O}_X,\mathcal{O}_X)\big)  \ar[l]^{\sim}_{\sigma}
}
\end{equation}
and 
\begin{equation}\label{diag-2}
\xymatrix{
  & \mathbb{H}^{p+q}\big(X,\mathcal{C}^{\bullet}(\mathcal{O}_X,\mathcal{O}_X)\big)  \ar[ld]^{\sim}_{\sigma} \ar[rd]_{\sim}^{\zeta}  & \\
   \Ext^n_{\mathcal{O}_{X\times X}}(\Delta_X,\Delta_X)   &&
  H^n\big(\check{C}^\bullet(\mathcal{U},\mathcal{C}^{\bullet}(A,E^\vee\otimes E))\big) \ar[ll]^>>>>>>>>>>>>>>>>>>>>>>>>>>>>{\sim}_>>>>>>>>>>>>>>>>>>>>>>>>>>>>{\mathfrak{Q}^n}  
}
\end{equation}
commute.
\end{lemma}
\begin{proof} The quasi-isomorphisms (\ref{bar2}), (\ref{morita5}) and the isomorphism (\ref{morita3}) induce canonical isomorphisms $\rho$ and $\sigma$, and the commutativity of (\ref{diag-2}). In addition, they induce an isomorphism of $E_2$-spectral sequences
\begin{equation}\label{sps-1}
E_2^{p,q}=\check{H}^p\big(\mathcal{U},\ext^q_{\mathcal{O}_{X\times X}}(\Delta_X,\Delta_X)\big)\Rightarrow \Ext^n_{\mathcal{O}_{X\times X}}(\Delta_X,\Delta_X)
\end{equation}
and
\begin{equation}\label{sps-2}
E_2^{' p,q}= \check{H}^p\big(\mathcal{U},\mathscr{H}^q(\mathcal{C}^{\bullet}(\mathcal{O}_X,\mathcal{O}_X))\big)\Rightarrow \mathbb{H}^{p+q}\big(X,\mathcal{C}^{\bullet}(\mathcal{O}_X,\mathcal{O}_X)\big).
\end{equation}
Thus the decomposition (\ref{hodge-decomp-8}) induces a decomposition $\Upsilon^{'n}$ and a commutative diagram
\begin{equation}\label{diag-1.1}
\xymatrix{
  \bigoplus_{p+q=n} \check{H}^p\big(\mathcal{U},\ext^q_{\mathcal{O}_{X\times X}}(\Delta_X,\Delta_X)\big) \ar[d]^{\widetilde{\Upsilon}^{n}}_{\wr}  
  & \bigoplus_{p+q=n} \check{H}^p\big(\mathcal{U},\mathscr{H}^q(\mathcal{C}^{\bullet}(\mathcal{O}_X,\mathcal{O}_X))\big) \ar[d]_{\wr}^{\eta} \ar[l]^<<<<<{\sim}_<<<<<{\rho}  \\
   \Ext^n_{\mathcal{O}_{X\times X}}(\Delta_X,\Delta_X)   & \mathbb{H}^{n}\big(X,\mathcal{C}^{\bullet}(\mathcal{O}_X,\mathcal{O}_X)\big)  \ar[l]^{\sim}_{\sigma}.
}
\end{equation}
It remains to show $\widetilde{\Upsilon}^{n}=\Upsilon^n$. 
 Following \cite[\S 2]{Swan96}, let $\mathscr{C}_{i}$ be the sheaf associated to the presheaf $U\mapsto C_{i}(\Gamma(U,\mathcal{O}_X))=\Gamma(U,\mathcal{O}_X)^{\otimes_{\Bbbk}i+1}$, and together with the usual Hochschild boundary map $\mathfrak{b}$, we obtain a complex of sheaves of $\mathcal{O}_X$-modules, denoted by $\mathscr{C}_{\bullet}$. Then by \cite[theorem 2.1 and 2.5]{Swan96}, there is the following commutative diagram of isomorphisms
\begin{equation}\label{diag-1.2}
\xymatrix{
  \bigoplus_{p+q=n} \check{H}^p\big(\mathcal{U},\ext^q_{\mathcal{O}_{X\times X}}(\Delta_X,\Delta_X)\big) \ar[d]^{\Upsilon^n}_{\wr}   
  & \bigoplus_{p+q=n} \check{H}^p\big(\mathcal{U},\mathscr{H}^q(\mathcal{C}^{\bullet}(\mathcal{O}_X,\mathcal{O}_X))\big) \ar[d]_{\wr}^{\eta'} \ar[l]^<<<<{\sim}_<<<<{\rho} \\
   \Ext^n_{\mathcal{O}_{X\times X}}(\Delta_X,\Delta_X)   & \bext_{\mathcal{O}_X}^n\big(\mathcal{C}_{\bullet}(\mathcal{O}_X),\mathcal{O}_X\big)  \ar[l]^{\sim}_{\sigma'}.
}
\end{equation}
In fact, \cite[theorem 2.5]{Swan96} says that there is an $E_2$-spectral sequence 
\begin{equation}\label{sps-3}
E_2^{'' p,q}= \check{H}^p\big(\mathcal{U},\mathscr{H}^q(\mathcal{C}^{\bullet}(\mathcal{O}_X,\mathcal{O}_X))\big)\Rightarrow \bext_{\mathcal{O}_X}^{p+q}\big(\mathcal{C}_{\bullet}(\mathcal{O}_X),\mathcal{O}_X\big)
\end{equation}
which is isomorphic to the spectral sequence (\ref{sps-1}), and then the decomposition $\Upsilon^n$ follows from the right one $\eta'$, which is also deduced from the $\lambda$-decomposition of $\mathcal{C}^{\bullet}(\mathcal{O}_X,\mathcal{O}_X)$.

Therefore the spectral sequences (\ref{sps-2}) and (\ref{sps-3}) are isomorphic, thus induce an isomorphism $\chi:\mathbb{H}^{n}\big(X,\mathcal{C}^{\bullet}(\mathcal{O}_X,\mathcal{O}_X)\big)\xrightarrow{\sim} \bext_{\mathcal{O}_X}^n\big(\mathcal{C}_{\bullet}(\mathcal{O}_X),\mathcal{O}_X\big)$. Then $\widetilde{\Upsilon}^{n}=\Upsilon^n$ is equivalent to the commutativity of the decompositions $\eta$ and $\eta'$:
\begin{equation}\label{diag-1.3}
\xymatrix{
 & \bigoplus_{p+q=n} \check{H}^p\big(\mathcal{U},\mathscr{H}^q(\mathcal{C}^{\bullet}(\mathcal{O}_X,\mathcal{O}_X))\big) \ar[ld]^{\sim}_{\eta} \ar[rd]_{\sim}^{\eta'} &  \\
 \mathbb{H}^{n}\big(X,\mathcal{C}^{\bullet}(\mathcal{O}_X,\mathcal{O}_X)\big)  \ar[rr]_{\sim}^{\chi} &&   \bext_{\mathcal{O}_X}^n\big(\mathcal{C}_{\bullet}(\mathcal{O}_X),\mathcal{O}_X\big) .
}
\end{equation}
But both decomposition $\eta$ and $\eta'$ follows from the same decomposition of 
\[
E_2^{' p,q}=E_2^{'' p,q}= \check{H}^p\big(\mathcal{U},\mathscr{H}^q(\mathcal{C}^{\bullet}(\mathcal{O}_X,\mathcal{O}_X))\big)
\]
 which in turn is induced by the $\lambda$-decomposition of $\mathcal{C}^{\bullet}(\mathcal{O}_X,\mathcal{O}_X)$,  the commutativity of (\ref{diag-1.3}) follows.
\end{proof}

\begin{corollary}
Given an affine open covering $\mathcal{U}$ of $X$, there is a canonical isomorphism
\begin{equation}\label{hodge-decomp-9}
\mathfrak{L}^n: \bigoplus_{p+q=n} \check{H}^p\big(\mathcal{U},\mathscr{H}^q(\mathcal{C}(A,E^{\vee}\otimes E))\big)\xrightarrow{\sim} H^n\big(\check{C}^\bullet(\mathcal{U},\mathcal{C}^{\bullet}(A,E^\vee\otimes E))\big)
\end{equation}
such that the following diagram
\begin{equation}\label{diag-3}
\xymatrix@C-5em{
  \bigoplus_{p+q=n} \check{H}^p\big(\mathcal{U},\ext^q_{\mathcal{O}_{X\times X}}(\Delta_X,\Delta_X)\big) \ar[ddd]^{\Upsilon^n}_{\wr} \ar[rr]_<<<<<<<<<<<<<<<<<<<<{\sim}^<<<<<<<<<<<<<<<<<<<<{\mathfrak{B}^n}  && 
  \bigoplus_{p+q=n} \check{H}^p\big(\mathcal{U},\mathscr{H}^q(\mathcal{C}^{\bullet}(A,E^{\vee}\otimes E))\big) \ar[ddd]_{\wr}^{\mathfrak{L}^n}\\
  & \bigoplus_{p+q=n} \check{H}^p\big(\mathcal{U},\mathscr{H}^q(\mathcal{C}^{\bullet}(\mathcal{O}_X,\mathcal{O}_X))\big) \ar[d]_{\wr}^{\eta} \ar[ru]_{\sim}^{\xi} \ar[lu]^{\sim}_{\rho}  &\\
  & \mathbb{H}^{n}\big(X,\mathcal{C}^{\bullet}(\mathcal{O}_X,\mathcal{O}_X)\big)  \ar[ld]^{\sim}_{\sigma} \ar[rd]_{\sim}^{\zeta}  & \\
   \Ext^n_{\mathcal{O}_{X\times X}}(\Delta_X,\Delta_X)   &&
  H^n\big(\check{C}^\bullet(\mathcal{U},\mathcal{C}^{\bullet}(A,E^\vee\otimes E))\big) \ar[ll]^<<<<<<<<<<<<<<<<<<<<<<<<<<{\sim}_<<<<<<<<<<<<<<<<<<<<<<<<<<{\mathfrak{Q}^n}  
}
\end{equation}
commutes.
\end{corollary}
\begin{proof} The commutativity of the upper triangle follows from naturality. The commutativity of the left trapezoid and the lower triangle is lemma \ref{commutativity-0}.
The isomorphism $\mathfrak{L}^n$ is induced by demanding the commutativity of  the right trapezoid.
\end{proof}

\subsection{An explicit description of $\Phi^2$}
In this subsection we give an explicit description of $\Phi^2$, and compare it to $u(\cdot,\cdot,\cdot)$ of section 5, and thus complete the proof of theorem \ref{comparison0}.
\begin{lemma}\label{const-cochain-0}
Let $l\geq 0$ be an integer, and $v\in \Hom_{\Bbbk}(A^{\otimes_{\Bbbk} l},A)$ such that $\mathfrak{b}(v)=0$. Let $v_i\in \Hom_{\Bbbk}(A^{\otimes_{\Bbbk} l},E^{\vee}\otimes E)
$ be the restriction of $v$ to $U_i$. Thus $\{v_i\}_{i\in I}\in \check{C}^{0}(\mathcal{U},\mathcal{C}^{l}(A,E^\vee\otimes E))$ induces a class in $H^n\big(\check{C}^\bullet(\mathcal{U},\mathcal{C}^{\bullet}(A,E^\vee\otimes E))\big)$, denoted by $\tilde{v}$. Then
\begin{equation}\label{const-cochain-1}
\Xi_E^n\circ\mathfrak{Q}^n(\tilde{v})=v.
\end{equation}
\end{lemma}
\begin{proof} This follows directly from the second statement of lemma \ref{morita2}, and the identification (\ref{morita3}).
\end{proof}

Now we are ready to give an explicit description of $\Psi^2$. Consider the following commutative diagram, keeping in mind the diagram (\ref{diag-3}) which microscopes the following right square:

\begin{equation}
\xymatrix@C-0.5em{
  \bigoplus_{p+q=n}\check{H}^p(\mathcal{U},\wedge^q T_X) \ar[r]^>>>>>{\mathfrak{E}^n}_>>>>>{\sim} & \bigoplus_{p+q=n} \check{H}^p\big(\mathcal{U},\ext^q_{\mathcal{O}_{X\times X}}(\Delta_X,\Delta_X)\big) \ar[d]^{\Upsilon^n}_{\wr} \ar[r]_>>>>>{\sim}^>>>>>{\mathfrak{B}^n} & \bigoplus_{p+q=n} \check{H}^p\big(\mathcal{U},\mathscr{H}^q(\mathcal{C}^{\bullet}(A,E))\big) \ar[d]_{\wr}^{\mathfrak{L}^n}\\
  & \Ext^n_{\mathcal{O}_{X\times X}}(\Delta_X,\Delta_X)  \ar[d]_{\wr}^{\Xi^n_E} &
  H^n\big(\check{C}^\bullet(\mathcal{U},\mathcal{C}^{\bullet}(A,E^\vee\otimes E))\big) \ar[l]^<<<<<<<<<<<{\sim}_<<<<<<<<<<{\mathfrak{Q}^n} \\
  & HH^n(A)& .
}
\end{equation}

\begin{definition}
Define 
\[
\mathcal{Z}^i(A,\Hhom_{\mathcal{O}_X}(E,E))=\ker\Big(\mathcal{C}^i(A,\Hhom_{\mathcal{O}_X}(E,E))\rightarrow \mathcal{C}^{i+1}(A,\Hhom_{\mathcal{O}_X}(E,E))\Big),
\]
and define $\mathcal{Z}_{(p)}^i(A,\Hhom_{\mathcal{O}_X}(E,E))$ to be the subsheaf of $\mathcal{Z}^i(A,\Hhom_{\mathcal{O}_X}(E,E))$ consists of the local sections cohomological to local sections of $\mathcal{C}^i_{(p)}(\mathcal{O}_X,\mathcal{O}_X)$ via the quasi-isomorphism (\ref{morita5}).
\end{definition}

Given a local section of $\mathcal{Z}^i(A,\Hhom_{\mathcal{O}_X}(E,E))$, one can choose a local basis of $E$ to show that  it lies in  $\mathcal{Z}_{(p)}^i(A,\Hhom_{\mathcal{O}_X}(E,E))$, by checking the criterion in  lemma \ref{hodge-decomp-6}.\\

Our general strategy to find an explicit description of $\Phi^n(\tau)$ for $\tau\in \check{H}^m(\mathcal{U},\wedge^l T_X)$ consists of the following steps: 
\begin{enumerate}
   \item Find an explicit expression for $\mathfrak{B}^n\circ \mathfrak{E}^n (\tau)$.
   \item Find an explicit expression for $\mathfrak{L}^n\circ\mathfrak{B}^n\circ \mathfrak{E}^n (\tau)$.
   \item By a \emph{zigzag} in a double complex, find an element $v$ in  $\check{C}^{0}(\mathcal{U},\mathcal{C}^{m+l}(A,E^\vee\otimes E))$ that differs by a coboundary from $\mathfrak{L}^n\circ\mathfrak{B}^n\circ \mathfrak{E}^n (\tau)$, and observe that $v$ in fact is of the form in the lemma \ref{const-cochain-0}. 
 \end{enumerate}  
The second step will make use of the $\lambda$-decomposition. The following construction illustrates an attempt to carry out this strategy, but it is not  completely fulfilled. The problem arises in the second step:  I don't know how to find $\widetilde{\mathbf{t}}^{m,l}$ that satisfies (\ref{zigzag-0}). In the final proof of theorem \ref{comparison0}, I will show that for $\Phi^2$, the construction \ref{constr-1} indeed provides a construction of $\widetilde{\mathbf{t}}^{1,1}$ and $\widetilde{\mathbf{t}}^{0,2}$.

\begin{construction}\label{constr-2}
 Let $\mathcal{U}=\{U_i\}_{i\in I}$ be an affine open covering of $X$, such that $E|_{U_i}$ is free for any $i\in I$. For any ordered set of indices $I=(i_0,...,i_k)\subset \mathcal{I}$, denote $U_{I}=U_{i_0}\cap...\cap U_{i_k}$. For each $U_{I}$, choose a connection $\nabla_I: E\rightarrow E\otimes_{\mathcal{O}_X} \Omega^1_X$. 
For  $a\in \Hhom_{\mathcal{O}_{U_I}}(E,E)$, define $\nabla_I(a)=\nabla_I\circ a-a\circ\nabla_I\in \Hhom_{\mathcal{O}_{U_I}}(E,E\otimes_{\mathcal{O}_{U_I}}\Omega^1)$.

Let  $\theta\in \Gamma(U_I, \wedge^l T_X)$. For $a_1\otimes...\otimes a_l\in \Hhom_{\mathcal{O}_{U_I}}(E,E)^{\otimes_{\Bbbk}l}$, define
\begin{equation}\label{Cotr1}
\mathrm{cotr}(\theta)(a_1\otimes...\otimes a_l)=\theta\llcorner\big(\nabla_I(a_{1})\circ...\circ \nabla_I(a_{l})\big)\in \Hhom_{\mathcal{O}_{U_I}}(E,E).
\end{equation}
For example, if $\theta=\theta_1\wedge...\wedge\theta_{l}$, where $\theta_j\in \Gamma(U_I, T_X)$ for $1\leq j\leq l$, and  set $(\nabla_I)_{\theta_j}(s)=\theta_j\llcorner \nabla_I(s)$ to be the covariant derivative, then
\begin{equation}\label{Cotr1.1}
\mathrm{cotr}(\theta)(a_1\otimes...\otimes a_l)=\sum_{\sigma\in S_l}\mathrm{sgn}(\sigma)(\nabla_I)_{\theta_1}(a_{1^\sigma})\circ...\circ (\nabla_I)_{\theta_l}(a_{l^\sigma}).
\end{equation}

Let $\tau$ be an element of $H^{m}(X,\wedge^{l}T_X)$ for certain integers $m,l\geq 0$. Let $\{\theta_I\}_{|I|=m+1}$ be a \v{C}ech representative of $\tau$, where $\theta_I\in \Gamma(U_I,\wedge^l T_X)$. 
Thus
\[
\{\mathrm{cotr}(\theta_I)\}_{|I|=m+1}\in \check{C}^{m}\Big(\mathcal{U},\mathcal{C}^l\big(A,\Hhom_{\mathcal{O}_X}(E,E)\big)\Big).
\]

Denote $\mathbf{t}^{m,l}=\{\mathrm{cotr}(\theta_I)\}_{|I|=m+1}$.
Look at the following commutative diagram, where $\mathcal{C}^l=\mathcal{C}^l(A,\Hhom_{\mathcal{O}_X}(E,E))$.
\[
\xymatrix{
 & \check{C}^{m-1}(\mathcal{U},\mathcal{C}^l) \ar[d]^{\delta} \ar[r]^{\mathfrak{b}} & \check{C}^{m-1}(\mathcal{U},\mathcal{C}^{l+1}) \ar[d]^{\delta}  \\
 \check{C}^{m}(\mathcal{U},\mathcal{C}^{l-1}) \ar[d]^{\delta} \ar[r]^{\mathfrak{b}} &\check{C}^{m}(\mathcal{U},\mathcal{C}^l) \ar[d]^{\delta} \ar[r]^{\mathfrak{b}} & \check{C}^{m}(\mathcal{U},\mathcal{C}^{l+1})  \\
 \check{C}^{m+1}(\mathcal{U},\mathcal{C}^{l-1})  \ar[r]^{\mathfrak{b}} &\check{C}^{m+1}(\mathcal{U},\mathcal{C}^l). & 
}
\]
Since $\{\theta_I\}_{|I|=m+1}\in Z^m(\mathcal{U},\wedge^l T_X)$, $\mathfrak{b}(\mathbf{t}^{m,l})=0$.  Moreover, by the definition (\ref{Cotr1}), and lemma \ref{hodge-decomp-6}, and trivializing $E$ by the connections chosen,  one easily sees  $\mathbf{t}^{m,l}\in \check{C}^{m}(\mathcal{U},\mathcal{Z}_{(l)}^l)$. But $\delta \mathbf{t}^{m,l}$ is not necessarily zero. 
Suppose we can  find $\widetilde{\mathbf{t}}^{m,l}$ such that
\begin{equation}\label{zigzag-0}
\begin{cases}
\widetilde{\mathbf{t}}^{m,l}\in \check{C}^{m}(\mathcal{U},\mathcal{Z}_{(l)}^l),\\
\widetilde{\mathbf{t}}^{m,l}-\mathbf{t}^{m,l}\in \mathfrak{b}\big(\check{C}^{m}(\mathcal{U},\mathcal{C}^{l-1})\big),\\
\delta\widetilde{\mathbf{t}}^{m,l}=\mathfrak{b}\widetilde{\mathbf{t}}^{m,l} =0.
\end{cases}
\end{equation}
  Then since 
\[
\check{H}^m(\mathcal{U},\mathcal{C}^l)=\Hom_{\Bbbk}\Big(A^{\otimes l},\check{H}^m\big(\mathcal{U},\Hhom_{\mathcal{O}_X}(E,E)\big)\Big)
\]
and $\check{H}^m\big(\mathcal{U},\Hhom_{\mathcal{O}_X}(E,E)\big)=\Ext^m(E,E)=0$ for $m\geq 1$, there exists $\mathbf{t}^{m-1,l}\in \check{C}^{m-1}(\mathcal{U},\mathcal{C}^l)$ such that $\delta \mathbf{t}^{m-1,l}=\mathbf{t}^{m,l}$. Put $\mathbf{t}^{m-1,l+1}=\mathfrak{b}\mathbf{t}^{m-1,l}$. Then $\mathbf{t}^{m-1,l+1}\in \check{C}^{m-1}(\mathcal{U},\mathcal{C}^{l+1})$ and $\delta \mathbf{t}^{m-1,l+1}=0$, $\mathfrak{b}\mathbf{t}^{m,l}=0$. So we can continue this process, until we obtain $\mathbf{t}^{0,m+l}\in \check{C}^{0}(\mathcal{U},\mathcal{C}^{m+l})$. Moreover, because $\mathfrak{b}\mathbf{t}^{0,m+l}=0$ and $\delta \mathbf{t}^{0,m+l}=0$, $\mathbf{t}^{0,m+l}$ lies in  $\Hom_{\Bbbk}(A^{\otimes m+l},A)$ and produces a Hochschild cocycle, and  we denote the resulting class in $HH^{m+l}(A)$ by $v(\tau)$.
\end{construction}
\pqed

\begin{theorem}\label{zigzag-1}
Given $\widetilde{\mathbf{t}}^{m,l}$ satisfying (\ref{zigzag-0}), 
then
\begin{equation}\label{zigzag-2}
\Phi^n(\tau)=(-1)^m v(\tau).
\end{equation}
\end{theorem}
\begin{proof} By the definition (\ref{cotr1}) of cotrace map,  the definition of affine HKR isomorphism (\ref{hkr0})-(\ref{hkr1}), and the construction of the quasi-isomorphism (\ref{morita5}), $\mathbf{t}^{m,l}$ represents $\mathfrak{B}^n\circ\mathfrak{E}^n(\tau)\in \check{H}^m\big(\mathcal{U},\mathscr{H}^l(\mathcal{C}^{\bullet}(A,E))\big)$. So does $\widetilde{\mathbf{t}}^{m,l}$. Moreover, $\widetilde{\mathbf{t}}^{m,l}$ represents a class in $H^n\big(\check{C}^\bullet(\mathcal{U},\mathcal{C}^{\bullet})\big)$, and by the first condition of (\ref{zigzag-0}), 
\[
[\widetilde{\mathbf{t}}^{m,l}]=\zeta\circ\eta\circ\rho^{-1}\circ\mathfrak{E}^n(\tau)
=\mathfrak{L}^n([\mathbf{t}^{m,l}]).
\]
By the  construction \ref{constr-2}, and the sign convention (\ref{cech-diff-0}), $(-1)^m\mathbf{t}^{0,m+l}$ and $\widetilde{\mathbf{t}}^{m,l}$  represents the same class in $H^n\big(\check{C}^\bullet(\mathcal{U},\mathcal{C}^{\bullet})\big)$. Thus (\ref{zigzag-2}) follows from lemma \ref{const-cochain-0}. 
\end{proof}

\begin{proof}[Proof of theorem \ref{comparison0}:]

Since $\Phi^2(0,\beta,\gamma)=\Phi^2(0,\beta,0)+\Phi^2(0,0,\gamma)$, and 
$u_{0,\beta,\gamma}=u_{0,\beta,0}+u_{0,0,\gamma}$ by lemma \ref{deform-bundle1}, we can prove theorem \ref{comparison0} in the case $\beta=0$ and the case $\gamma=0$ separately.

\emph{The case $\gamma=0$:}
This corresponds to the case $m=l=1$ in the construction \ref{constr-2}. Let $\widetilde{\mathbf{t}}^{1,1}=-\{g_{ij}\}_{i,j\in I}$, where $g_{ij}$ is defined in construction \ref{constr-1}. Then by the construction \ref{constr-1}, $\widetilde{\mathbf{t}}^{1,1}$ satisfies (\ref{zigzag-0}) by lemma \ref{hodge-decomp-6} and proposition \ref{euler-idem-1} (iii); in fact, this is automatic for $l=1$. Thus again by the construction \ref{constr-1}, we can take $\mathbf{t}^{0,1}=\{c_i\}_{i\in I}$. Then $\mathbf{t}^{0,2}=-u_{0,\beta,0}$. So by theorem \ref{zigzag-1}, $\Phi^2(0,\beta,0)=u_{0,\beta,0}$.

\emph{The case $\beta=0$:}
It suffices to show that $\tilde{\mathbf{t}}^{0,2}:=\{u_{0,0,\gamma}(\cdot,\cdot)_i\}_{i\in I}$ satisfies (\ref{zigzag-1}). The third condition of (\ref{zigzag-1}) follows by the construction of $u_{0,0,\gamma}$, see lemma \ref{lem-nd3} and \ref{deform-bundle1}.
 The second condition of ({\ref{zigzag-1}}) is a local property, so we can check this locally on each sufficiently small $U_i$. Thus suppose $\gamma=\partial_1 \wedge \partial_2$, where $\partial_k=\partial_{x_k}$, for $i=1,2$, and $\{x_k\}_{1\leq k\leq \dim X}$ are (étale) local coordinates of $X$. In addition we trivialize $E$ by choose a local basis, on $U_i$,  and obtain a corresponding connection $\nabla$. Take 
 \[
 c_i(a)=\nabla_1\otimes \nabla_2(a)-\nabla_1(a)\otimes \nabla_2,
 \]
 where $\nabla_k=\partial_k\llcorner \nabla$, $k=1,2$. For $a,a'\in A$, write $a$ and $a'$ as $(a_{rs})$ and $(a'_{rs})$ in the chosen local basis of $E$. Then  $c_i(a')=(C'_{rs})$, $c_i(a)=(C'_{rs})$, and $c_i(a'a)=(C''_{rs})$ where
\begin{eqnarray*}
C'_{rs}&=&\partial_1\otimes \partial_2 (a'_{rs})-\partial_1(a'_{rs})\otimes\partial_2,\\
C_{rs}&=&\partial_1\otimes \partial_2 (a_{rs})-\partial_1(a_{rs})\otimes\partial_2,\\
C''_{rs}&=&\partial_1\otimes \partial_2 (\sum_{p}a'_{r p}a_{ps})-\partial_1(\sum_{p}a'_{r p}a_{p s})\otimes\partial_2.
\end{eqnarray*}
Then
\begin{eqnarray}\label{comparison3}
&&-\sum_{s}C'_{rs}a_{st}-\sum_{s}a'_{rs}C_{st}+C''_{rt}\nonumber\\
&=& -\sum_s\big(\partial_1(a_{st})\partial_2(a'_{rs})+a_{st}\partial_2(a'_{rs})\partial_1-\partial_1(a'_{rs})\partial_2(a_{st})-a_{st}\partial_1(a'_{rs})\partial_2\big)\nonumber\\
&&-\sum_s\big(a'_{rs}\partial_2(a_{st})\partial_1-a'_{rs}\partial_1(a_{st})\partial_2\big)\nonumber\\
&&+\sum_s\big(a'_{rs}\partial_2(a_{st})\partial_1+a_{st}\partial_2(a'_{rs})\partial_1
-a'_{rs}\partial_1(a_{st})\partial_2-a_{st}\partial_1(a'_{rs})\partial_2\big)\nonumber\\
&=&\sum_{s}\big(\partial_1(a'_{rs})\partial_2(a_{st})-\partial_1(a_{st})\partial_2(a'_{rs})\big).
\end{eqnarray}
Comparing to (\ref{Cotr1.1}) one sees
\[
u_{0,0,\gamma}(a',a)=\mathrm{cotr}(\gamma)(a',a).
\]
Thus the second condition of (\ref{zigzag-0}) is shown. 

The  first condition of (\ref{zigzag-0}) is also local. By the expression (\ref{comparison3}), $u_{0,0,\gamma}$ is anti-symmetric in  $\partial_1$ and $\partial_2$.  
By lemma \ref{hodge-decomp-6} and proposition \ref{euler-idem-1} (iii), $u_{0,0,\gamma}\in \check{C}^{0}(\mathcal{U},\mathcal{Z}_{(2)}^2)$.
\end{proof}

\section{Open problems}
I propose several  problems partly inspired by theorem \ref{def-FEC0}.
\begin{enumerate}
  \item Bernardara and Bolognesi proposed a notion of \emph{categorical representability dimension}. By \cite[definition 2.4]{BB12}, one says that a smooth projective variety $X$ over $\Bbbk$ is \emph{categorically representable in dimension $n$} if $\mathrm{D}^\mathrm{b}(X)$ has a semiorthogonal decomposition 
  \[
  \mathrm{D}^{\mathrm{b}}(X)=\langle \mathbf{B}_1,...,\mathbf{B}_l\rangle
  \]
  such that each $\mathbf{B}_j$ is an admissible subcategory of $\mathrm{D}^{\mathrm{b}}(Y_j)$ where $Y_j$ is a smooth projective variety over $\Bbbk$ of dimension $\leq n$. By \cite[lemma 1.19]{AB15}, if $\Bbbk$ is separably closed, $X$ is categorically representable in dimension zero if and only if $X$ has a full exceptional collection.  So according to theorem \ref{def-FEC0}, the following question seems natural.
  \begin{question}
  For a family of smooth projective varieties, is  the  categorical representability dimension of the geometric fibers upper semicontinuous over the base scheme?
  \end{question}
  \item    In \cite{KT17} Kontsevich and Tschinkel showed:
\begin{theorem}\label{thm-KT17}
Let $\mathcal{B}$ be a curve over a field of characteristic zero, $\pi: \mathcal{X}\rightarrow \mathcal{B}$ a smooth proper morphism. Then for any closed point $b$ of $B$, the birational type of $\mathcal{X}_b$ depends only on the birational type of $\mathcal{X}_{\eta}$. In particular, if  $\mathcal{X}_{\eta}$ is rational then every fiber of $\pi$ is rational.
\end{theorem}

  From theorem \ref{thm-KT17}, one can deduce the following.
  \begin{proposition}\label{prop-open-2}
  Let $\Bbbk$ be a field of characteristic zero. Assume for every algebraically closed field $K$ containing $\Bbbk$,  every smooth proper variety $X$ over $K$ having a full exceptional collection is rational. Let $S$ be an irreducible scheme separated and of finite type over $\Bbbk$, and $\pi: \mathcal{X}\rightarrow S$ a smooth projective morphism. If there exists a geometric fiber of $\pi$ that  has a full exceptional collection, then every geometric fiber of $\pi$ is rational.
  \end{proposition}
  \begin{proof} By theorem \ref{def-FEC1}, the geometric fiber over the generic point $\eta$ of $S$ has a full exceptional collection, and by the assumption, it is rational. Then there exists a finite  extension $L$ of $\kappa(\eta)$ such that $\mathcal{X}_L$ is rational. By theorem \ref{thm-KT17} and a noetherian induction one sees that every geometric fiber of $\pi$ is rational.
  \end{proof}

It seems a folklore problem  whether smooth projective varieties having full exceptional collections are  rational. In dimension $2$ this is a conjecture attributed to Orlov. Proposition \ref{prop-open-2}  provides a way to find a non-rational smooth projective variety with a full exceptional collection. For example, by \cite{Kawa06}, a smooth projective toric variety has a full exceptional collection, so we ask following question. 
  \begin{question}
  Does  every deformation (with an irreducible base scheme) of a smooth  toric variety  remain rational? 
  \end{question}
  If the answer is negative, then by proposition \ref{prop-open-2} there exist non-rational smooth projective varieties that have  full exceptional collections. Of course we can also ask similar questions for all the varieties having full exceptional collections, but among them, toric varieties seem the most probable ones that have non-rational deformations.

  \item Fix a base field $\Bbbk$. Consider the following conditions on a set of isomorphism classes of smooth proper schemes over $\Bbbk$.
  \begin{enumerate}
    \item[(0)] $\Spec(\Bbbk)\in S$;
    \item[(i)] If $X,Y\in S$ then $X\times Y\in S$;
    \item[(i')] If $X\times Y\in S$ then $X,Y\in S$;
    \item[(ii)] If $X\in S$ and $E$ is a vector bundle on $X$, then the projective bundle $\mathbb{P}(E)\in S$;
    \item[(ii')] If  $E$ is a vector bundle on $X$ and the projective bundle $\mathbb{P}(E)\in S$, then $X\in S$;
    \item[(iii)] If $X,Y\in S$ and $Y$ is a smooth closed subscheme of $X$, then the blow-up $\mathrm{Bl}_Y X\in S$; 
    \item[(iii')] If $Y$ and $\mathrm{Bl}_Y X\in S$ where $Y$ is a smooth closed subscheme of $X$, then $X\in S$. 
  \end{enumerate}
  We know that the elements of the smallest set $S$ satisfying (0), (i), (ii) and (iii) have full exceptional collections (\cite[theorem 2.6 and 4.3]{Orl92}). Denote the smallest set $S$ satisfying (0), (i), (ii), (iii), and (i'), (ii'), (iii') by $\mathrm{PBB}_\Bbbk$ (PBB stands for products, bundles, and blow-ups), and we call an element of $\mathrm{PBB}_\Bbbk$ a \emph{PBB-variety} over $\Bbbk$.  Then does every PBB-variety has a full exceptional collection? Are the previously known examples of varieties having full exceptional collections all PBB-varieties? For example, quadrics (by blowing up lower dimensional quadrics on projective spaces) and  smooth toric varieties (\cite{AMR99}, \cite{Wlo97}) are PBB-varieties. We expect that  small deformations of  smooth toric varieties provide examples of non-PBB-varieites varieties having full exceptional collections.

  \item By theorem \ref{def-FEC0}, if there is a smooth projective surface $S$  over $\mathbb{C}$ having a full exceptional collection, there exists a smooth projective surface $S_0$  over a number field with this property such that $(S_0)_{\mathbb{C}}$ is a deformation of $S$. Is it possible to attack Orlov's conjecture by studying certain Diophantine properties of such surfaces of general type over number fields?

\end{enumerate}

\textsc{School of Mathematics, Sun Yat-sen University, Guangzhou 510275, P.R. China}

\emph{Email address}: huxw06@gmail.com


\begin{thebibliography}{000}
\bibitem[AB17]{AB17}
Auel, Asher; Bernardara, Marcello.
Cycles, derived categories, and rationality. Surveys on recent developments in algebraic geometry, 199–266, 
Proc. Sympos. Pure Math., 95, Amer. Math. Soc., Providence, RI, 2017.
\bibitem[AMR99]{AMR99}
Abramovich, Dan; Matsuki, Kenji; Rashid, Suliman. A note on the factorization theorem of toric birational maps after Morelli and its toroidal extension. Tohoku Math. J. (2) 51 (1999), no. 4, 489–537. Matsuki, Kenji. Correction: ``A note on the factorization theorem of toric birational maps after Morelli and its toroidal extension''. Tohoku Math. J. (2) 52 (2000), no. 4, 629–631.
\bibitem[AT08]{AT08}
Anel, Mathieu; Toën, Bertrand.
Dénombrabilité des classes d'équivalences dérivées de variétés algébriques. 
J. Algebraic Geom. 18 (2009), no. 2, 257–277. 
\bibitem[AB15]{AB15}
Auel, Asher; Bernardara, Marcello. Semiorthogonal decompositions and birational geometry of del Pezzo surfaces over arbitrary fields. Preprint arXiv:1511.07576.
\bibitem[ARS]{ARS} 
Auslander, Maurice; Reiten, Idun; Smalø, Sverre O. Representation theory of Artin algebras. Corrected reprint of the 1995 original. Cambridge Studies in Advanced Mathematics, 36. Cambridge University Press, Cambridge, 1997.
\bibitem[ASS]{ASS} 
Assem, Ibrahim; Simson, Daniel; Skowroński, Andrzej. Elements of the representation theory of associative algebras. Vol. 1. Techniques of representation theory. London Mathematical Society Student Texts, 65. Cambridge University Press, Cambridge, 2006.
\bibitem[Bar85]{Bar85}
Barlow, Rebecca.
Rational equivalence of zero cycles for some more surfaces with $p_g=0$. 
Invent. Math. 79 (1985), no. 2, 303–308. 
\bibitem[Barr68]{Barr68}
Barr, Michael.
Harrison homology, Hochschild homology and triples. 
J. Algebra 8 (1968), 314–323. 

\bibitem[Bay04]{Bay04}
Bayer, Arend.
Semisimple quantum cohomology and blowups. 
Int. Math. Res. Not. 2004, no. 40, 2069–2083. 

\bibitem[BB12]{BB12}
Bernardara, Marcello; Bolognesi, Michele.
Categorical representability and intermediate Jacobians of Fano threefolds. Derived categories in algebraic geometry, 1–25, 
EMS Ser. Congr. Rep., Eur. Math. Soc., Zürich, 2012. 
\bibitem[Bod15]{Bod15}
Bodzenta, Agnieszka.
DG categories and exceptional collections.
Proc. Amer. Math. Soc. 143 (2015), no. 5, 1909–1923.
\bibitem[Bon89]{Bon89}
Bondal, Alexei I. 
Representations of associative algebras and coherent sheaves. (Russian) 
Izv. Akad. Nauk SSSR Ser. Mat. 53 (1989), no. 1, 25--44; translation in 
Math. USSR-Izv. 34 (1990), no. 1, 23–42.
\bibitem[BGKS15]{BGKS15}
Böhning, Christian; Graf von Bothmer, Hans-Christian; Katzarkov, Ludmil; Sosna, Pawel.
Determinantal Barlow surfaces and phantom categories.
J. Eur. Math. Soc. (JEMS) 17 (2015), no. 7, 1569–1592. 
\bibitem[BH13]{BH13}
Buchweitz, Ragnar-Olaf; Hille, Lutz.
Hochschild (co-)homology of schemes with tilting object.
Trans. Amer. Math. Soc. 365 (2013), no. 6, 2823–2844. 
\bibitem[C\u{a}l05]{Cal05}
C\u{a}ld\u{a}raru, Andrei. 
The Mukai pairing. II. The Hochschild-Kostant-Rosenberg isomorphism. Adv. Math. 194 (2005), no. 1, 34–66.
\bibitem[Cib86]{Cib86}
Cibils, Claude.
Hochschild homology of an algebra whose quiver has no oriented cycles. Representation theory, I (Ottawa, Ont., 1984), 55-59, 
Lecture Notes in Math., 1177, Springer, Berlin, 1986.

\bibitem[COS13]{COS13}
Canonaco, Alberto; Orlov, Dmitri; Stellari, Paolo.
Does full imply faithful? (English summary) 
J. Noncommut. Geom. 7 (2013), no. 2, 357–371.

\bibitem[Dub98]{Dub98}
Dubrovin, B.
Geometry and analytic theory of Frobenius manifolds. 
Proceedings of the International Congress of Mathematicians, Vol. II (Berlin, 1998). 
Doc. Math. 1998, Extra Vol. II, 315–326. 

\bibitem[ET]{ET}
El Zein, Fouad; Tu, Loring W.
From sheaf cohomology to the algebraic de Rham theorem. Hodge theory, 70–122, 
Math. Notes, 49, Princeton Univ. Press, Princeton, NJ, 2014. 
\bibitem[EGAII]{EGAII}
Grothendieck, A. Eléments de géométrie algébrique (rédigés avec la collaboration de Jean Dieudonné): II. Étude globale élémentaire de quelques classes de morphismes
Publications mathématiques de l'IHES., tome 8 (1961), p. 5-222.
\bibitem[EGAIII]{EGAIII}
Grothendieck A. Eléments de géométrie algébrique (rédigés avec la collaboration de Jean Dieudonné): III. Etude cohomologique des faisceaux cohérents, Premiere partie. Publications Mathématiques de l'IHES, 1961, 11: 5-167. 
\bibitem[Ger64]{Ger64}
Gerstenhaber, Murray. On the deformation of rings and algebras. Ann. of Math. (2) 79 (1964), 59–103.
\bibitem[GS87]{GS87}
Gerstenhaber, Murray; Schack, S. D.
A Hodge-type decomposition for commutative algebra cohomology. 
J. Pure Appl. Algebra 48 (1987), no. 3, 229–247. 
\bibitem[Gode58]{Gode58}
Godement, Roger. Topologie algébrique et théorie des faisceaux.  Actualit'es Sci. Ind. No. 1252. Publ. Math. Univ. Strasbourg. No. 13 Hermann, Paris 1958.
\bibitem[Gro17]{Gro17}
Gross, Philipp.
Tensor generators on schemes and stacks.
Algebr. Geom. 4 (2017), no. 4, 501–522. 
\bibitem[HS88]{HS88}
Happel, Dieter; Schaps, Mary.
Deformations of tilting modules. Perspectives in ring theory (Antwerp, 1987), 1–20, 
NATO Adv. Sci. Inst. Ser. C Math. Phys. Sci., 233, Kluwer Acad. Publ., Dordrecht, 1988. 
\bibitem[Har66]{Har66}
Hartshorne, Robin. Residues and duality. Lecture notes of a seminar on the work of A. Grothendieck, given at Harvard 1963/64. With an appendix by P. Deligne. Lecture Notes in Mathematics, No. 20 Springer-Verlag, Berlin-New York 1966.
\bibitem[HKR62]{HKR62}
Hochschild, G.; Kostant, Bertram; Rosenberg, Alex.
Differential forms on regular affine algebras. 
Trans. Amer. Math. Soc. 102 1962 383–408. 
\bibitem[HT10]{HT10}
Huybrechts, Daniel; Thomas, Richard P. Deformation-obstruction theory for complexes via Atiyah and Kodaira-Spencer classes. Math. Ann. 346 (2010), no. 3, 545–569. 
\bibitem[Ill71]{Ill71}
Illusie, Luc. Conditions de finitude relatives. Théorie des Intersections et Théorème de Riemann-Roch. Springer, Berlin, Heidelberg, 1971. 222–273.
\bibitem[Ill05]{Ill05}
Illusie, Luc. Grothendieck's existence theorem in formal geometry. With a letter  of Jean-Pierre Serre. Math. Surveys Monogr., 123, Fundamental algebraic geometry, 179–233, Amer. Math. Soc., Providence, RI, 2005.
\bibitem[Kawa06]{Kawa06}
Kawamata, Yujiro. Derived categories of toric varieties. Michigan Math. J. 54 (2006), no. 3, 517–535.
\bibitem[KT17]{KT17} Kontsevich M, Tschinkel Y. Specialization of birational types. Preprint arXiv:1708.05699, 2017.

\bibitem[Kuz11]{Kuz11}
Kuznetsov, Alexander.
Base change for semiorthogonal decompositions.
Compos. Math. 147 (2011), no. 3, 852–876.
\bibitem[Kuz15]{Kuz15}
Kuznetsov, Alexander.
Height of exceptional collections and Hochschild cohomology of quasiphantom categories.
J. Reine Angew. Math. 708 (2015), 213–243.  
\bibitem[Lieb06]{Lieb06}
Lieblich, Max.
Moduli of complexes on a proper morphism. 
J. Algebraic Geom. 15 (2006), no. 1, 175–206. 
\bibitem[Lip09]{Lip09}
Lipman, Joseph.
Notes on derived functors and Grothendieck duality.  Foundations of Grothendieck duality for diagrams of schemes, 1–259, 
Lecture Notes in Math., 1960, Springer, Berlin, 2009. 
\bibitem[Loday]{Loday} 
Loday, Jean-Louis. Cyclic homology. Appendix E by María O. Ronco. Second edition. Chapter 13 by the author in collaboration with Teimuraz Pirashvili. Grundlehren der Mathematischen Wissenschaften, 301. Springer-Verlag, Berlin, 1998.
\bibitem[Low05]{Low05}
Lowen, Wendy.
Obstruction theory for objects in abelian and derived categories.
Comm. Algebra 33 (2005), no. 9, 3195–3223. 

\bibitem[Orl92]{Orl92}
Orlov, D. O.
Projective bundles, monoidal transformations, and derived categories of coherent sheaves. 
Izv. Ross. Akad. Nauk Ser. Mat. 56 (1992), no. 4, 852--862; translation in 
Russian Acad. Sci. Izv. Math. 41 (1993), no. 1, 133–141.

\bibitem[Rae16]{Rae16}
Raedschelders, Theo. Non-split Brauer-Severi varieties do not admit full exceptional collections. Preprint arXiv:1605.09216 (2016).
\bibitem[Spa88]{Spa88}
Spaltenstein, N. Resolutions of unbounded complexes. Compositio Math. 65 (1988), no. 2, 121–154.
\bibitem[Swan96]{Swan96} 
Swan, Richard G.
Hochschild cohomology of quasiprojective schemes.
J. Pure Appl. Algebra 110 (1996), no. 1, 57–80.
\bibitem[Toda05]{Toda05}
 Toda, Yukinobu. Deformations and Fourier-Mukai transforms. J. Differential Geom. 81 (2009), no. 1, 197–224.
\bibitem[Vial17]{Vial17}
Vial, Charles.
Exceptional collections, and the Néron-Severi lattice for surfaces. 
Adv. Math. 305 (2017), 895–934. 
\bibitem[Wło97]{Wlo97} 
 Włodarczyk, Jarosław.
Decomposition of birational toric maps in blow-ups \& blow-downs.
Trans. Amer. Math. Soc. 349 (1997), no. 1, 373–411.
\bibitem[Yeku02]{Yeku02}
Yekutieli, Amnon. The continuous Hochschild cochain complex of a scheme. Canad. J. Math. 54 (2002), no. 6, 1319–1337.
\end{thebibliography}
\end{document}